\documentclass{proc-l}

\usepackage{amssymb}

\usepackage{graphicx}

\usepackage{natbib}

\copyrightinfo{}{}

\newtheorem{theorem}{Theorem}[section]
\newtheorem{lemma}[theorem]{Lemma}

\theoremstyle{definition}

\theoremstyle{remark}
\newtheorem{remark}[theorem]{Remark}

\numberwithin{equation}{section}

\usepackage{microtype}

\usepackage{subcaption}

\usepackage[width=1.0\textwidth]{caption}

\usepackage{booktabs} 

\usepackage{hyperref}

\usepackage{soul,xcolor}
\definecolor{darkred}{rgb}{.7,0,0}
\definecolor{grey}{rgb}{.7,.6,.5}

\newcommand{\agree}[1]{{\color{black}{#1}}}

\usepackage[normalem]{ulem}
\newcommand\mlremove{\bgroup\markoverwith
{\textcolor{red}{\rule[.5ex]{2pt}{1pt}}}\ULon}
\newcommand\asremove{\bgroup\markoverwith
{\textcolor{cyan}{\rule[.5ex]{2pt}{1pt}}}\ULon}
\newcommand\agreeremove{\bgroup\markoverwith
{\textcolor{blue}{\rule[.5ex]{2pt}{1pt}}}\ULon}

\usepackage{amsmath,amsbsy,amsgen,amscd,amsfonts} 
\usepackage{cleveref}
\usepackage{bbm} 
\usepackage{mathrsfs} 
\usepackage{bbold}
\usepackage[utf8]{inputenc}
\usepackage{csquotes}

\newcommand{\scaleeps}{\varepsilon}
\newcommand{\epsm}{\epsilon}
\newcommand{\epsapprox}{\delta}
\newcommand{\cg}{\mathcal{G}}

\newcommand{\ER}{\mathcal{E}_R}
\newcommand{\EG}{\mathcal{E}_G}
\newcommand{\cL}{\mathcal{L}}
\newcommand{\cI}{\mathcal{I}}

\newcommand{\Tm}{T_{\rm{max}}}
\newcommand{\cM}{\mathcal{M}}
\newcommand{\ms}{m^\ast}
\newcommand{\ts}{\theta^\ast_{\infty}}
\newcommand{\tst}{\theta^\ast_{T}}
\newcommand{\Ai}{A_{\infty}}
\newcommand{\AT}{A_{\scriptsize{T}}}
\newcommand{\bi}{b_{\infty}}
\newcommand{\bT}{b_{\scriptsize{T}}}

\newcommand{\erm}{$\square$}

\providecommand{\mathbbm}{\mathbb} 

\newcommand{\R}{\mathbbm{R}}

\newcommand{\N}{\mathbbm{N}}

\newcommand{\cA}{\mathcal{A}}

\newcommand{\sX}{\mathsf{X}}
\newcommand{\sY}{\mathsf{Y}}

\newcommand{\bbR}{\mathbb{R}}

\newcommand{\Expect}{\operatorname{\mathbb{E}}}

\DeclareMathOperator*{\argmin}{argmin} 

\newtheorem{assumption}{Assumption}

\begin{document}

\setcitestyle{numbers,open={[},close={]}} 
\bibliographystyle{plainnat}

\title[Machine Learning of Model Error]{A Framework for Machine Learning of\\ Model Error in Dynamical Systems}


\author{Matthew E. Levine}
\address{Dept. of Computing and Mathematical Sciences\\
California Institute of Technology\\
Pasadena, CA 91125, USA}
\email{mlevine@caltech.edu}
\thanks{The authors are grateful to David Albers, Oliver Dunbar,
Ian Melbourne, Fei Lu, Ivan D. Jimenez Rodriguez, and Yisong Yue for helpful discussions.
The work of MEL and AMS was supported by NIH RO1 LM012734
``Mechanistic Machine Learning''.
MEL is also supported by the National Science Foundation Graduate Research Fellowship under Grant No. DGE‐1745301.
AMS is also supported by NSF (award AGS-1835860), NSF (award DMS-1818977),
the Office of Naval Research (award N00014-17-1-2079),
and the AFOSR under MURI award number FA9550-20-1-0358
(Machine Learning and Physics-Based Modeling and Simulation).
The computations presented here were conducted in the Resnick High Performance Center, a facility supported by Resnick Sustainability Institute at the California Institute of Technology.}

\author{Andrew M. Stuart}
\address{Dept. of Computing and Mathematical Sciences\\
California Institute of Technology\\
Pasadena, CA 91125, USA}
\email{astuart@caltech.edu}


\date{July 13, 2021.}



\begin{abstract}

The development of data-informed predictive models for dynamical systems is of widespread interest in many disciplines.
We present a unifying framework for
blending mechanistic and machine-learning approaches
to identify dynamical systems from noisily and partially observed data.
We compare pure data-driven learning with hybrid
models which incorporate imperfect domain knowledge,
referring to the discrepancy between an assumed truth model
and the imperfect mechanistic model as {\em model error.}
Our formulation is agnostic to the chosen machine learning model,
is presented in both continuous- and discrete-time settings,
and is compatible both with model errors that exhibit substantial memory and errors that are memoryless.

First, we study memoryless linear (w.r.t. parametric-dependence) model error
from a learning theory perspective, defining excess risk and generalization error.
For ergodic continuous-time systems, we prove that
both excess risk and generalization error
are bounded above by terms that diminish with the
square-root of $T$, the time-interval over which training data is specified.

Secondly, we study scenarios that benefit from modeling with memory, proving
universal approximation theorems for two classes
of continuous-time recurrent neural networks (RNNs): both
can learn memory-dependent
model error, assuming that it is governed
by a finite-dimensional hidden variable and that,
together, the observed and hidden variables form
a continuous-time Markovian system.
In addition, we connect one class of RNNs
to reservoir computing, thereby relating learning of memory-dependent error to recent work
on supervised learning between Banach spaces using random features.

Numerical results are presented (Lorenz '63, Lorenz '96 Multiscale systems) to compare purely data-driven and hybrid approaches,
finding hybrid methods less data-hungry and more parametrically efficient.
We also find that, while a continuous-time framing allows for robustness
to irregular sampling and desirable domain-interpretability,
a discrete-time framing can provide similar or better predictive performance,
especially when data are undersampled and the vector field
defining the true dynamics cannot be identified.
Finally, we demonstrate numerically how data assimilation can be leveraged to learn hidden dynamics from noisy, partially-observed data,
and illustrate challenges in representing memory by this approach,
and in the training of such models.

\end{abstract}

\maketitle


\section{Introduction}
\label{sec:intro}

\subsection{Background and Literature Review}
\label{ssec:back}

The modeling and prediction of dynamical systems and time-series is
an important goal in numerous domains, including biomedicine, climatology,
robotics, and the social sciences.
Traditional approaches to modeling these systems appeal to careful study of their mechanisms,
and the design of targeted equations to represent them.
These carefully built {\em mechanistic models}
have impacted humankind in numerous
arenas, including our ability to land spacecraft on celestial bodies,
provide high-fidelity numerical weather prediction, and artificially
regulate physiologic processes, through the use of pacemakers and
artificial pancreases, for example.
This paper focuses on the learning of {\em model error}:
we assume that an imperfect mechanistic model is known, and that
data are used to improve it. We introduce a framework for this problem,
focusing on distinctions between Markovian and non-Markovian
model error, providing a unifying review of relevant literature,
developing some underpinning theory related to both
the Markovian and non-Markovian settings, and presenting
numerical experiments which illustrate our key findings.

To set our work in context, we
first review the use of data-driven methods for time-dependent
problems, organizing the literature review around
four themes comprising \Cref{sssec:1,sssec:2,sssec:3,sssec:4};
these are devoted, respectively,
to pure data-driven methods, hybrid methods that build
on mechanistic models, non-Markovian models that describe memory,
and applications of the various approaches. Having set the
work in context, in \Cref{ssec:cont} we detail the
contributions we make, and describe the
organization of the paper.

\subsubsection{Data-Driven Modeling of Dynamical Systems}
\label{sssec:1}

A recent wave of machine learning successes in data-driven
modeling, especially in imaging sciences,
has shown that we can demand even more from existing models, or that
we can design models of more complex phenomena than heretofore.
Traditional models built from, for example, low order polynomials
and/or linearized model reductions, may appear limited
when compared to the flexible function approximation frameworks
provided by neural networks and kernel methods.
Neural networks, for example, have a long history of success in modeling dynamical systems
\citep{narendra_neural_1992,gonzalez-garcia_identification_1998,krischer_model_1993,rico-martinez_continuous-time_1994,rico-martines_discrete-_1993,rico-martinez_discrete-_1992,lagaris_artificial_1998}
and recent developments in deep learning for operators continue to propel this trend \citep{lu_deeponet_2020,bhattacharya_model_2020,li_markov_2021,li_fourier_2021}.

The success of neural networks arguably relies on
balanced expressivity and generalizability, but
other methods also excel in learning parsimonious and
generalizable representations of dynamics. A particularly
popular methodology is to perform
sparse regression over a dictionary of
vector fields, including the use of thresholding approaches (SINDy)
\citep{brunton_discovering_2016} and $L_1$-regularized polynomial regression \citep{tran_exact_2017,schaeffer_learning_2017,schaeffer_extracting_2018,schaeffer_extracting_2020}.
Non-parametric methods, like Gaussian process models \citep{rasmussen_gaussian_2006},
have also been used widely for modeling nonlinear dynamics
\citep{wang_gaussian_2005,frigola_variational_2014,kocijan_modelling_2016,chkrebtii_bayesian_2016}.
While a good choice of kernel is often essential for the success of these methods, recent progress has been made towards automatic hyperparameter tuning via parametric kernel flows \citep{hamzi_learning_2021}.
Successes with Gaussian process models were also extended
to high dimensional problems by using random feature map
approximations \citep{rahimi_random_2008}
within the context of data-driven learning of parametric
partial differential equations (PDEs)
and solution operators \citep{nelsen_random_2020}.
Advancements to data-driven methods based on Koopman operator theory and
dynamic mode decomposition
also offer exciting new possibilities for predicting nonlinear dynamics
from data
\citep{tu_dynamic_2014, korda_data-driven_2020,alexander_operator-theoretic_2020}.

It is important to consider whether to model in discrete- or continuous-time, as both have potential advantages.
The primary positive for continuous-time modeling lies in its flexibility and
interpretability.
In particular, continuous-time approaches are more readily and naturally applied to irregularly sampled timeseries data, e.g. electronic health record data \citep{rubanova_latent_2019}, than discrete-time methods.
Furthermore, this flexibility with respect to timestep enables simple
transferability of a model learnt from discrete-time data at one timestep,
to new settings with a different timestep and indeed to variable timestep
settings; the learned right-hand-side can be used to generate numerical
solutions at any timestep. On the other hand, applying a discrete-time model to a new timestep either requires exact alignment of subsampled data or some post-processing interpolation step.
Continuous-time models may also provide greater interpretability than discrete-time methods when the right-hand-side of the ordinary differential
equation (ODE) is a more physically interpretable object than the $\Delta t$-solution operator (e.g. for equation discovery, \citep{kaheman_learning_2019}).

Traditional implementations of continuous-time learning require accurate estimation of time-derivatives of the state,
but this may be circumvented using approaches that leverage
autodifferentiation software \citep{ouala_learning_2020,rubanova_latent_2019,kahemanAutomaticDifferentiationSimultaneously2022}
or methods which learn from statistics derived from time-series, such as
moments or correlation functions \citep{schneider2021learning}.
\citet{keller_discovery_2021,du_discovery_2021} provide rigorous analysis demonstrating
how inference of a continuous-time model from discrete-time data must be
conducted with great care;
they prove how stable and consistent linear multistep methods for continuous-time integration
may not possess the same guarantees when used for the inverse problem, i.e. discovery of dynamics. \citet{queiruga_continuous--depth_2020}
provide pathological illustrations of this phenomenon in the context of Runge-Kutta methods.

Discrete-time approaches, on the other hand, are easily deployed when train and test data sample rates are the same.
For applications in which data collection is easily configured (e.g. simulated settings, available automatic sensors, etc.), discrete-time methods are typically much easier to implement and test than continuous-time methods.
Moreover, they allow for \enquote{non-intrusive} model correction, as additions are applied outside of the numerical integrator;
this may be relevant for practical integration with complex simulation software.
In addition, discrete-time approaches can be preferable when there is unavoidably large error in continuous-time inference \citet{chorin_discrete_2015,lu_comparison_2016}.

Both non-parametric and parametric
model classes are used in the learning of dynamical systems,
with the latter connecting to the former via the
representer theorem, when Gaussian process
regression \citep{rasmussen_gaussian_2006} is used
\citep{burov_kernel_2020,gilani_kernel-based_2021,harlim_machine_2021}.

\subsubsection{Hybrid Mechanistic and Data-Driven Modeling}
\label{sssec:2}

Attempts to transform domains that have relied on traditional
mechanistic models, by using
purely data-driven (i.e. \emph{de novo} or \enquote{learn from scratch}) approaches, often fall short.
Now, there is a growing recognition by machine learning researchers that these mechanistic models are very valuable \citep{miller_breimans_2021}, as they represent the distillation of centuries of data collected in countless studies
and interpreted by domain experts.
Recent studies have consistently found advantages of hybrid methods
that blend mechanistic knowledge and data-driven techniques;
\citet{willard_integrating_2021} provide a thorough review of this shift amongst scientists and engineers.
Not only do these hybrid methods improve predictive performance \citep{pathak_hybrid_2018},
but they also reduce data demands \citep{rackauckas_universal_2020}
and improve interpretability and trustworthiness,
which is essential for many applications. This is exemplified by work in
autonomous drone landing \citep{shi_neural_2018} and helicopter flying \citep{rackauckas_hybrid_2021},
as well as predictions for COVID-19 mortality risk \citep{sottile_real-time_2021} and COVID-19 treatment response \citep{qian_integrating_2021}.

The question of how best to use the power of data and machine learning to
leverage and build upon our existing mechanistic knowledge is thus
of widespread current interest.
This question and research direction has been anticipated over
the last thirty years of research activity at the interface
of dynamical systems with machine learning
\citep{rico-martinez_continuous-time_1994,wilson_generalised_1997,lovelett_partial_2020}, and now a critical mass of effort is developing.
A variety of studies have been tackling these questions in weather and climate modeling
\citep{kashinath_physics-informed_2021,farchi_using_2021};
even in the imaging sciences, where
pure machine learning has been spectacularly successful,
emerging work shows that incorporating knowledge of
underlying physical mechanisms improves
performance in a variety of image recognition tasks \citep{ba_blending_2019}.

As noted and studied by \citet{ba_blending_2019,freno_machine-learning_2019} and others, there are a few common high-level approaches for blending machine learning with mechanistic knowledge:
(1) use machine learning to learn additive residual corrections for the mechanistic model
\agree{\citep{saveriano_data-efficient_2017,shi_neural_2018,kaheman_learning_2019,harlim_machine_2021,farchi_using_2021,lu_data-based_2017,lu_data-driven_2020,yinAugmentingPhysicalModels2021}};
(2) use the mechanistic model as an input or feature for a machine learning model
\citep{pathak_hybrid_2018,lei_hybrid_2020,borra_effective_2020};
(3) use mechanistic knowledge in the form of a differential equation
as a final layer in a neural network representation of the
solution, or equivalently define the loss function to include
approximate satisfaction of the differential equation
\citep{raissi_physics-informed_2019,raissi_multistep_2018,chen_solving_2021,smith_eikonet_2020};
and (4) use mechanistic intuition to constrain or inform
the machine learning architecture
\citep{haber_stable_2017,maulik_reduced-order_2021}.
Many other successful studies have developed specific designs that further hybridize these and other perspectives
\citep{hamilton_hybrid_2017,freno_machine-learning_2019,yi_wan_bubbles_2020,jia_physics-guided_2021}.
In addition, parameter estimation for mechanistic models is a well-studied topic in data assimilation, inverse problems, and other mechanistic modeling communities, but recent approaches that leverage machine learning for this task may create new opportunities for
accounting for temporal parameter variations \citep{miller_learning_2020}
and unknown observation functions \citep{linial_generative_2021}.

An important distinction should be made between physics-informed \emph{surrogate} modeling and
what we refer to as \emph{hybrid} modeling.
Surrogate modeling primarily focuses on replacing high-cost, high-fidelity mechanistic model simulations
with similarly accurate models that are cheap to evaluate.
These efforts have shown great promise by training machine learning models on expensive high-fidelity simulation data,
and have been especially successful when the
underlying physical (or other domain-specific) mechanistic knowledge and equations are
incorporated into the model training \citep{raissi_physics-informed_2019} and architecture \citep{maulik_reduced-order_2021}.
We use the term hybrid modeling, on the other hand, to indicate when
the final learned system involves interaction (and possibly feedback) between
mechanism-based and data-driven models \citep{pathak_hybrid_2018}.

In this work, we focus primarily on hybrid methods that learn residuals to an imperfect mechanistic model.
We closely follow the discrete-time hybrid modeling framework developed by \cite{harlim_machine_2021},
while providing new insights from the continuous-time modeling perspective.
The benefits of this form of hybrid modeling, which we and many others have observed, are not yet fully understood in a theoretical sense.
Intuitively, nominal mechanistic models are most useful when they encode key nonlinearities that are not readily inferred
using general model classes and modest amounts of data.
Indeed, classical approximation theorems for fitting polynomials, fourier modes, and other common function bases directly reflect this relationship by bounding the error with respect to a measure of complexity of the target function (e.g. Lipschitz constants, moduli of continuity, Sobolev norms, etc.) \citep{devore_constructive_1993}[Chapter 7].
Recent work by \citet{e_priori_2019} provides a priori error bounds for two-layer neural networks and kernel-based regressions, with
constants that depend explicitly on the norm of the target function
in the model-hypothesis space (a Barron space and a reproducing kernel Hilbert space, resp.).
At the same time, problems for which mechanistic models only capture low-complexity trends (e.g. linear) may still be good candidates for hybrid learning (over purely data-driven), as an accurate linear model reduces the parametric burden for the machine-learning task; this effect is likely accentuated in data-starved regimes.
Furthermore, even in cases where data-driven models perform satisfactorily,
a hybrid approach may improve interpretability, trustworthiness, and controllability without sacrificing performance.

Hybrid models are often cast in Markovian, memory-free settings
where the learned dynamical system (or its learned residuals) are solely dependent on the
observed states. This approach can be highly effective when measurements of all
relevant states are available or when the influence of the unobserved states is adequately described by a function of the observables. This is the perspective employed by \citet{shi_neural_2018},
where they learn corrections to physical equations of motion for an autonomous vehicle
in regions of state space where the physics perform poorly---
these residual errors are \emph{driven} by un-modeled turbulence during landing, but can be \emph{predicted} using the observable states of the vehicle (i.e. position, velocity, and acceleration).
This is also the perspective taken in applications of high-dimensional multiscale dynamical systems,
wherein sub-grid closure models parameterize the effects of expensive fine-scale interactions (e.g. cloud simulations)
as functions of the coarse variables \citep{grabowski_coupling_2001,khairoutdinov_cloud_2001,tan_extended_2018,brenowitz_prognostic_2018,ogorman_using_2018,rasp_deep_2018,schneider2021learning,beucler_enforcing_2021}.
The result is a hybrid dynamical system with a physics-based equation defined on the coarse variables with a Markovian correction term that accounts for the effects of the expensive fine scale dynamics.

\subsubsection{Non-Markovian Data-Driven Modeling}
\label{sssec:3}

Unobserved and unmodeled processes are often responsible for
model errors that cannot be represented in a Markovian fashion within
the observed variables alone. This need has driven substantial advances
in memory-based modeling. One approach to this is the use of delay
embeddings \citep{takens_detecting_1981}.
These methods are inherently tied to discrete
time representations of the data and, although very successful in
many applied contexts, are of less value when the goal of
data-driven learning is to fit
continuous-time models; this is a desirable modeling goal in many settings.

An alternative to understanding memory is via
the Mori-Zwanzig formalism, which is a fundamental building
block in the presentation of memory and hidden variables and
may be employed for both discrete-time and continuous-time
models.
Although initially developed primarily in the context of
statistical mechanics, it provides the basis for understanding
hidden variables in dynamical systems, and thus underpins many
generic computational tools applied in this setting
\citep{chorin_optimal_2000,zhu_estimation_2018,gouasmi_priori_2017}. It
has been successfully applied to problems in fluid turbulence \citep{duraisamy_turbulence_2019,parish_dynamic_2017} and molecular dynamics \citep{li_computing_2017,hijon_morizwanzig_2010}.
\citet{lin_data-driven_2021} demonstrate connections between Mori-Zwanzig and delay embedding theory in the context of non-linear autoregressive models using Koopman operator theory.
Indeed, \citet{gilani_kernel-based_2021} shows a correspondence
between the Mori-Zwanzig representation of the Koopman operator
and Taken's delay-embedding flow map.
Studies by \citet{ma_model_2019,wang_recurrent_2020} demonstrate how the Mori-Zwanzig
formalism motivates the use of recurrent neural networks (RNNs)
\citep{rumelhart_learning_1986,goodfellow_deep_2016}
as a deep learning approach to non-Markovian closure modeling.
\citet{harlim_machine_2021} also use the Mori-Zwanzig formalism to deduce a non-Markovian closure model, and evaluate RNN-based approximations of the closure dynamics.
Closure modeling using RNNs has recently emerged as a new way to
learn memory-based closures
\citep{kani_dr-rnn_2017,chattopadhyay_data-driven_2020-1,harlim_machine_2021}.

Although the original formulation of Mori-Zwanzig as a general purpose
approach to modeling partially observed systems was in continuous-time
\citep{chorin_optimal_2000}, many practical implementations adopt a
discrete-time picture \citep{darve_computing_2009,chorin_discrete_2015,lin_data-driven_2021}.
This causes the learned memory terms to depend on sampling rates, which,
in turn, can inhibit flexibility and interpretability.

Recent advances in continuous-time memory-based modeling, however, may be applicable to these non-Markovian hybrid model settings.
The theory of continuous-time RNNs (i.e. formulated as differential equations, rather than a recurrence relation)
was studied in the 1990s \citep{funahashi_approximation_1993,beer_dynamics_1995}, albeit for equations with a specific
additive structure.
This structure was exploited in a continuous-time reservoir computing (RC) approach by \citet{lu_attractor_2018} for reconstructing chaotic attractors from data.
Comparisons between RNNs and RC (a subclass of RNNs with random parameters fixed in the recurrent state)
in discrete-time have yielded mixed conclusions in terms of their relative efficiencies and ability to retain memory \citep{pyle_domain-driven_2021,gauthierNextGenerationReservoir2021a,vlachas_backpropagation_2020,chattopadhyay_data-driven_2020}.
Recent formulations of continuous-time RNNs have departed slightly from the additive structure,
and have focused on constraints and architectures that ensure stability and accuracy of the
resulting dynamical system
\citep{chang_antisymmetricrnn_2019,erichson_lipschitz_2020,niu_recurrent_2019,rubanova_latent_2019,
sherstinsky_fundamentals_2020,ouala_learning_2020}.
In addition, significant theoretical work has been performed for linear RNNs in continuous-time \citep{li_curse_2020}.
Nevertheless, these various methods have not yet been formulated within a hybrid modeling framework, nor has their approximation power been
carefully evaluated in that context.
A recent step in this direction, however, is the work by \citet{gupta_neural_2021}, which tackles non-Markovian hybrid modeling in continuous-time with neural network-based delay differential equations (DDEs).

\subsubsection{Noisy Observations and Data Assimilation}
\label{ssec:noiseSetting}
For this work we consider settings in which the observations may be both
noisy and partial; the observations may be partial either because the system
is undersampled in time, or because certain variables are not observed at all. We emphasize
that ideas from statistics can be used to smooth and/or interpolate data
to remove noise and deal with undersampling \cite{craven1978smoothing}
and to deal with missing data \cite{meng1997algorithm};
and ideas from data assimilation
\cite{asch_data_2016,law_data_2015,reich_probabilistic_2015}
can be used to remove noise and to learn about
unobserved variables \agree{\cite{chenAutodifferentiableEnsembleKalman2021,
gottwald_supervised_2021,gottwaldCombiningMachineLearning2021}}.
In some of our experiments we will use noise-free data in continuous-time,
to clearly expose issues separate from noise/interpolation; but in other
experiments we will use methodologies from data assimilation to enhance
our learning \cite{chenAutodifferentiableEnsembleKalman2021}.

\subsubsection{Applications of Data-Driven Modeling}
\label{sssec:4}

In order to deploy hybrid methods in real-world scenarios,
we must also be able to cope with noisy, partial observations.
Accommodating the learning of model error in this setting,
as well as state estimation, is an active area of
research in the data assimilation (DA) community
\citep{pulido_stochastic_2018,farchi_using_2021,bocquet_bayesian_2020}.
Learning dynamics from noisy data is generally non-trivial for nonlinear systems---there is a chicken-and-egg problem in which accurate state estimation typically relies on the availability of correct models,
and correct models are most readily identified using accurate
state estimates.
Recent studies have addressed this challenge by attempting to jointly learn the noise and the dynamics.
\citet{gottwald_supervised_2021} approach
this problem from a data assimilation perspective,
and employ an Ensemble Kalman Filter (EnKF) to iteratively update the parameters for their dynamics model, then filter the current state using the updated dynamics. \agree{A recent follow-up to this work applies the DA-approach to partially-observed systems, and learns a model on a space of discrete-time delay-embeddings \citep{gottwaldCombiningMachineLearning2021}.}
Similar studies were performed by \citet{brajard_combining_2021}, and applied specifically in model error scenarios \citep{brajard_combining_2020,farchi_using_2021,wikner_combining_2020}.
\agree{\citet{ayedLearningDynamicalSystems2019a} focus on learning a continuous-time neural network representation of an ODE from partial observations, and learn a separate encoder neural network to map a historical warmup sequence to likely initial conditions in the un-observed space.}
 \citet{kaheman_learning_2019} approach this problem from a variational perspective,
performing a single optimization over all noise sequences and dynamical parameterizations.
\citet{nguyen_em-like_2019} use an Expectation-Maximization (EM) perspective to compare these variational and ensemble-based approaches, and further study is needed to understand the trade-offs between these styles of optimization.
\citet{chenAutodifferentiableEnsembleKalman2021} study an EnKF-based optimization scheme that performs joint, rather than EM-based learning, by
running gradient descent on an architecture that backpropagates through the data assimilator.

We note that data assimilators are themselves dynamical systems, which can be tuned (using optimization and machine learning) to provide more accurate state updates and more efficient state identification.
However, while learning improved DA schemes is sometimes viewed as a strategy for coping with model error \citep{zhu_adaptive_2000},
we see the optimization of DA and the correction of model errors as two separate problems that should be addressed individually.

When connecting models of dynamical systems to real-world data, it is also essential to recognize
that available observables may live in a very different space than the underlying dynamics.
Recent studies have shown ways to navigate this using autoencoders and dynamical systems models to jointly learn a latent embedding and dynamics in that latent space \citep{champion_data-driven_2019}.
Proof of concepts for similar approaches primarily focus on image-based inputs, but have potential for applications in medicine \citep{linial_generative_2021} and reduction of nonlinear PDEs \citep{maulik_reduced-order_2021}.

\subsection{Our Contributions}
\label{ssec:cont}

Despite this large and recent body
of work in data-driven learning methods
and hybrid modeling strategies, many challenges remain for understanding how to
best combine mechanistic and machine-learned models; indeed, the answer is
highly dependent on the application.  Here, we construct a mathematical
framework that unifies many of the common approaches for blending mechanistic
and machine learning models; having done so we provide strong evidence for
the value of hybrid approaches. Our contributions are listed as follows:

\begin{enumerate}

\item We provide an overarching framework for learning model error from (possibly noisy) data in dynamical systems settings,
studying both discrete- and continuous-time models,
together with both memoryless (Markovian) and memory-dependent
representations of the model error.
This formulation is agnostic to choice of mechanistic model and class of machine learning functions.

\item We study the Markovian learning problem in the context of ergodic
continuous-time dynamics, proving bounds on excess
risk and generalization error.

\item We present a simple approximation theoretic approach to
learning memory-dependent (non-Markovian)
model error in continuous-time, proving a form of
universal approximation for two families of memory-dependent
model error defined using recurrent neural networks.

\item We describe numerical experiments which: a) demonstrate the utility of
learning model error in comparison both with pure data-driven learning and
with pure (but slightly imperfect) mechanistic modeling;
b) compare the benefits of learning
discrete- versus continuous-time models;
c) demonstrate the utility of auto-differentiable data assimilation to learn dynamics from partially observed, noisy data;
d) explain issues arising
in memory-dependent model error learning in the (typical) situation
where the dimension of the memory variable is unknown.

\end{enumerate}

In \Cref{sec:setting}, we address contribution 1.
by defining the general settings
of interest for dynamical systems in both continuous- and discrete-time.
We then link these underlying systems to a machine learning framework
in Sections \ref{sec:ML} and \ref{sec:opt}; in the former we formulate
the problem in the setting of statistical learning, and in the latter we
define concrete optimization problems found from finite parameterizations
of the hypothesis class in which the model error is sought.
\Cref{sec:imp} is focused on specific choices of
architectures, and underpinning theory
for machine learning methods with these choices:
we analyze linear methods from the
perspective of learning theory in the context of ergodic
dynamical systems (contribution 2.); and we describe an
approximation theorem for continuous-time hybrid
recurrent neural networks (contribution 3.).
Finally, \Cref{sec:experiments} presents our detailed
numerical experiments; we
apply the methods in \Cref{sec:imp}
to exemplar dynamical systems of the forms outlined in \Cref{sec:setting},
and highlight our {findings (contribution 4.).

\section{Dynamical Systems Setting}
\label{sec:setting}

In the following, we use the phrase \emph{Markovian model error} to describe
model error expressible entirely in terms of the observed variable at the
current time, the memoryless situation; \emph{non-Markovian model error}
refers to the need to express the model error in terms of the past
history of the observed variable.

We present a general framework for modeling a dynamical system with
Markovian model error, first in continuous-time (\Cref{ssec:MarkovContinuous})
and then in discrete-time (\Cref{ssec:MarkovDiscrete}).
We then extend the framework to the setting of non-Markovian model error
(\Cref{ssec:MultiScale}), including a parameter $\scaleeps$ which enables us to
smoothly transition from scale-separated problems (where Markovian
closure is likely to be accurate) to problems where the unobserved
variables are not scale-separated from those observed (where
Markovian closure is likely to fail and memory needs to be accounted for).

It is important to note that the continuous-time formulation necessarily assumes an underlying data-generating process that is continuous in nature. The discrete-time formulation can be viewed as a discretization of an underlying continuous system, but can also represent systems that are truly discrete.

The settings that we present are all intended to represent and classify common situations that
arise in modeling and predicting dynamical systems.
In particular, we stress two key features.
First, we point out that mechanistic models (later referred to as a vector field $f_0$ or flow map $\Psi_0$) are often available and may provide predictions with reasonable fidelity. However, these models are often simplifications of the true system, and thus can be improved with data-driven approaches. Nevertheless, they provide a useful starting point that can reduce the complexity and data-hunger of the learning problems.
In this context, we study trade-offs between discrete- and continuous-time framings.
While we begin with fully-observed contexts in which the dynamics are Markovian with respect to the observed state $x$,
 we later note that we may only have access to partial observations $x$ of a larger system $(x,y)$.
By restricting our interest to prediction of these observables, we show how a latent dynamical process (e.g. a RNN) has the power to reconstruct the correct dynamics for our observables.

\subsection{Continuous-Time}
\label{ssec:MarkovContinuous}

Consider the following dynamical system

\begin{equation}
 \label{eq:fdag}
\begin{aligned}
&\dot{x} = f^\dag(x),\quad
x(0) = x_0,
\end{aligned}
 \end{equation}
and define $\sX_s := C([0,s];\R^{d_x})$.
If $f^\dag \in C^1(\R^{d_x}; \R^{d_x})$ then \eqref{eq:fdag}
has  solution $x(\cdot) \in \sX_T$
for any $T<\Tm=\Tm(x_0) \in \R^+$, the maximal interval of existence.

The primary model error scenario we envisage in this section is one in which
the vector field $f^\dag$ can only be partially known or accessed: we assume that
\begin{equation*}
\label{eq:ffz}
f^\dag=f_0 + m^\dag
\end{equation*}
where $f_0$ is known to us and $m^\dag$ is not known.
For any $f_0 \in C^1(\R^{d_x}; \R^{d_x})$ (regardless of its fidelity), there exists a function $m^\dag(x) \in C^1(\R^{d_x}; \R^{d_x})$ such that \eqref{eq:fdag} can be rewritten as
\begin{equation}
 \label{eq:fmdag}
 \dot{x} = f_0(x) + m^\dag(x).
\end{equation}
However, for this paper, it is useful to think of $m^\dag$
as being small relative to $f_0$;
the function $m^\dag$ accounts for \emph{model error.}
While the approach in \eqref{eq:fmdag} is targeted at learning residuals of $f_0$,
$f^\dag$ can alternatively be reconstructed from $f_0$ through a different function
$m^\dag(x) \in C^1(\R^{2d_x}; \R^{d_x})$ using the form
\begin{equation}
 \label{eq:fmaug}
 \dot{x} = m^\dag(x, f_0(x)).
\end{equation}

Both approaches are defined on spaces that allow perfect reconstruction of $f^\dag$.
However, the first formulation hypothesizes that the missing information is additive;
the second formulation provides no such indication.
Because the first approach ensures substantial usage of $f_0$, it has advantages in settings where $f_0$ is trusted by practitioners and model explainability is important.
The second approach will likely see advantages in settings where there is a simple non-additive form of model error, including coordinate transformations and other (possibly state-dependent) nonlinear warping functions of the nominal physics $f_0$.
Note that the use of $f_0$ in representing the model error in the augmented-input setting of \eqref{eq:fmaug}
includes the case of not leveraging $f_0$ at all.
It is, hence, potentially more useful than simply adopting an
$x-$dependent model error; but it requires
learning a more complex function.

The augmented-input method also has connections to model stacking \citep{wolpert_stacked_1992} or bagging \citep{breimanBaggingPredictors1996};
this perspective can be useful when there are $N$ model hypotheses:

\begin{equation*}
  \label{eq:fmstack}
  \dot{x} = m^\dag\big(x, f^{(1)}_0(x), \ \hdots \, f^{(N)}_0(x); \theta\big).
\end{equation*}
The residual-based design in \eqref{eq:fmdag} relates more to model boosting \citep{schapireStrengthWeakLearnability1990}.

Our goal is to use machine learning to approximate these corrector functions $m^\dag$ using our nominal knowledge $f_0$ and observations of a trajectory $\{x(t)\}_{t=0}^T \in \sX_T$, for some $T<\Tm(x_0)$, from the true system \eqref{eq:fdag}.
In this work, we consider only the case of learning $m^\dag(x)$
in equation \eqref{eq:fmdag}.
For now the reader may
consider $\{x(t)\}_{t=0}^T$ given without noise so that,
in principle, $\{\dot{x}(t)\}_{t=0}^T$ is known and may be leveraged.
In practice this will not be the case, for example if the data are
high-frequency but discrete in time; we address this issue in what
follows.

\subsection{Discrete-Time}
\label{ssec:MarkovDiscrete}
Consider the following dynamical system

\begin{equation}
 \label{eq:fdag_discrete}
 x_{k+1} = \Psi^\dag(x_k)
\end{equation}
and define $\sX_K :=\ell^\infty\bigl(\{0,\dots,K\};\R^{d_x}\bigl)$.
If $\Psi^\dag \in C(\R^{d_x}; \R^{d_x})$, the map yields solution
$\{x_{k}\}_{k \in \mathbb{Z}^+} \in \sX_\infty :=\ell^\infty\bigl(\mathbb{Z}^+;\R^{d_x}\bigl).$\footnote{Here we define $\mathbb{Z}^+=\{0,\dots,\}$, the non-negative
integers, including zero.}
As in the continuous-time setting, we assume we only have access to an approximate mechanistic model $\Psi_0 \in C(\R^{d_x}; \R^{d_x})$,
which can be corrected using an additive residual term $m^\dag \in C(\R^{d_x}; \R^{d_x})$:
\begin{equation}
 \label{eq:fm_discrete}
 x_{k+1} = \Psi_0(x_k) + m^\dag(x_k),
\end{equation}
or by feeding $\Psi_0$ as an input to a corrective warping function $m^\dag \in C(\R^{2d_x}; \R^{d_x})$
\begin{equation*}
 \label{eq:fmaug_discrete}
 x_{k+1} = m^\dag(x_k, \Psi_0(x_k));
\end{equation*}
we focus our experiments on the additive residual framing in \eqref{eq:fm_discrete}.

Note that the discrete-time formulation can be made compatible with continuous-time data
sampled uniformly at rate $\Delta t$ (i.e. $x(k \Delta t) = x_k$ for $k \in \N$).
To see this, let $\Phi^\dag(x_0,t) := x(t)$ be the solution operator for \eqref{eq:fdag}
(and $\Phi_0$ defined analogously for $f_0$).
We then have

\begin{subequations}
  \begin{align}
    \Psi^\dag(v) &:= \Phi^\dag(v, \Delta t) \label{eq:Psidag} \\
    \Psi_0(v) &:= \Phi_0(v, \Delta t), \label{eq:Psi0}
  \end{align}
\end{subequations}
which can be obtained via numerical integration of $f^\dag, f_0$, respectively.

\subsection{Partially Observed Systems (Continuous-Time)}
\label{ssec:MultiScale}
The framework in \Cref{ssec:MarkovContinuous,ssec:MarkovDiscrete}
assumes that the system dynamics are Markovian with respect to observable $x$.
Most of our experiments are performed in the fully-observed Markovian case.
However, this assumption rarely holds in real-world systems.
Consider a block-on-a-spring experiment conducted in an introductory physics laboratory.
In principle, the system is strictly governed by the position and momentum of the block (i.e. $f_0$), along with a few scalar parameters.
However (as most students' error analysis reports will note), the dynamics are also driven by a variety of external factors, like a wobbly table or a poorly greased track.
The magnitude, timescale, and structure of the influence of these different factors are rarely known;
and yet, they are somehow encoded in the discrepancy between the nominal equations of motion and the (noisy) observations of this multiscale system.

Thus we also consider the setting in which the dynamics of $x$ is not Markovian.
If we consider $x$ to be the observable states of a Markovian system in
dimension higher than $d_x$, then we can write the full system as

\begin{subequations}
 \label{eq:fgdag}
 \begin{align}
 \dot{x} &= f^\dag(x,y), \quad x(0)=x_0 \\
 \dot{y} &= \frac{1}{\scaleeps}g^\dag(x,y), \quad y(0)=y_0.
 \end{align}
 \end{subequations}
Here $f^\dag \in C^1(\R^{d_x}\times \R^{d_y}; \R^{d_x})$, $g^\dag \in C^1(\R^{d_x}\times \R^{d_y}; \R^{d_y})$, and $\scaleeps>0$ is a constant measuring
the degree of scale-separation (which is large when $\scaleeps$ is small).
The system yields solution
\footnote{With $\sY_T$ defined analogously to $\sX_T.$}
$x(\cdot) \in \sX_T, y(\cdot) \in \sY_T$
for any $T<\Tm(x(0),y(0)) \in \mathbb{R}^+,$ the maximal interval of existence.
We view $y$ as the complicated, unresolved, or unobserved aspects
of the true underlying system.

For any $f_0 \in C^1(\R^{d_x}; \R^{d_x})$ (regardless of its fidelity), there exists a function $m^\dag(x,y) \in C^1(\R^{d_x}\times \R^{d_y}; \R^{d_x})$ such that \eqref{eq:fgdag} can be rewritten as
\begin{subequations}
 \label{eq:fgmdag}
 \begin{align}
 \dot{x} &= f_0(x) + m^\dag(x,y) \\
 \dot{y} &= \frac{1}{\scaleeps}g^\dag(x,y).
 \end{align}
\end{subequations}
Now observe that, by considering the solution of
equation (\ref{eq:fgmdag}b) as a function of the
history of $x$, the influence of $y(\cdot) \in \sY_t$ on the solution $x(\cdot) \in \sX_t$ can be captured by a parameterized (w.r.t. $t$)
family of operators $m^\dag_t \colon \sX_t \times \R^{d_y} \times \R^+ \mapsto \R^{d_x}$ on the historical trajectory $\{x(s)\}_{s=0}^t$, unobserved initial condition $y(0)$, and scale-separation parameter $\scaleeps$ such that

\begin{equation}
 \label{eq:f0closurememory}
 \begin{aligned}
 \dot{x}(t) &= f_0 \bigl(x(t) \bigr) + m^\dag_t \bigl(\{x(s)\}_{s=0}^t ; \ y(0), \ \scaleeps \bigr).
 \end{aligned}
\end{equation}

Our goal is to use machine learning to find a Markovian model, in which
$x$ is part of the state variable, using our nominal knowledge $f_0$
and observations of a trajectory $\{x(t)\}_{t=0}^T \in \sX_T$, for
some $T<\Tm(x_0,y_0)$, from the true system \eqref{eq:fgdag}; note
that $y(\cdot)$ is not observed and nothing is assumed
known about the vector field $g^\dag$ or the parameter $\scaleeps.$

Note that equations \eqref{eq:fgdag}, \eqref{eq:fgmdag} and
\eqref{eq:f0closurememory} are all equivalent formulations of
the same problem and have identical solutions. The third formulation
points towards two intrinsic difficulties: the unknown \enquote{function}
to be learned is in fact defined by a family
of operators $m^\dag_t$ mapping the Banach space of
path history into $\R^{d_x}$; secondly the operator is parameterized
by $y(0)$ which is unobserved. We will address the first issue
by showing that the operators $m^\dag_t$ can be arbitrarily
well-approximated from within a family of differential equations
in dimension $\R^{2d_x+d_y}$; the second issue may be addressed
by techniques from data assimilation
\citep{asch_data_2016,law_data_2015,reich_probabilistic_2015}
once this approximating family is learned.
We emphasize, however, that we do not investigate the practicality
of this approach to learning non-Markovian systems and much remains
to be done in this area.

It is also important to note that these non-Markovian operators
$m^\dag_t$ can sometimes be adequately approximated by invoking
a Markovian model for $x$ and simply learning function $m^\dag(\cdot)$
as in \Cref{ssec:MarkovContinuous}.
For example, when $\scaleeps \to 0$ and the $y$ dynamics, with $x$ fixed, are sufficiently mixing, the averaging principle \citep{bensoussan_asymptotic_2011,vanden-eijnden_fast_2003,pavliotis_multiscale_2008}
may be invoked to deduce that

$$\lim_{\scaleeps \to 0} m^\dag_t \bigl(\{x(s)\}_{s=0}^t; \ y(0), \scaleeps \bigr) =
m^\dag(x(t))$$
for some $m^\dag$ as in \Cref{ssec:MarkovContinuous}. This fact is used
in section 3 of \citep{jiang_modeling_2020} to study the learning of
closure models for linear Gaussian stochastic differential
equations (SDEs).

It is highly advantageous to identify settings where Markovian modeling is sufficient, as it is a simpler learning problem.
We find that learning $m^\dag_t$ is necessary when there is
significant memory required to explain the dynamics of $x$;
learning $m^\dag$
is sufficient when memory effects are minimal.
In \Cref{sec:experiments}, we show that Markovian closures can perform well for certain tasks even when the scale-separation factor $\scaleeps$ is not small.
In \Cref{sec:ML} we demonstrate how the family of operators
$m^\dag_t$ may be represented through ODEs,
appealing to ideas which blend continuous-time RNNs with an assumed known
vector field $f_0.$

\subsection{Partially Observed Systems (Discrete-Time)}
\label{ssec:MultiScaleD}

The discrete-time analog of the previous setting considers a mapping
\begin{subequations}
 \label{eq:fgdag_discrete}
 \begin{align}
 x_{k+1} &= \Psi_1^\dag(x_k,y_k) \\
 y_{k+1} &= \Psi_2^\dag(x_k,y_k)
 \end{align}
 \end{subequations}
with $\Psi_1^\dag \in C(\R^{d_x}\times \R^{d_y}; \R^{d_x})$, $\Psi_2^\dag \in C(\R^{d_x}\times \R^{d_y}; \R^{d_y})$, yielding solutions $\{x_{k}\}_{k \in \mathbb{Z}^+} \in \sX_\infty$
and $\{y_{k}\}_{k \in \mathbb{Z}^+} \in \sY_\infty.$
We assume unknown $\Psi_1^\dag, \Psi_2^\dag$, but known approximate model $\Psi_0$
to rewrite \eqref{eq:fgdag_discrete} as

\begin{subequations}
 \label{eq:fgmdag_discrete}
 \begin{align}
 x_{k+1} &= \Psi_0(x_k) + m^\dag(x_k,y_k) \\
 y_{k+1} &= \Psi_2^\dag(x_k,y_k).
 \end{align}
\end{subequations}
We can, analogously to \eqref{eq:f0closurememory}, write a solution in the space of observables as
\begin{equation}
 \label{eq:Psi0closurememory}
 \begin{aligned}
   x_{k+1} &= \Psi_0 \bigl(x_k \bigr) + m^\dag_k \bigl(\{x_s\}_{s=0}^k , \ y_0\bigr)
 \end{aligned}
\end{equation}
with $m^\dag_k \colon \sX_k \times \R^{d_y} \to \R^{d_x}$, a function of
the historical trajectory $\{x_s\}_{s=0}^k$ and the unobserved
initial condition $y_0$.
If this discrete-time system is computed from the time $\Delta t$
map for \eqref{eq:fdag} then, for $\scaleeps\ll 1$ and when averaging
scenarios apply as discussed in \Cref{ssec:MultiScale},
the memoryless model in \eqref{eq:fm_discrete} may be used.

\section{Statistical Learning for Ergodic Dynamical Systems}
\label{sec:ML}
Here, we present a learning theory framework within which
to consider methods for discovering model error from data.
We outline the learning theory in a continuous-time Markovian setting
(using possibly discretely sampled data), and point to its
analogs in discrete-time and non-Markovian settings.

In the discrete-time settings, we assume access to discretely sampled training data  $\{x_k = x(k \Delta t)\}_{k=0}^K$, where $\Delta t$ is a uniform sampling rate and we assume that $K\Delta t=T.$
In the continuous-time settings, we assume access to continuous-time training data $\{\dot{x}(t), x(t)\}_{t=0}^T$;
\Cref{ssec:setup} discusses the
important practical question of estimating $\dot{x}(t), x(t)$ from discrete
(but high frequency) data.
In either case, consider the problem of identifying $m \in \mathcal{M}$
(where $\mathcal{M}$ represents the model, or hypothesis, class)
that minimizes a loss function quantifying closeness of $m$ to $m^\dag.$
In the Markovian setting we choose a measure $\mu$ on $\R^{d_x}$ and define the loss

$$\mathcal{L}_\mu(m,m^\dag) := \int_{\R^{d_x}} \|m(x) - m^\dag(x)\|^2_2 d\mu(x).$$

If we assume that, at the true $m^\dag$, $x(\cdot)$ is ergodic
with invariant density $\mu$,
then we can exchange time and space averages to see,
for infinitely long trajectory $\{x(t)\}_{t \ge 0}$,

\begin{align*}
  \mathcal{I}_\infty(m) &:= \lim_{T \to \infty} \frac{1}{T} \int_0^T \|m(x(t)) - m^\dag(x(t))\|^2_2 dt \\
  &= \int_{\R^{d_x}} \|m(x) - m^\dag(x)\|^2_2 d\mu(x) \\
  &= \mathcal{L}_\mu(m,m^\dag).
\end{align*}
Since we may only have access to a trajectory dataset of finite length $T$, it is natural to define
$$\mathcal{I}_T(m) :=  \frac{1}{T} \int_0^T \|m(x(t)) - m^\dag(x(t))\|^2_2 dt$$
and note that, by ergodicity,
$$\lim_{T \to \infty} \mathcal{I}_T(m) = \mathcal{L}_\mu(m,m^\dag).$$
Finally, we can use \eqref{eq:fmdag} to get
\begin{equation}
\label{eq:lossc}
\mathcal{I}_T(m) =  \frac{1}{T} \int_0^T \| \dot{x}(t) - f_0(x(t)) - m(x(t))  \|^2_2 dt.
\end{equation}
This, possibly regularized, is a natural loss function to employ when continuous-time data
is available, and should be viewed as approximating $\mathcal{L}_\mu(m,m^\dag).$
We can use these definitions to frame the problem of learning
model error in the language of statistical learning \citep{vapnik_nature_2013}.

If we let $\cM$ denote the hypothesis class over which we seek to
minimize $\mathcal{I}_T(m)$ then we may define
$$\ms_\infty=\argmin_{m \in \cM}\mathcal{L}_\mu(m,m^\dag)=\argmin_{m \in \cM}\mathcal{I}_\infty(m), \quad \ms_T=\argmin_{m \in \cM} \mathcal{I}_T(m).$$
The {\em risk} associated with seeking to approximate $m^\dag$ from the class $\cM$
is defined by $\mathcal{L}_\mu(\ms_\infty,m^\dag)$, noting that this is $0$
if $m^\dag \in \cM.$ The risk measures the intrinsic error incurred
by seeking to learn $m^\dagger$ from the restricted class $\cM$, which
typically does not include $m^\dagger;$ it is an approximation theoretic
concept which encodes the richness of the hypothesis class $\cM$. The
risk may be decreased by increasing the expressiveness of $\cM$. Thus
risk is independent of the data employed.
\emph{Empirical risk minimization} refers to minimizing
$\mathcal{I}_T$ (or a regularized version) rather than $\mathcal{I}_\infty,$
and this involves the specific instance of data that is available.
To quantify the effect of data volume on learning $m^\dag$ through
empirical risk minimization, it is helpful to introduce the following
two concepts.  The {\em excess risk} is defined by
\begin{equation}
  \label{eq:RT}
  R_T := \mathcal{I}_\infty(\ms_T)-\mathcal{I}_\infty(\ms_\infty)
\end{equation}
and represents the additional approximation error incurred by using
data defined over a finite time horizon $T$ in the estimate of
$m^\dag.$
The {\em generalization error} is
\begin{equation}
  \label{eq:GT}
  G_T := \mathcal{I}_T(\ms_T)-\mathcal{I}_\infty(\ms_T)
\end{equation}
and represents the discrepancy between training error, which is
defined using a finite trajectory, and idealized test error, which
is defined using an infinite length trajectory (or, equivalently,
the invariant measure $\mu$), when evaluated
at the estimate of the function $m^\dag$ obtained from finite data.
We return to study excess risk and generalization error
in the context of linear (in terms of parametric-dependence)
models for $m^{\dag}$, and under ergodicity assumptions on the data
generating process, in \Cref{ssec:learninglinear}.

We have introduced a machine learning framework
in the continuous-time Markovian setting,
but it may be adopted in discrete-time and in non-Markovian settings.
In \Cref{sec:opt}, we define appropriate objective functions for each of these cases.

\begin{remark} The developments we describe here for learning in
ODEs can be extended to the case of
learning SDEs; see
\citep{bento_information_2011,kutoyants_statistical_2004}.
In that setting, consistency in the large $T$ limit is well-understood.
It would be interesting to build on the learning theory perspective
described here to study statistical consistency for
ODEs; the approaches developed
in the work by \citet{mcgoff_consistency_2015,su_large_2021} are potentially
useful in this regard.
\erm
\end{remark}

\section{Parameterization of the Loss Function}
\label{sec:opt}
In this section, we define explicit optimizations for learning (approximate)
model error functions $m^\dag$ for the Markovian settings, and
model error operators $m^\dag_t$ for the non-Markovian settings;
both continuous- and discrete-time formulations are given.
We defer discussion of specific approximation architectures
to the next section. Here we make a notational transition from
optimization over (possibly non-parametric) functions
$m \in \cM$ to functions parameterized by $\theta$ that
characterize the class $\cM$.

In all the numerical experiments
in this paper, we study the use of continuous- and
discrete-time approaches to model data generated by a
continuous-time process. The set-up in this section reflects
this setting, in which two key parameters appear:
$T$, the continuous-time horizon for the data;
and $\Delta t$, the frequency of the data. The latter
parameter will always appear in the discrete-time models;
but it may also be implicit in continuous-time models
which need to infer continuous-time quantities from
discretely sampled data. We relate $T$ and $\Delta t$ by
$K\Delta t=T.$
We present the general forms of $\mathcal{J}_T(\theta)$ (with optional regularization terms $R(\theta)$). Optimization via derivative-based
methodology requires either analytic differentiation of the
dynamical system model with respect to parameters, or the use
of autodifferentiable ODE
solvers \citep{rubanova_latent_2019}.

\subsection{Continuous-Time Markovian Learning}
\label{ssec:ctml}
Here, we approximate the Markovian closure term in \eqref{eq:fmdag} with a parameterized function $m(x; \ \theta)$.
Assuming full knowledge of $\dot{x}(t), x(t)$, we learn the correction term
for the flow field by minimizing the following objective function of $\theta$:

\begin{align}
  \label{eq:Jcmark}
  \mathcal{J}_T(\theta) &=  \frac{1}{T}\int_0^T \big\| \dot{x}(t) - f_0(x(t)) - m(x(t);\ \theta) \big \|^2dt + R(\theta)
\end{align}
Note that $\mathcal{J}_T(\theta)=\mathcal{I}_T\bigl(m(\ \cdot \ ;\ \theta)\bigr)+R(\theta);$ thus the proposed methodology is a regularization of the
empirical risk minimization described in the preceding section.

Notable examples that leverage this framing include: the
paper \citep{kaheman_learning_2019}, where $\theta$ are coefficients
for a library of low-order polynomials and $R(\theta)$ is a
sparsity-promoting regularization defined by the SINDy framework;
\agree{the paper \citep{yinAugmentingPhysicalModels2021}, where $\theta$ are parameters of a deep neural network (DNN) and $L_2$ regularization is applied to the weights;}
the paper \citep{shi_neural_2018}, where $\theta$ are DNN parameters and $R(\theta)$ encodes
constraints on the Lipschitz constant for $m$ provided by spectral
normalization; and the paper
\citep{watson_applying_2019} which applies this approach to the Lorenz '96 Multiscale system using neural networks with an $L_2$ regularization on the weights.

\subsection{Discrete-Time Markovian Learning}
\label{ssec:dtml}

Here, we learn the Markovian correction term in \eqref{eq:fm_discrete} by minimizing:

\begin{align}
  \label{eq:Jdmark}
  \mathcal{J}_T(\theta) &= \frac{1}{K}\sum_{k=0}^{K-1} \big\| x_{k+1} - \Psi_0(x_k) - m(x_k;\ \theta) \big\|^2 + R(\theta)
\end{align}
This is the natural discrete-time analog of \eqref{eq:Jcmark} and
may be derived analogously, starting from a discrete analog of the
loss $\mathcal{L}_{\mu}(m,m^\dag)$ where now $\mu$ is assumed to
be an ergodic measure for \eqref{eq:fdag_discrete}. If a discrete
analog of \eqref{eq:lossc} is defined, then parameterization of $m$,
and regularization, leads to \eqref{eq:Jdmark}.
This is the underlying model assumption in the work by \citet{farchi_using_2021}.

\subsection{Continuous-Time Non-Markovian Learning}
\label{ssec:ctnml}
We can attempt to recreate the dynamics in $x$ for \eqref{eq:f0closurememory}
by modeling the non-Markovian residual term.
A common approach is to augment the dynamics space with a variable
$r \in \R^{d_r}$ leading to a model of the form
\begin{subequations}
  \label{eq:ml_cont_mem}
  \begin{align}
    \dot{x} &= f_0(x) + f_1(r, x;\ \theta) \label{eq:ml_cont_mem1} \\
    \dot{r} &= f_2(r, x;\ \theta) \label{eq:ml_cont_mem2}.
  \end{align}
\end{subequations}
We then seek a $d_r$ large enough, and then parametric models
$\{f_j(r,x; \cdot)\}_{j=1}^2$ expressive enough, to ensure that
the dynamics in $x$ are reproduced by \eqref{eq:ml_cont_mem}.
Note that, although the model error in $x$ is non-Markovian,
as it depends on the history of $x$, we are seeking to explain
observed $x$ data by an enlarged model, including hidden variables $r$,
in which the dynamics of $[x,r]$ is Markovian.

When learning hidden dynamics from partial observations,
we must jointly infer the missing states $r(t)$ and the, typically parameterized,
governing dynamics $f_1,f_2$. Furthermore, when the family of
parametric models is not closed with respect to translation of
$r$ it will also be desirable to learn $r_0$; when $x$ is observed noisily, it is similarly important to learn $x_0$.

To clarify discussions of \eqref{eq:ml_cont_mem} and its training from data,
let $u = [x,r]$ and $f$ be the concatenation of the vector fields given by $f_0,f_1,f_2$ such that
\begin{equation}
\label{eq:udot}
\dot{u} = f(u; \ \theta),
\end{equation}
with solution $u(t; \ v, \theta)$ solving \eqref{eq:udot} (and, equivalently, \eqref{eq:ml_cont_mem}) with initial condition $v$ (i.e. $u(0; \ v, \theta) = v$).
Consider observation operators $H_x, H_r$, such that $x = H_x u$, and $r = H_r u$, and further
define noisy observations of $x$ as
$$z(t) = x(t) + \eta(t),$$
where $\eta$ is i.i.d. observational noise.
We now outline three
optimization approaches to learning from noisily, partially observed data $z$.

\subsubsection{Optimization; Hard Constraint For Missing Dynamics} \label{sssec:H}
Since \eqref{eq:ml_cont_mem} is deterministic, it may suffice to jointly learn parameters and initial condition
$u(0)=u_0$ by minimizing \citep{rubanova_latent_2019}:
\begin{equation}
  \label{eq:Jcmem}
\begin{aligned}
  \mathcal{J}_T(\theta, u_0) &= \frac{1}{T}\int_0^T \big\| z(t) - H_x u(t; \ u_0, \theta) \big\|^2 dt + R(\theta) \\
\end{aligned}
\end{equation}
\agree{A similar approach was applied in \citep{ayedLearningDynamicalSystems2019a},
but where initial conditions were learnt as outputs of an additional DNN encoder network that maps observation sequences (of fixed length and temporal discretization) to initial conditions.}

\subsubsection{Optimization; Weak Constraint For Missing Dynamics} \label{sssec:W}
The hard constraint minimization is very sensitive
for large $T$ in settings where the dynamics is chaotic. This can
be ameliorated, to some extent, by considering the objective function
\begin{equation}
  \label{eq:JcmemWeakConstraint}
\begin{aligned}
  \mathcal{J}_T\big(\theta, u(t)\big) &= \frac{1}{T}\int_0^T \big\| z(t) - H_x u(t) \big\|^2 dt  \\
  & \quad  + \lambda \frac{1}{T}\int_0^T \big\| \dot{u}(t) - f(u(t);\ \theta) \big\|^2 dt.
\end{aligned}
\end{equation}
This objective function is employed in \citep{ouala_learning_2020}, where it is
motivated by the weak-constraint variational formulation (4DVAR)
arising in data assimilation \cite{law_data_2015}.

\subsubsection{Optimization; Data Assimilation For Missing Dynamics}
\label{ssec:mlda}
The weak constraint approach may still scale poorly with $T$ large,
and still relies on gradient-based optimization to infer hidden states.
To avoid these potential issues, we follow the recent work of \citep{chenAutodifferentiableEnsembleKalman2021},
using filtering-based methods to estimate the hidden state. This implicitly
learns initializations and it removes noise from data. It allows
computation of gradients of the resulting loss function
back through the filtering algorithm to learn model parameters.
We define a filtered state
$$\hat{u}_{t,\tau} := \hat{u}_t\Big(\tau; \ \hat{v}, \theta_\textrm{DYN}, \theta_\textrm{DA}, \big\{z(t+s)\big\}_{s=0}^\tau\Big)$$
as an estimate of $u(t+\tau) | \{z(t+s)\}_{s=0}^\tau$ when initialized at $\hat{u}_{t,0}=\hat{v}$.
\footnote{In practice we have found that setting $\hat{v} = 0$ works well.}
In this formulation, we distinguish $\theta_\textrm{DYN}$ as parameters for modeling dynamics via \eqref{eq:ml_cont_mem},
and $\theta_\textrm{DA}$ as hyper-parameters governing the specifics of a data assimilation scheme.
Examples of $\theta_\textrm{DA}$ are the constant gain matrix $K$ that must be chosen for 3DVAR, or parameters of the inflation and localization methods
deployed within Ensemble Kalman Filtering.
By parameterizing these choices as $\theta_\textrm{DA}$, we can optimize them jointly with model parameters $\theta_\textrm{DYN}$. To do this, let $\theta = [\theta_\textrm{DYN}, \theta_\textrm{DA}]$ and minimize
\begin{equation}
  \label{eq:JcmemMatt}
\begin{aligned}
  \mathcal{J}_T(\theta) &=  \frac{1}{(T-\tau_1-\tau_2) \tau_2} \int_{t=0}^{T-\tau_1-\tau_2} \int_{s=0}^{\tau_2}  \big\| z(t + \tau_1 + s) - H_x u(s;\ \hat{u}_{t,\tau_1}, \theta) \big\|^2 ds  \ dt.  \\
\end{aligned}
\end{equation}
Here, $\tau_1$ denotes the length of assimilation time used to estimate the state which initializes a parameter-fitting over window of duration $\tau_2;$ this
parameter-fitting leads to the inner-integration over $s$.
This entire picture is then translated through $t$ time units and the
objective function is found by integrating over $t$.
Optimizing \eqref{eq:JcmemMatt} can be understood as a minimization over short-term forecast errors generated from all assimilation windows. The inner integral takes a fixed start time $t$, applies data assimilation over a window $[t, t+\tau_1]$ to estimate an initial condition $\hat{u}_{t,\tau_1}$, then computes a short-term ($\tau_2$) prediction error resulting from this DA-based initialization.
The outer integral sums these errors over all available windows in long trajectory of data of length $T$.

In our work, we perform filtering using a simple 3DVAR method, whose constant
gain can either be chosen as constant, or can be learnt from data.
When constant, a natural choice is $K \propto H_x^T$, and this approach
has a direct equivalence to
standard warmup strategies employed in RNN and RC training \citep{vlachas_backpropagation_2020,pathak_hybrid_2018}.
The paper \citep{chenAutodifferentiableEnsembleKalman2021} suggests
minimization of a similar objective, but considers more general observation operators $h$, restricts the outer integral to non-overlapping windows, and solves the filtering problem with an EnKF with known state-covariance structure.

\begin{remark}
To motivate learning parameters of the data assimilation we make the following
observation: for problems in which the model is known
(i.e. $\theta_\textrm{DYN}$ is fixed) we observe
successes with the approach of identifying 3DVAR gains that
empirically outperform the theoretically derived
gains in \citep{law_analysis_2013}. Similar is to be expected for
parameters defining inflation and localization in the EnKF.
\erm
\end{remark}

\begin{remark}
  \label{rem:f1f2}
Specific functional forms of $f_1, f_2$ (and their corresponding parameter
inference strategies) reduce \eqref{eq:ml_cont_mem} to various approaches.
For the continuous-time RNN analysis that we discuss in Section \ref{sec:imp}
we will start by considering settings in which $f_1$ and $f_2$ are approximated
from expressive function classes, such as neural networks.
We will then specify to models in which $f_1$ is linear in $r$
and independent of $x$, whilst $f_2$ is a single layer neural network.
It is intuitive that the former may be more expressive and
allow a smaller $d_r$ than the latter; but the latter connects directly
to reservoir computing, a connection we make explicitly in what
follows. Our numerical experiments in Section \ref{sec:experiments}
will be performed in both settings: we will train models from the more
general setting; and by carefully designed experiments we will shed light
on issues arising from over-parameterization, in the sense of choosing
to learn a model in dimension higher than that of the true observed-hidden
model, working in the setting of linear coupling term $f_1,$ depending only on $r$.
\erm
\end{remark}
\begin{remark} The recent paper \citep{gupta_neural_2021} proposes
an interesting, and more computationally tractable, approach to
learning model error in the presence of memory.
They propose to learn a closure operator
$m_\tau( \cdot \ ; \ \theta)\colon \sX_\tau \to \R^{d_x}$ for a
 DDE with finite memory $\tau$:

\begin{equation}
  \label{eq:ml_cont_mem_delay}
    \dot{x}(t) = f_0\big(x(t)\big) + m_\tau \big(\{x(t - s)\}_{s=0}^\tau;\ \theta\big);
\end{equation}
neural networks are used to learn the operator $m_{\tau}.$
Alternatively, Gaussian processes are used to fit a specific
class of stochastic delay differential equation (SDDE)
\eqref{eq:ml_cont_mem_delay} in \citep{schneider2021learning}.
However, although delay-based approaches have seen some
practical success, in many applications they present issues for
domain interpretability and Markovian ODE or PDE closures are more desirable.
\erm
\end{remark}

\subsection{Discrete-Time Non-Markovian Learning}
\label{ssec:dtnml}
In similar spirit to \Cref{ssec:ctnml},
we can aim to recreate discrete-time dynamics in $x$ for \eqref{eq:Psi0closurememory} with model
\begin{subequations}
  \label{eq:ml_disc_mem}
  \begin{align}
    x_{k+1} &= \Psi_0(x_k) + \Psi_1(r_k, x_k;\ \theta) \label{eq:ml_disc_mem1} \\
    r_{k+1} &= \Psi_2(r_k, x_k;\ \theta) \label{eq:ml_disc_mem2}
  \end{align}
\end{subequations}
and objective function

\begin{equation}
  \label{eq:Jdmem}
\begin{aligned}
  \mathcal{J}_T(\theta, r_0) &= \frac{1}{K}\sum_{k=0}^{K-1} \big\| x_{k+1} - \Psi_0(x_k) - \Psi_1(r_k, x_k;\ \theta) \big\|^2 + R(\theta) \\
  & \textrm{s.t.} \quad  \{r_k\}_{k=1}^{K-1} \quad \textrm{solves \eqref{eq:ml_disc_mem2}.}
\end{aligned}
\end{equation}
Observe that estimation of initial condition $r_0$ is again crucial, and the data assimilation methods discussed in \Cref{ssec:ctnml}
can be adapted to this discrete-time setting.
The functional form of $\Psi_1, \Psi_2$ (and their corresponding parameter inference strategies) reduce \eqref{eq:ml_disc_mem} to various approaches, including recurrent neural networks, latent ODEs,
and delay-embedding maps (e.g. to get a delay embedding map, $\Psi_2$ is a shift operator).
\citet{pathak_hybrid_2018} use reservoir computing (a random features
analog to RNN, described in the next section) with $L_2$ regularization to study an approach similar to \eqref{eq:ml_disc_mem}, but included $\Psi_0(x_k)$ as a feature in $\Psi_1$ and $\Psi_2$ instead of using it as the central model upon which to learn residuals.
The data-driven super-parameterization approach in \citep{chattopadhyay_data-driven_2020-1} also appears to follow the underlying assumption of \eqref{eq:ml_disc_mem}.
\Citet{harlim_machine_2021} evaluate hybrid models of form
\eqref{eq:ml_disc_mem} both in settings where delay embedding closures are
employed and where RNN-based approximations via LSTMs are employed.

\section{Underpinning Theory}
\label{sec:imp}
In this section we
identify specific hypothesis classes $\mathcal{M}.$
We do this using random feature maps \citep{rahimi_random_2008}
in the Markovian settings (\Cref{ssec:mmrf}),
and using recurrent neural networks in the memory-dependent setting.
We then discuss these problems
from a theoretical standpoint. In \Cref{ssec:learninglinear}
we study excess risk and generalization error in the context of linear
models (a setting which includes the random features model as
a special case).
And we conclude by discussing the use of RNNs \citep{goodfellow_deep_2016}[Chapter 10]
for the non-Markovian settings (discrete- and continuous-time) in \Cref{ssec:nmrnn}; we
present an approximation theorem for continuous-time hybrid RNN models.
Throughout this section, the specific use of random
feature maps and of recurrent neural networks
is for illustration only;
other models could, of course, be used.

\subsection{Markovian Modeling with Random Feature Maps}
\label{ssec:mmrf}
In principle, any hypothesis class can be used to learn $m^\dag$.
However, we focus on models that are easily trained
on large-scale complex systems
and yet have proven approximation power
for functions between finite-dimensional Euclidean spaces.
For the Markovian modeling case, we use random feature maps;
like traditional neural networks, they possess arbitrary
approximation power \citep{rahimi_weighted_2008,rahimi_uniform_2008},
but further benefit from a quadratic minimization problem in the training phase, as do kernel or Gaussian process methods.
In our case studies, we found random feature models sufficiently expressive, we found optimization easily implementable, and we found the learned
models generalized well.
Moreover, their linearity with respect to unknown parameters enables a straightforward analysis of excess risk and generalization error in \Cref{ssec:learninglinear}.
Details on the derivation and specific design choices for our random feature modeling approach can be found in \Cref{sec:rfAppendix}, where we explain
how we sample $D$ random feature functions $\varphi: \R^{d_x} \to \R$
and stack them to form a vector-valued feature map
$\phi \colon \R^{d_x} \to \R^D$. Given this
random function $\phi$, we define the hypothesis class
\begin{equation}
\label{eq:hc}
\mathcal{M}=\{m \colon \R^{d_x} \to \R^{d_x} \ | \ \exists \ C \in \R^{d_x \times D}:
m(x)=C\phi(x)\}.
\end{equation}

\subsubsection{Continuous-Time}
\label{ssec:ImpMarkCont}
In the continuous-time framing, our Markovian closure model uses
hypothesis class \eqref{eq:hc} and thus takes the form
\begin{equation*}
\begin{aligned}
  \dot{x} &= f_0(x) + C\phi(x(t)).\\
\end{aligned}
\end{equation*}
We rewrite \eqref{eq:Jcmark} for this particular case with an $L_2$ regularization parameter $\lambda \in \R^+$:

\begin{equation}
  \label{eq:J_rfrc}
  \mathcal{J}_T(C) = \frac{1}{2T} \int_0^T \left\|\dot{x}(t) - f_0(x(t)) - C\phi \big(x(t)\big) \right\|^2dt + \frac{\lambda}{2} \|C\|^2.
\end{equation}

We employ the notation $A \otimes B := A B^T$ for the outer-product between
matrices $A \in \R^{m \times n},  B \in \R^{l \times n}$, and
the following notation for time-average
$$\overline{A}_T := \frac{1}{T} \int_0^T A(t) dt$$
of $A\in L^1([0,T],\R^{m \times n}).$
The objective function in \eqref{eq:J_rfrc} is quadratic and convex in $C$ and thus is globally minimized for the unique $C^\ast$ which makes the derivative of $\mathcal{J}_T$ zero. Consequently, the minimizer $C^\ast$ satisfies the following linear equation
(derived in \Cref{sec:ipderivation}):
\begin{equation}
  \label{eq:lip}
      (Z + \lambda I) (C^\ast)^T = Y.
\end{equation}
Here, $I \in \R^{D \times D}$ is the identity and
\begin{equation}
  \label{eq:ZY}
  \begin{aligned}
  Z &= \overline{[\phi \otimes \phi]}_T \in \R^{D \times D},\\
  Y &= \overline{[\phi \otimes m^\dag]}_T \in \R^{D \times d_{x}}.
  \end{aligned}
\end{equation}
Of course $m^\dag$ is not known, but $m^\dag(t)=\dot{x}(t)-f_0(x(t))$
can be computed from data.

To summarize, the algorithm proceeds as follows: 1) create a
realization of random feature vector $\phi$; 2) compute the
integrals in \eqref{eq:ZY} to obtain $Z, Y$; and 3) solve the
linear matrix equation \eqref{eq:lip} for $C^\ast$. Together
this leads to our approximation $m^\dag(x) \approx \ms_T(x; \ \theta) := C^\ast \phi(x)$.

\subsubsection{Discrete-Time}
In discrete-time, our Markovian closure model is
\begin{equation*}
\begin{aligned}
  x_{k+1} &= \Psi_0(x_k) + C\phi(x_k),\\
\end{aligned}
\end{equation*}
and is learnt by minimizing
\begin{equation}
  \label{eq:Jd_rfrc}
  \mathcal{J}_T(\theta) = \frac{1}{K}\sum_{k=0}^{K-1} \big\| x_{k+1} - \Psi_0(x_k) - C\phi \big(x(t)\big) \big\|^2 +  \frac{\lambda}{2} \|C\|^2.
\end{equation}
The objective function in \eqref{eq:Jd_rfrc} is quadratic in $C$ and thus globally minimized
at $C^\ast$.
As in \Cref{ssec:ImpMarkCont}, we can
compute $Z,Y$ and solve a linear system for $C^\ast$ to approximate $m^\dag(x) \approx \ms_T(x; \ \theta) := C^\ast \phi(x)$.
This formulation closely mirrors the fully data-driven linear regression approach in \citep{gottwald_supervised_2021}.


\subsection{Learning Theory for Markovian Models with Linear Hypothesis Class}
\label{ssec:learninglinear}

In this subsection we provide estimates of the
excess risk and generalization error in the context of
learning $m^\dag$ in \eqref{eq:fmdag} from a trajectory
over time horizon $T$. We study
ergodic continuous-time models
in the setting of \Cref{ssec:ctml}.
To this end we consider the
very general linear hypothesis class given by
\begin{equation}
\label{eq:hc2}
\mathcal{M}=\{m \colon \R^{d_x} \to \R^{d_x} \ | \ \exists \ \theta \in \R^{p}:
m(x)=\sum_{\ell=1}^p \theta_\ell f_{\ell}(x)\};
\end{equation}
we note that if the $\{f_\ell\}$ are i.i.d. draws of function $\phi$
in the case $D=d_x$ then this too reduces to a random features model, but
that our analysis in the context of statistical learning does not
rely on the random features structure.
In fact our analysis can be used to provide learning theory for other linear settings, where $\{f_\ell\}$ represents a dictionary of hypothesized features whose coefficients are to be learnt from data.
Nonetheless, universal approximation for random features \citep{rahimi_random_2008}
provides an important example of an approximation class for which the
loss function $\cI_{\infty}$ may be made arbitrarily small by choice of $p$ large enough
and appropriate choice of parameters, and the reader may find it useful to
focus on this case.
We also note that the theory we present in this subsection
is readily generalized to working with
hypothesis class \eqref{eq:hc}.

We make the following ergodicity assumption about the
data generation process:

\begin{assumption}
\label{asp:1}
Equation \eqref{eq:fmdag} possesses a compact attractor $\cA$
supporting invariant measure $\mu.$ Furthermore the dynamical system
on $\cA$ is ergodic with respect to $\mu$ and satisfies
a central limit theorem of the following form: for all H\"older continuous
$\varphi: \mathbb{R}^{d_x} \mapsto \mathbb{R}$, there is $\sigma^2=
\sigma^2(\varphi)$ such that
\begin{equation}
\label{eq:CLT}
\sqrt{T}\Biggl(\frac{1}{T}\int_0^T \varphi\bigl(x(t)\bigr)dt-\int_{\mathbb{R}^{d_x}}
\varphi\bigl(x\bigr)\mu(dx)\Biggr) \Rightarrow N(0,\sigma^2)
\end{equation}
where $\Rightarrow$ denotes convergence in distribution with respect to $x(0) \sim \mu$.
Furthermore a law of the iterated logarithm holds:
almost surely with respect to $x(0) \sim \mu$,
\begin{equation}
\label{eq:LIL}
{\rm limsup}_{T \to \infty}\Bigl(\frac{T}{\log\log T}\Bigr)^{\frac12}\Biggl(\frac{1}{T}\int_0^T \varphi\bigl(x(t)\bigr)dt-\int_{\mathbb{R}^{d_x}}
\varphi\bigl(x\bigr)\mu(dx)\Biggr) =\sigma.
\end{equation}
\end{assumption}

\begin{remark}
Note that in both \eqref{eq:CLT} and \eqref{eq:LIL} $\varphi(\cdot)$ is
only evaluated on (compact) $\cA$ obviating the need for any boundedness
assumptions on $\varphi(\cdot)$. In the work of Melbourne and co-workers,
Assumption \ref{asp:1} is proven to hold for a class of
differential equations, including the Lorenz '63 model at, and
in a neighbourhood of, the classical parameter values:
in \citep{holland_central_2007} the
central limit theorem is established; and in \citep{bahsoun_variance_2020}
the continuity of $\sigma$ in $\varphi$ is proven.
Whilst it is in general very difficult to prove such results
for any given chaotic dynamical system, there is strong
empirical evidence for such results in many chaotic dynamical systems
that arise in practice. This combination of theory and empirical
evidence justify studying the learning of model error under
Assumption \ref{asp:1}. \citet{tran_exact_2017} were the
first to make use of the theory of Melbourne and coworkers
to study learning of chaotic differential equations from
time-series.
\erm
\end{remark}

Given $m$ from hypothesis class $\mathcal{M}$ defined by \eqref{eq:hc2}
we define
\begin{equation}
\label{eq:ts}
\ts={\rm argmin}_{\theta \in \mathbb{R}^p} \mathcal{I}_\infty\bigl(m(\cdot\ ;\theta)\bigr) \ = {\rm argmin}_{\theta \in \mathbb{R}^p} \mathcal{L}_\mu\bigl(m(\cdot\ ;\theta)\bigr)
\end{equation}
and
\begin{equation}
\label{eq:tst}
\tst={\rm argmin}_{\theta \in \mathbb{R}^p} \mathcal{I}_T\bigl(m(\cdot\ ;\theta)\bigr).
\end{equation}
(Regularization is not needed in this setting because the data
is plentiful---a continuous-time trajectory---and the number of
parameters is finite.)
Then $\ts, \tst$ solve the linear systems
\begin{equation*}
\Ai \ts= \bi, \quad \AT \tst = \bT
\end{equation*}
where
\begin{align*}
(\Ai)_{ij}&=\int_{\R^{d_x}} \bigl\langle f_i(x), f_j(x) \bigr\rangle\, \mu(dx),
\quad
&(\bi)_{j}&=\int_{\R^{d_x}}  \bigl\langle m^\dag(x), f_j(x) \bigr\rangle\, \mu(dx),\\
(\AT)_{ij}&=\frac{1}{T}\int_0^T \bigl\langle f_i\bigl(x(t)\bigr), f_j\bigl(x(t)\bigr) \bigr\rangle\, dt,\quad
&(\bT)_{j}&=\frac{1}{T}\int_0^T \bigl\langle m^\dag\bigl(x(t)\bigr), f_j\bigl(x(t)\bigr) \bigr\rangle\, dt.
\end{align*}
These facts can be derived analogously to the derivation
in \Cref{sec:ipderivation}. Given $\ts$ and $\tst$ we also define

$$\ms_\infty=m(\cdot\ ;\ts),\,\, \ms_T=m(\cdot\ ;\tst).$$

Recall that it is assumed that $f^\dag, f_0$, and $m^\dag$ are $C^1.$
We make the following assumption regarding the vector fields defining
hypothesis class $\mathcal{M}.$

\begin{assumption}
\label{asp:2}
The functions $\{f_\ell\}_{\ell=0}^p$ appearing in definition
\eqref{eq:hc2} of the hypothesis class $\mathcal{M}$ are H\"older continuous
on $\R^{d_x}$. In addition, the matrix $\Ai$ is invertible.
\end{assumption}

\begin{theorem}
\label{t:erge}
Let Assumptions \ref{asp:1} and \ref{asp:2} hold. Then the scaled excess
risk $\sqrt{T}R_T$ in \eqref{eq:RT} (resp. scaled generalization error $\sqrt{T}|G_T|$ in \eqref{eq:GT})
is bounded above by
$\|\ER\|$ (resp. $\|\EG\|$), where random
variable $\ER \in \mathbb{R}^{p}$ (resp. $\EG \in \mathbb{R}^{p+1}$)
converges in distribution to $N(0,\Sigma_R)$ (resp. $N(0, \Sigma_G$))
w.r.t. $x(0) \sim \mu$ as $T \to \infty.$
Furthermore, there is constant
$C>0$ such that, almost surely w.r.t. $x(0) \sim \mu$,
$${\rm limsup}_{T \to \infty}\Bigl(\frac{T}{\log\log T}\Bigr)^{\frac12}
\bigl(R_T+|G_T|\bigr) \le C.$$
\end{theorem}

The proof is provided in \Cref{sec:riskyproof}.

\begin{remark}
The convergence in distribution shows that, with high probability with respect to initial data, the excess risk and the generalization error are bounded above by terms of size $1/\sqrt{T}.$ This can be improved
to give an almost sure result, at the cost of the factor of $\sqrt{\log \log T}$.
The theorem shows that
(ignoring log factors and acknowledging the probabilistic
nature of any such statements)
trajectories of length ${\mathcal{O}(\epsilon^{-2})}$ are required
to produce  bounds on the excess risk and generalization error
of size ${\mathcal{O}(\epsilon)}$.

The bounds on excess risk and generalization error
also show that empirical risk minimization (of $\mathcal{I}_T$)
approaches the theoretically
analyzable concept of risk minimization (of $\mathcal{I}_\infty$)
over hypothesis class \eqref{eq:hc2}.
The sum of the excess risk $R_T$ and
the generalization error $G_T$ gives
\begin{equation*}
  E_T:=\mathcal{I}_T(\ms_T)-\mathcal{I}_\infty(\ms_\infty).
\end{equation*}
We note that $\mathcal{I}_T(\ms_T)$ is computable, once the
approximate solution $\ms_T$ has been identified; thus, when
combined with an estimate for $E_T$, this leads to an estimate for
the risk associated with the hypothesis class used.

If the approximating space $\mathcal{M}$ is rich enough, then
approximation theory may be combined with  \Cref{t:erge} to estimate the trajectory error
resulting from the learned dynamical system.
Such an approach is pursued in Proposition 3 of \citep{zhang_estimating_2020} for SDEs.
Furthermore, in that setting, knowledge of rate of mixing/decay of correlations
for SDEs may be used to quantify constants appearing in the error bounds.
It would be interesting to pursue such an analysis for chaotic ODEs with
known mixing rates/decay of correlations.
Such results on mixing are less well-developed, however, for chaotic
ODEs; see discussion of this point in \citep{holland_central_2007},
and the recent work \citep{bahsoun_variance_2020}.

Work by \citet{zhang_error_2021} demonstrates that error bounds on learned model error terms can be extended to bound
error on reproduction of invariant statistics for ergodic SDEs.
Moreover, \citet{e_priori_2019} provide a direction for proving similar bounds on model error learning using nonlinear function classes (e.g. two-layer neural networks).

Finally we remark on the dependence of the risk and generalization
error bounds on the size of the model error. It is intuitive that
the amount of data required to learn model error should decrease
as the size of the model error decreases. This is demonstrated numerically
in \Cref{ssec:l63} (c.f. \Cref{fig:compare_eps_pathak,fig:compare_eps_GP}).
Here we comment that Theorem \ref{t:erge} also exhibits this
feature: examination of the proof in Appendix \ref{sec:riskyproof} shows that
all upper bounds on terms appearing in the excess and generalization
error are proportional to $m^\dag$ itself or to $m_{\infty}^*$, its
approximation given an infinite amount of data; note that $m_{\infty}^*=m^\dag$
if the hypothesis class contains the truth.
\erm
\end{remark}

\subsection{Non-Markovian Modeling with Recurrent Neural Networks}
\label{ssec:nmrnn}

Recurrent Neural Networks (RNNs) are one of the \emph{de facto} tools for modeling systems with memory.
Here, we show straightforward
residual implementations of RNNs for continuous- and discrete-time, with the goal
of modeling non-Markovian model error.

\subsubsection{General Case}
Equation (\ref{eq:ml_cont_mem}b), and its coupling to (\ref{eq:ml_cont_mem}a),
constitute a very general way to account for memory-dependent model
error in the dynamics of $x$.
In fact, for $f_1, f_2$ sufficiently expressive (e.g. random feature functions, neural networks, polynomials), and $d_r \geq d_y$, solutions to \eqref{eq:ml_cont_mem} can approximate solutions to \eqref{eq:fgmdag} arbitrarily well.
We make this type of universal approximation theorem concrete
in Theorems \ref{t:rnnThmGeneral} and \ref{t:rnnThm}.
We start by proving Theorem \ref{t:rnnThmGeneral},
which rests on the following assumptions:

\begin{assumption}
\label{A1gen}
Functions $f^\dag, g^\dag, f_0, f_1, f_2$ are all globally Lipschitz.
\end{assumption}

Note that this implies that $m^\dag$ is also globally Lipschitz.

\begin{assumption}
\label{A3gen}
Fix $T>0$. There exist $\rho_0 \in \R, \rho_T \in \R$
such that, for equation \eqref{eq:fgmdag},
$\bigl(x(0), y(0)\bigr) \in B(0,\rho_0)$
implies that
 $\bigl(x(t), y(t)\bigr) \in B(0,\rho_T)$
$\forall \ t \in [0,T]$.
\end{assumption}

\begin{assumption}
\label{A4}
The hidden state in \eqref{eq:ml_cont_mem}, $r \in \R^{d_r}$, has
the same dimension as the true hidden state $y$ in \eqref{eq:fgmdag};
that is $d_r = d_y$.
\end{assumption}

\begin{assumption}
\label{Aapprox}
Let functions $f_1(\cdot \ ; \ \theta) \in C^1(\R^{d_x}\times\R^{d_y}; \ \R^{d_x})$
and $f_2(\cdot \ ; \ \theta) \in C^1(\R^{d_x}\times\R^{d_y}; \ \R^{d_y})$
be parameterized \footnote{Here we define $\mathbb{N}=\{1,2,\dots,\}$,
the strictly positive integers.}
by $n \in \mathbb{N}$
and $\theta \in \R^n$.
Then, for any $\delta > 0$, there exists $n>0$ and $\theta \in \R^n$ such that
$$\sup_{x,y \in B(0,\rho_T)}\|f^\dag(x,y) - f_1(x,y; \ \theta) \| \leq \delta$$
and
$$\sup_{x,y \in B(0,\rho_T)}\|g^\dag(x,y) - f_2(x,y; \ \theta) \| \leq \delta$$
\end{assumption}

Note that \Cref{Aapprox} can be satisfied by any parametric function class
possessing a universal approximation property for maps
between finite-dimensional Euclidean spaces, such as
neural networks, polynomials and random feature methods. The next theorem
transfers this universal approximation property for maps between
Euclidean spaces to a universal approximation property for representation
of model error with memory; this is a form of infinite dimensional
approximation since, via its own dynamics,
the memory variable $r$ maps the past history
of $x$ into the model error correction term in the dynamics for $x$.

\begin{theorem}
  \label{t:rnnThmGeneral}
Let Assumptions \ref{A1gen}-\ref{Aapprox} hold. Fix any $T>0$ and $\rho_0>0,$
let $x(\cdot),y(\cdot)$ denote the solution of \eqref{eq:fgmdag} with
$\scaleeps=1$ and let
$x_\delta(\cdot),r_\delta(\cdot)$ denote the solution of \eqref{eq:ml_cont_mem}
with parameters $\theta \in \bbR^n$.
Then, for any $\epsapprox>0$ and any $T>0$,
there is a parameter dimension $n=n_\delta \in \N$
and parameterization $\theta=\theta_\delta \in \R^{n_\delta}$
with the property that, for any initial condition
$\bigl(x(0), y(0)\bigr) \in B(0,\rho_0)$
for \eqref{eq:fgmdag},
there is initial condition
$(x_\delta(0),r_\delta(0)) \in \R^{d_x+d_y}$
for \eqref{eq:ml_cont_mem}, such that
$$\sup_{t \in [0,T]} \|x-x_\delta\| \le \epsapprox.$$
\end{theorem}

The proof is provided in \Cref{sec:rnnproofGen};
it is a direct consequence of the approximation power of $f_1, f_2$ and the Gronwall Lemma.

\begin{remark}
  \label{rem:drdy}
Note that this existence theorem also holds for $d_r > d_y$ by freezing
the dynamics in the excess dimensions and initializing it at, for
example, $0$.  However it is possible for augmentations with $d_r > d_y$
to introduce numerical instability when imperfectly initialized
in the excess dimensions, despite their provable correctness
when perfectly initialized  (see \Cref{ssec:irnn}).
Nevertheless, we did not encounter such issues
when training the general model class on the
examples considered in this paper
-- see \Cref{ssec:partialNoisyObs}).
\erm
\end{remark}

\subsubsection{Linear Coupling}
We now study a particular form RNN in which the coupling term
$f_1$ appearing in \eqref{eq:ml_cont_mem} is linear and depends
only on the hidden variable:
 \begin{subequations}
   \label{eq:hybridRNN_cont}
   \begin{align}
     \dot{x} &= f_0(x) + Cr \\
     \dot{r} &= \sigma(Ar + Bx + c).
   \end{align}
 \end{subequations}
Here $\sigma$ is an activation function.
The specific linear coupling form is of particular interest because of the
connection we make (see Remark \ref{rem:rc} below) to reservoir computing.
The goal is to choose $A,B,C,c$ so that output $\{x(t)\}_{t \geq 0}$
matches output of \eqref{eq:fgmdag}, without observation
of $\{y(t)\}_{t \geq 0}$ or knowledge of $m^\dag$ and $g^\dag.$
As in the general case from the preceding subsection,
inherent in choosing these matrices $A,B,C$ and vector $c$
is a choice of embedding dimension for variable $r$ which will
typically be larger than dimension of $y$ itself.  The idea is
to create a recurrent state $r$ of sufficiently
large dimension $d_r$ whose evolution equation takes $x$ as input and,
after a final linear
transformation, approximates the missing dynamics $m^\dag(x,y).$

There is existing approximation theory for discrete-time RNNs
\citep{schafer_recurrent_2007} showing that a discrete-time analog
of our linear coupling set-up can be used to approximate discrete-time
systems arbitrarily well; see also Theorem 3 of \citep{harlim_machine_2021}.
There is also a general approximation theorem using continuous-time RNNs
proved in \citep{funahashi_approximation_1993}, but it does not
apply to the linear-coupling setting. We thus extend the work in these
three papers to the context of residual-based
learning as in \eqref{eq:hybridRNN_cont}.
We state the theorem after making three assumptions upon which it rests:

\begin{assumption}
\label{A1}
Functions $f^\dag, g^\dag, f_0$ are all globally Lipschitz.
\end{assumption}

Note that this implies that $m^\dag$ is also globally Lipschitz.

\begin{assumption}
\label{Asigma}
Let $\sigma_0 \in C^1(\R; \R)$ be bounded and monotonic, with bounded
first derivative. Then $\sigma(u)$
defined by $\sigma(u)_i=\sigma_0(u_i)$
satisfies $\sigma \in C^1(\R^p; \R^p).$
\end{assumption}

\begin{assumption}
\label{A3}
Fix $T>0$. There exist $\rho_0 \in \R, \rho_T \in \R$
such that, for equation \eqref{eq:fgmdag},
$\bigl(x(0), y(0)\bigr) \in B(0,\rho_0)$
implies that
 $\bigl(x(t), y(t)\bigr) \in B(0,\rho_T)$
$\forall \ t \in [0,T]$.
\end{assumption}

\begin{theorem}
  \label{t:rnnThm}
Let Assumptions \ref{A1}-\ref{A3} hold. Fix any $T>0$ and $\rho_0>0,$
let $x(\cdot),y(\cdot)$ denote the solution of \eqref{eq:fgmdag} with
$\scaleeps=1$ and let
$x_\delta(\cdot),r_\delta(\cdot)$ denote the solution
of \eqref{eq:hybridRNN_cont} with parameters $\theta \in \bbR^n$.
Then, for any $\epsapprox>0$ and any $T>0$,
there is embedding dimension $d_r \in \N$, parameter dimension
$n=n_{\delta} \in \N$
and parameterization $\theta=\theta_\delta
= \{A_\delta,\ B_\delta,\ C_\delta, c_\delta\}$
with the property that, for any initial condition
$\bigl(x(0), y(0)\bigr) \in B(0,\rho_0)$
for \eqref{eq:fgmdag},
there is initial condition
$(x_\delta(0),r_\delta(0)) \in \R^{d_x+d_r}$
for \eqref{eq:hybridRNN_cont}, such that
$$\sup_{t \in [0,T]} \|x-x_\delta\| \le \epsapprox.$$
\end{theorem}

The complete proof is provided in \Cref{sec:rnnproof};
here we describe its basic structure.
Define $m(t):=m^\dag\bigl(x(t),y(t)\bigr)$
and, with the aim of finding a differential equation for $m(t)$,
recall \eqref{eq:fgmdag} with $\scaleeps=1$ and define the vector
field
\begin{equation}
\label{eq:hdag}
h^\dag(x,y) := \nabla_x m^\dag(x,y) [f_0(x) + m^\dag(x,y)] + \nabla_y m^\dag(x,y) g^\dag(x,y).
\end{equation}
Since $\dot{m}(t)$ is the time derivative of $m^\dag\big(x(t), y(t)\big)$,
when $(x,y)$ solve \eqref{eq:fgmdag} we have
$$\dot{m} = h^\dag(x,y).$$

Motivated by these observations, we now introduce a new system of
autonomous ODEs for the variables $(x,y,m)
\in \R^{d_x} \times \R^{d_y} \times \R^{d_x}$:
\begin{subequations}
 \label{eq:fgmdag2}
 \begin{align}
 \dot{x} &= f_0(x) + m \\
 \dot{y} &= g^\dag(x,y) \\
 \dot{m} & = h^\dag(x,y).
 \end{align}
\end{subequations}
To avoid a proliferation of symbols we use the same letters
for $(x,y)$ solving equation \eqref{eq:fgmdag2} as for
$(x,y)$ solving equation \eqref{eq:fgmdag}.
We now show $m=m^\dag(x,y)$ is an invariant manifold
for \eqref{eq:fgmdag2}; clearly, on this manifold, the dynamics of
$(x,y)$ governed by \eqref{eq:fgmdag2}
reduces to the dynamics of $(x,y)$ governed by \eqref{eq:fgmdag}.
Thus $m(t)$ must be initialized at
$m^\dag\bigl(x(0),y(0)\bigr)$ to ensure
equivalence between the solution of \eqref{eq:fgmdag2} and
\eqref{eq:fgmdag}.

The desired invariance of manifold $m=m^\dag(x,y)$
under the dynamics \eqref{eq:fgmdag2} follows from the identity
\begin{equation}
\label{eq:identity}
\frac{d}{dt}\Bigl(m-m^\dag(x,y)\Bigr)=-\nabla_x m^\dag(x,y)\bigl(
m-m^\dag(x,y)\bigr).
\end{equation}
The identity is derived by noting that, recalling \eqref{eq:hdag}
for the definition of $h^\dag$, and
then using \eqref{eq:fgmdag2},
\begin{align*}
\frac{d}{dt}m&=h^\dag(x,y)\\
&=\nabla_x m^\dag(x,y) [f_0(x) + m^\dag(x,y)] + \nabla_y m^\dag(x,y) g^\dag(x,y)\\
&=\nabla_x m^\dag(x,y) [f_0(x) + m)] + \nabla_y m^\dag(x,y) g^\dag(x,y)\\
&\quad\quad\quad -\nabla_x m^\dag(x,y)\bigl(
m-m^\dag(x,y)\bigr)\\
&=\frac{d}{dt}m^\dag(x,y)-\nabla_x m^\dag(x,y)\bigl(
m-m^\dag(x,y)\bigr).
\end{align*}
We emphasize this calculation is performed under the dynamics
defined by \eqref{eq:fgmdag2}.

The proof of the RNN approximation
property proceeds by approximating vector fields
$g^\dag(x,y), h^\dag(x,y)$ by neural networks and
introducing linear transformations of $y$ and $m$
to rewrite the approximate version of system \eqref{eq:fgmdag2}
in the form \eqref{eq:hybridRNN_cont}.
The effect of the approximation of the vector fields on
the true solution is then propagated through the system
and its effect controlled via a straightforward Gronwall argument.

\begin{remark}
\label{rem:harlim}
The details of training continuous-time RNNs to ensure accuracy
and long-time stability are a subject of current research
\citep{chang_antisymmetricrnn_2019,erichson_lipschitz_2020,ouala_learning_2020,chenAutodifferentiableEnsembleKalman2021}
and in this paper we confine the training of RNNs to an example
in the general setting, and not the case of linear coupling.
Discrete-time RNN training, on the other hand, is much more mature, and has produced satisfactory accuracy and stability for settings with uniform sample rates that are consistent across train and testing scenarios \citep{harlim_machine_2021}.
The form with linear coupling is widely studied in
discrete time models.  Furthermore, sophisticated variants on RNNs,
such as Long-Short Term Memory (LSTM)
RNNs \citep{hochreiter_long_1997} and Gated Recurrent Units (GRU) \citep{choPropertiesNeuralMachine2014}, are often more effective, although similar
in nature RNNs. However, the potential formulation, implementation and
advantages of these variants in the continuous-time setting
\citep{niu_recurrent_2019} is not yet understood.
We refer readers to \citep{goodfellow_deep_2016} for background on discrete
RNN implementations and
backpropagation through time (BPTT).
For implementations of continuous-time RNNs, it is common to leverage the success of the automatic BPTT code written in PyTorch and Tensorflow by discretizing \eqref{eq:hybridRNN_cont} with an ODE solver that
is compatible with these autodifferentiation tools (e.g. \texttt{torchdiffeq} by \citep{rubanova_latent_2019},
\texttt{NbedDyn} by \citep{ouala_learning_2020}, and \texttt{AD-ENKF} by \citep{chenAutodifferentiableEnsembleKalman2021}).
This compatibility can also be achieved by use of explicit Runge-Kutta schemes \citep{queiruga_continuous--depth_2020}.
Note that the discretization of \eqref{eq:hybridRNN_cont} can (and perhaps should) be much finer than the data sampling rate $\Delta t$,
but that this requires reliable estimation
of $x(t), \dot{x}(t)$ from discrete data.
\erm
\end{remark}

\begin{remark}
\label{rem:da}
The need for data assimilation
\citep{asch_data_2016,law_data_2015,reich_probabilistic_2015}
to learn the initialization
of recurrent neural networks
may be understood as follows.
Since $m^\dag$ is not known and $y$ is not observed
(and in particular $y(0)$ is not known)
the desired initialization for \eqref{eq:fgmdag2}, and thus also
for approximations of this equation in which $g^\dag$ and $h^\dag$ are replaced
by neural networks, is not known. Hence, if an RNN is trained
on a particular trajectory, the initial condition that is required
for accurate approximation of \eqref{eq:fgmdag} from an unseen initial
condition is not known. Furthermore the invariant manifold $m=m^\dag(x,y)$
may be unstable under numerical approximation.
However if some observations of the
trajectory starting at the new initial condition are used,
then data assimilation techniques can
potentially learn the initialization for the RNN
and also stabilize the invariant manifold.
Ad hoc initialization methods are common practice \citep{haykin_modeling_2007,cho_learning_2014,bahdanau_neural_2016,pathak_model-free_2018},
and rely on forcing the learned RNN with a short sequence of observed data to synchronize the hidden state.
The success of these approaches likely rely on RNNs' abilities to emulate data assimilators \citep{harter_data_2012};
however, a more careful treatment of the initialization problem may enable substantial advances.
\erm
\end{remark}

\begin{remark}
\label{rem:rc}
Reservoir computing (RC) is a variant on RNNs which has the advantage of
leading to a quadratic optimization problem \citep{jaeger__2001,lukosevicius_reservoir_2009,grigoryeva_echo_2018}.
Within the context of the continuous-time RNN \eqref{eq:hybridRNN_cont} they correspond
to randomizing $(A,B,c)$ in (\ref{eq:hybridRNN_cont}b) and then
choosing only parameter $C$ to fit the data.
To be concrete, this leads to
$$r(t)=\cg_t \bigl(\{x(s)\}_{s=0}^t; \ r(0),A,B,c\bigr);$$
here $\cg_t$ may be viewed as a random function of the
path-history of $x$ upto time $t$ and of the initial condition for $r.$
Then $C$ is determined by minimizing the quadratic function
\begin{equation*}
  \mathcal{J}_T(C) = \frac{1}{2T} \int_0^T \left\|\dot{x}(t) - f_0(x(t)) - Cr(t) \right\|^2dt + \frac{\lambda}{2} \|C\|^2.
\end{equation*}
This may be viewed as a random feature approach on the Banach space
$\mathsf{X}_T$; the use of random features for learning of mappings
between Banach spaces is studied by \citet{nelsen_random_2020},
and connections between random features and reservoir computing were introduced by \citet{dong_reservoir_2020}.
In the specific setting described here, care will be needed in choosing
probability measure on $(A,B,c)$ to ensure a well-behaved map
$\cg_t$; furthermore data assimilation ideas
\citep{asch_data_2016,law_data_2015,reich_probabilistic_2015} will be needed to
learn an appropriate $r(0)$ in the prediction phase, as discussed in
Remark \ref{rem:da} for RNNs.
\erm
\end{remark}

\section{Numerical Experiments}
\label{sec:experiments}
In this section, we present numerical experiments intended to test different hypotheses
about the utility of hybrid mechanistic and data-driven modeling.
We summarize our findings in Section \ref{ssec:find}.
We define the overarching experimental setup in \Cref{ssec:setup},
then introduce our criteria for evaluating model performance in \Cref{ssec:eval}.
In the Lorenz '63 (L63) experiments (\Cref{ssec:l63}), we investigate how a simple Markovian random features model error term can be recovered using discrete
and continuous-time methods, and how those methods scale with the magnitude of error, data sampling rate,
availability of training data, and number of learned parameters.
In the Lorenz '96 Multiscale (L96MS) experiments (\Cref{ssec:l96}), we take this a step further by learning a Markovian random features closure term
for a scale-separated system, as well as systems with less scale-separation.
As expected, we find that the Markovian closure approach is highly accurate for a scale-separated regime.
We also see that the Markovian closure has merit even in cases with reduced scale-separation.
However, this situation would clearly benefit from learning a closure term with memory, a topic we turn to in
\Cref{ssec:partialNoisyObs}, where we demonstrate that
non-Markovian closure models can be learnt from noisy, partially observed data;
for low-dimensional cases (e.g. L63), our method of training converges to
return a good model with high short-term accuracy and long-term statistical
validity. For higher-dimensional cases (e.g. L96MS), we find the method to hold promise, but further research is required in this general area.
In \Cref{ssec:irnn}, we demonstrate why non-Markovian closures must be carefully initialized and/or controlled (e.g. via data assimilation) in order to ensure their long-term stability and short-term accuracy.

\subsection{Summary of Findings from Numerical Experiments}
\label{ssec:find}
\begin{enumerate}

 \item We find that hybrid modeling has better predictive performance than purely data-driven methods in a wide range of settings (see \Cref{fig:compare_eps_pathak,fig:compare_eps_GP} of \Cref{ssec:l63}): this includes scenarios where $f_0$ is highly accurate (but imperfect) and scenarios where $f_0$ is highly inaccurate (but nevertheless faithfully encodes much of the true structure for $f^\dag$).

  \item We find that hybrid modeling is more data-efficient than purely data-driven approaches (\Cref{fig:compare_T} of \Cref{ssec:l63}).

  \item We find that hybrid modeling is more parameter-efficient than purely data-driven approaches (\Cref{fig:compare_rfDim} of \Cref{ssec:l63}).

  \item Purely data-driven discrete-time modeling
can suffer from instabilities in the small timestep limit
$\Delta t \ll 1$;
  hybrid discrete-time approaches can alleviate this issue
when they are built from an integrator
  $\Psi_0$, as this will necessarily encode
the correct parametric dependence on $\Delta t \ll 1$ (\Cref{fig:compare_dt} of \Cref{ssec:l63}).

  \item In order to leverage standard supervised regression techniques, continuous-time methods require good estimates of derivatives $\dot{x}(t)$ from the data. \Cref{fig:compare_dt} of \Cref{ssec:l63}
quantifies this estimation as a function of data sample rate.

  \item Non-Markovian model error can be captured by Markovian terms in scale-separated cases.
  \Cref{ssec:l96} demonstrates this quantitatively in \Cref{fig:l96bignew},
  and qualitatively in \Cref{fig:l96residuals}.
  Beyond the scale-separation limit, Markovian terms will fail for trajectory forecasting.
  However, Markovian terms may still reproduce invariant statistics in dissipative systems
  \agree{(for example, in cases with short memory-length)}.
  \Cref{ssec:l96} demonstrates this quantitatively in \Cref{fig:l96bignew};
  \Cref{fig:l96residuals} offers intuition for these findings.

\item
Non-Markovian description of model error is needed to accurately represent problems
where the hidden dynamics is not scale-separated from the observed dynamics.
\Cref{ssec:partialNoisyObs} shows how partial and noisy observations can be exploited by augmented ODE models of form \eqref{eq:ml_cont_mem}
when the noise and hidden dynamics are learnt implicitly by auto-differentiable data assimilation.
We observe high-quality reconstruction of the L63 system along its first-component when choosing a correct (\Cref{fig:l63dy2}) or overly enlarged (\Cref{fig:l63dy10}) hidden dimension.
We also observe promising reconstruction of the L96MS system in its slow components (\Cref{fig:l96dy72}); however, long-time solutions to the learnt model exhibited instabilities inconsistent with the true system.

\item Non-Markovian models must be carefully initialized, and indeed
data assimilation is needed, in order to ensure accuracy
(\Cref{ssec:irnn}) of invariant statistics (\Cref{fig:da_l63_ay_inv}), long-term stability (\Cref{fig:da_l63_ay2_traj_long}), and accurate short-term predictions (\Cref{fig:da_l63_ay2_traj_short}).
We explain observed phenomena in terms of the properties of
the desired lower-dimensional
invariant manifold which is embedded within the higher dimensional
system used as the RNN's basis of approximation.

\end{enumerate}

\subsection{Learning Markovian Model Errors from Noise-Free Data}

\subsubsection{Experimental Set-Up}
\label{ssec:setup}
In the Markovian error modeling experiments described in \Cref{ssec:l63,ssec:l96}, whether using continuous- or discrete-time
models,  we train a random features model on noise-free trajectories from the true system (an ODE). The problems we study
provably have a compact global attractor and are provably
(L63) or empirically (L96MS) ergodic; the invariant
distribution is supported on the global attractor and captures
the statistics  of long-time trajectories which, by ergodicity,
are independent of initial condition.
The data trajectories are generated using scipy's implementation of Runge-Kutta 5(4) (via $\texttt{solve\_ivp}$) with absolute and relative tolerances both $10^{-9}$ and maximum step size $10^{-4}$ \citep{dormand_family_1980,virtanen_scipy_2020}.
In order to obtain statistical results, we create 5 training trajectories from the true system of interest with initial conditions sampled independently from its attractor.
Note that each training trajectory is long enough to explore the
attractor, and each is used to train a separate model; the purpose is to observe the variance in learnt models with respect to randomly sampled paths through the attractor.
We use the same sampling procedure to generate short independent validation and testing trajectories---we use 7 validation trajectories and 10 testing trajectories (these are short because we only use them to evaluate a model's short term forecast performance; when assessing long-term statistics of a learnt model, we compare to very long simulations from the true system).
All plots use error bars to represent empirical estimates of the mean and standard deviation of the presented performance metric, as computed by
ensembling the performance of the 5 models (one per training trajectory) over the 10 testing trajectories for a total of 70 random performance evaluations.

Each training procedure also involves an independent draw of the
random feature functions as defined in \eqref{eq:rfd}.
A validation step is subsequently performed to optimize the hyperparameters $\omega, \beta$, as well as the regularization parameter $\lambda$.
We automate this validation using Bayesian Optimization \citep{mockus_bayesian_1989,nogueira_bayesian_2014}, and find that it typically identifies good hyperparameters within 30 iterations. The entire process of entraining a model
to a single, long training trajectory (including hyperparameter validation)
typically takes approximately $30$ minutes
on a single core of a 2.1GHz Skylake CPU with an allocated 1GB RAM.
Given a realization of random features and an optimal $\lambda$, we obtain the minimizer $C^\ast$ using the Moore-Penrose Pseudoinverse implemented in scipy ($\texttt{pinv2}$).
This learned $C^\ast$, paired with its random feature realization, is then used to predict 10 unseen testing trajectories (it is given the true initial condition for each of these testing trajectories).

When implementing in continuous-time, given high frequency but
discrete-time data, two computational issues must be addressed:
(i) extrapolation of the data to continuous-time; (ii) discretization
of the resulting integrals. The approach we adopt avoids ``inverse
crimes'' in which favourable behaviour is observed because
of agreement between the data generation mechanism (with a specific
integrator) and the approximation of the objective functions
\citep{colton_inverse_2013,kaipio_statistical_2005,wirgin_inverse_2004};
see \citet{queiruga_continuous--depth_2020} for further illustration
of this issue and \citet{keller_discovery_2021,du_discovery_2021} for a rigorous analysis of
this inversion process in the context of linear multistep integration methods for deep learning.
We interpolate the data with a spline, to obtain
continuous-time trajectories, and then discretize the integrals
using a simple Riemann sum; this strikes a desirable
balance between robustness and efficiency and avoids inverse crimes.
The discrete-time approaches, however,
are able to learn not only model-discrepancy, but also integrator-based discrepancies;
hence, the discrete-time methods may artificially appear to outperform continuous-time approaches,
when, in fact, their performances might simply be considered to be
comparable.

\subsubsection{Evaluation Criteria}
\label{ssec:eval}
Models are evaluated against the test set for their ability to predict individual trajectories, as well as invariant statistics (the invariant measure and the auto-correlation function).

\textbf{Trajectory Validity Time:} Given threshold $\gamma>0$, we
 find the first time $t_{\gamma}$ at which
the norm of discrepancy between true and approximate
solutions reaches $\gamma$:
$$t_{\gamma}= {\rm argmin}_{t \in [0,T]} \bigg \{ t \colon \ \| x(t) - x_m(t) \|  \geq \gamma \overline{\| x(t)\|} \bigg\},$$
where $x(t)$ is the true solution to \eqref{eq:fmdag}, $x_{m}(t)$ is the learned approximation,
and the normed time average $\overline{\| x(t)\|}$ is approximated from training data.
If the threshold is not violated on $[0,T]$, we define $t_\gamma := T$; this is rare in practice.
We take $\gamma = 0.05$ (i.e. $5\%$ relative divergence).

\textbf{Invariant Distribution:} To quantify errors in our reconstruction of the invariant measure,
we consider the Kullback-Leibler (KL) divergence \citep{kullback_information_1951} between the true invariant measure $\mu$
and the invariant measure produced by our learned model $\mu_m$.
We approximate the divergence
$$d_\text{KL}(\mu, \mu_m) := \int_\R \log\bigg(\frac{d\mu}{d\mu_m}\bigg)d\mu$$
by integrating kernel density estimates
with respect to the Lebesgue measure.

\textbf{Autocorrelation:}
We compare the autocorrelation function (ACF) with respect to the invariant distribution of the true and learned models.
We approximate the ACF using a fast-fourier-transform for convolutions \citet{seabold_statsmodels_2010},
and compare them via a normalized $L_2$ norm of their difference.


\subsubsection{Lorenz '63 (L63)}
\label{ssec:l63}

\paragraph{\textbf{Setting}}
The L63 system \citep{lorenz_deterministic_1963} is described by the following ODE
\begin{equation}
  \begin{aligned}
    \label{eq:l63}
    \dot{u}_x &= a(u_y-u_x) \\
    \dot{u}_y &= bu_x - u_y - u_xu_z \\
    \dot{u}_z &= -cu_z + u_xu_y
  \end{aligned}
\end{equation}
whose solutions are known to exhibit chaotic behavior for parameters $a=10,\ b=28,\ c=\frac{8}{3}$.
We align these equations with our framework, starting from
equation \eqref{eq:fdag}, by letting
$x = (u_x, u_y, u_z)^T$ and
defining $f^\dag(x)$ to be the vector field appearing on
the right-hand-side in \eqref{eq:l63}.
We define a discrete solution operator $\Psi^\dag$ by numerical integration of $f^\dag$ over a fixed time window $\Delta t$ corresponding to a uniform data sampling rate, so that the true system is given by
\eqref{eq:fdag} in continuous-time and \eqref{eq:Psidag} in discrete-time.

To simulate scenarios in which our available physics are good, but imperfect,
we assume there exists additive unknown model error of form
\begin{equation}
  \label{eq:mEps}
  m^\dag(x) = \epsm \ m_1(x)
\end{equation}
with function $m_1$ determining the structure of model error,
and scalar coefficient $\epsm$ determining its magnitude.
Recall that $f^\dag=f_0+m^\dag$ and we assume $f_0$ is known to us.
Our task is then to learn $f^\dag$ by learning $m^\dag$ and adding
it to $f_0$.
The discrete solution operator $\Psi_0$ is obtained as in \eqref{eq:Psi0} by
numerical integration of $f_0$ over a fixed time window $\Delta t$.

To simplify exposition, we explicitly define $m^\dag$,
then let $f_0 := f^\dag - m^\dag$.
We first consider the setting where
\begin{equation}
  \label{eq:mPathak}
  m_1(x) := \ \begin{bmatrix} 0 \\ b u_x \\ 0\end{bmatrix}
\end{equation}
(as in \citep{pathak_hybrid_2018})
and modulate $\epsm$ in \eqref{eq:mEps} to control the magnitude of the error term.
In this case, $f_0$ can be viewed as the L63 equations
with perturbed parameter $\tilde{b} = b(1-\epsm)$, where $b$ is
artificially decreased by $100\epsm\%$.

Then, we consider a more general case of heterogeneous, multi-dimensional residual error by
drawing $m_1$ from a zero-mean Gaussian Process (GP) with a radial basis kernel (lengthscale 10).
We form a map from $\R^3$ into itself by constructing
three independent draws from a scalar-valued GP on $\R^3.$
The resulting function is visualized in two-dimensional
projections in \Cref{fig:gp_contour}.

\begin{figure}[!htbp]
\centering
  \includegraphics[width=1\textwidth]{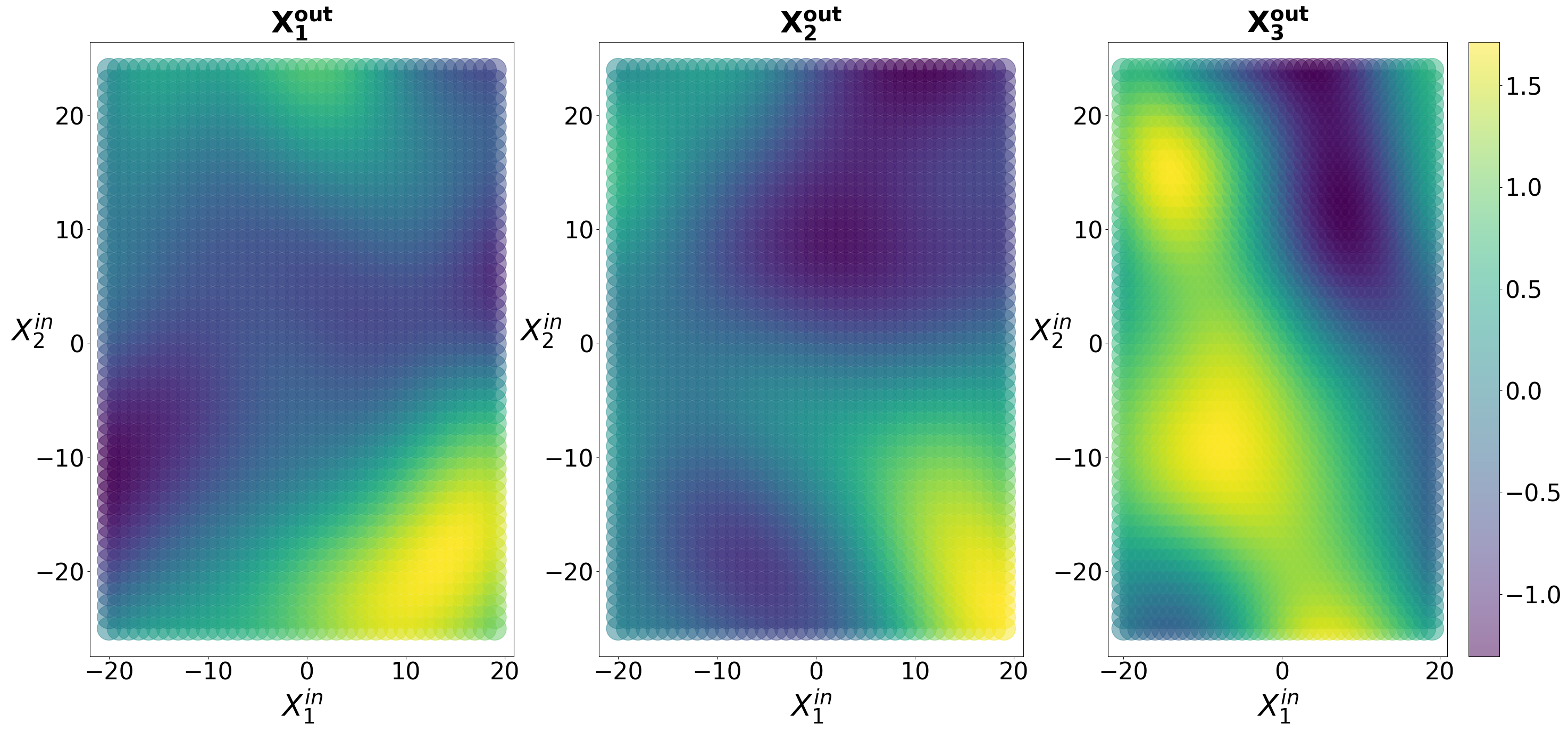}
 \caption{
 Here we visualize an example of the function $m_1$ in \eqref{eq:mEps}, which is obtained as a single random draw from a zero-mean Gaussian Process mapping $\R^3 \to \R^3$.
 We have plotted its output surface as three scalar functions (left to right) of the first two inputs (the plot axes) with the third input component fixed at 0.
 }
\label{fig:gp_contour}
\end{figure}

Observe that in the continuous-time framing, changes to $\epsm$ do not impact the complexity of the learned error term; however, it does grow the magnitude of the error term.
In the discrete-time framing, larger values of $\epsm$ can magnify the complexity of the discrepancy $\Psi_0(x) - \Psi^\dag(x)$.

\paragraph{\textbf{Results}}

We perform a series of experiments with the L63 system in order to illustrate key points about using data to learn model errors in dynamical systems.
First, we demonstrate that hybrid modeling tends to outperform data-only and physics-only methods in terms of prediction quality.
We first consider model error as in \eqref{eq:mPathak}; see \Cref{fig:compare_eps_pathak}
in which we study performance (validity time) versus
model error amplitude ($\epsm$),
using random feature maps with $D=200$, and a
single trajectory of length $T=100$ sampled at
timestep $\Delta t=0.001$.  Unless otherwise specified, this is also
the configuration used in subsequent experiments.

We see identical trends in \Cref{fig:compare_eps_GP} for a more general case with the non-parametric model error term constructed
from Gaussian processes.  Interestingly, we see that for small
and moderate amounts of model error $\epsm$,
the hybrid methods substantially outperform data-only and physics-only methods.
Eventually, for large enough model discrepancy,
the hybrid-methods and data-only methods have similar
performance; indeed the hybrid-method may be outperformed
by the data-only method at large discrepancies.
For the simple model error this appears to occur when the
discrepancy term is larger in magnitude
than $f_0$ (e.g. for $b=28$ and $\epsm=2$, the model error term $\epsm b u_x$ can take on values larger than $f^\dag$ itself).

\Cref{fig:compare_eps_GP} also shows that a continuous-time approach is favored over discrete-time when using data-only methods, but suggests the converse in the hybrid modeling context. We suspect this is an artifact of the different integration schemes used in data generation, training, and testing phases; the data are generated with a higher-fidelity integrator than the one available in training and testing.
For the continuous-time method, this presents a fundamental limitation to the forecast quality (we chose this to avoid having artificially high forecast validity times).
However, the discrete-time method can overcome this by not only learning the mechanistic model discrepancy, but also the discrepancy term associated with a mis-matched integrator.
This typically happens when a closure is perfectly learnable and deterministic (i.e. our Lorenz '63 example);
in this case, the combination of physics-based and integrator-sourced closures can be learned nearly perfectly.
In later experiments with a multiscale system, the closures are considered approximate (they model the mean of a noisy underlying process)
and the discrete- and continuous-time methods perform more similarly, because the inevitable imperfections of the learned closure term dominate the error rather than the mis-specified integrator.
Note that approximate closures driven by scale-separation are much more realistic;
thus we should not expect the hybrid discrete-time method to dramatically outperform hybrid continuous-time methods unless other limitations are present (e.g. slow sampling rate).

Importantly, the parameter regime
for which hybrid methods sustain advantage over the imperfect
physics-only method is substantial; the latter has trajectory
predictive performance which drops off rapidly for very small $\epsm$.
This suggests that an apparently uninformative model can be efficiently
modified, by machine learning techniques,
to obtain a useful model that outperforms
a \emph{de novo} learning approach.

\begin{figure}
\centering
\begin{subfigure}{0.47\textwidth}
  \includegraphics[width=\textwidth]{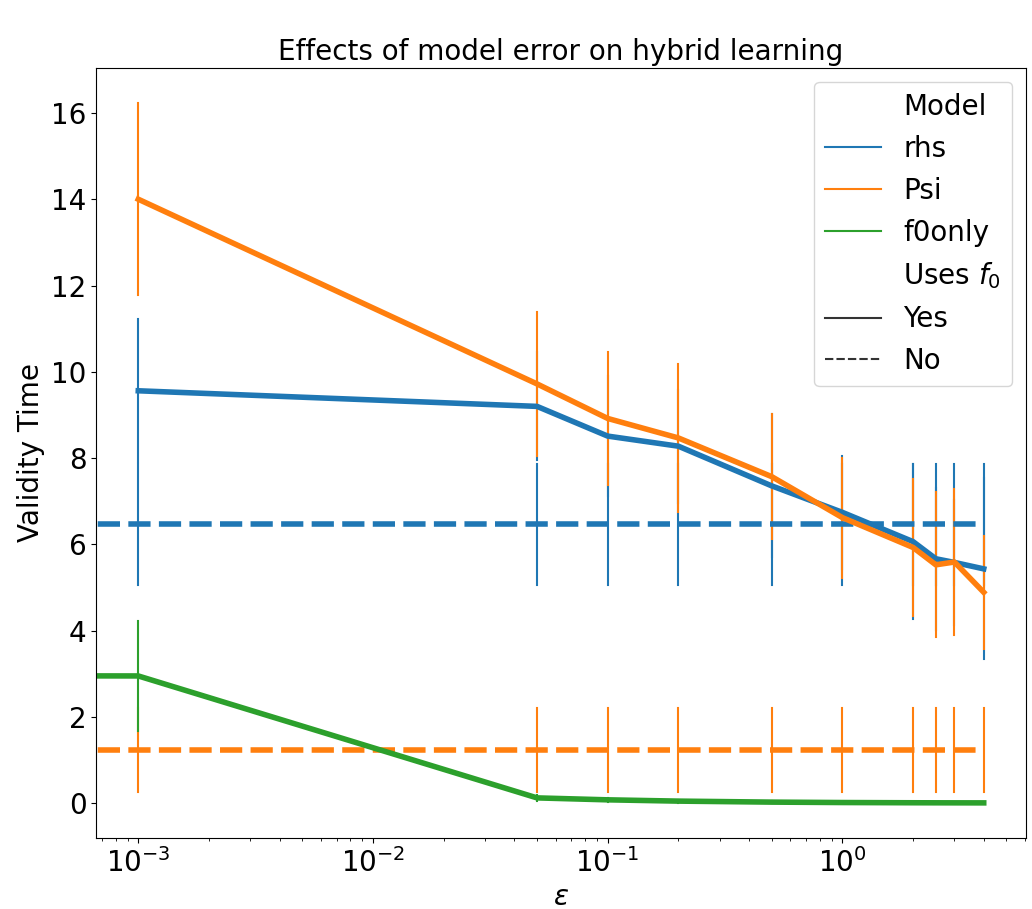}
  \caption{}
  \label{fig:compare_eps_pathak}
\end{subfigure} \hfill 
\begin{subfigure}{0.47\textwidth}
  \includegraphics[width=\textwidth]{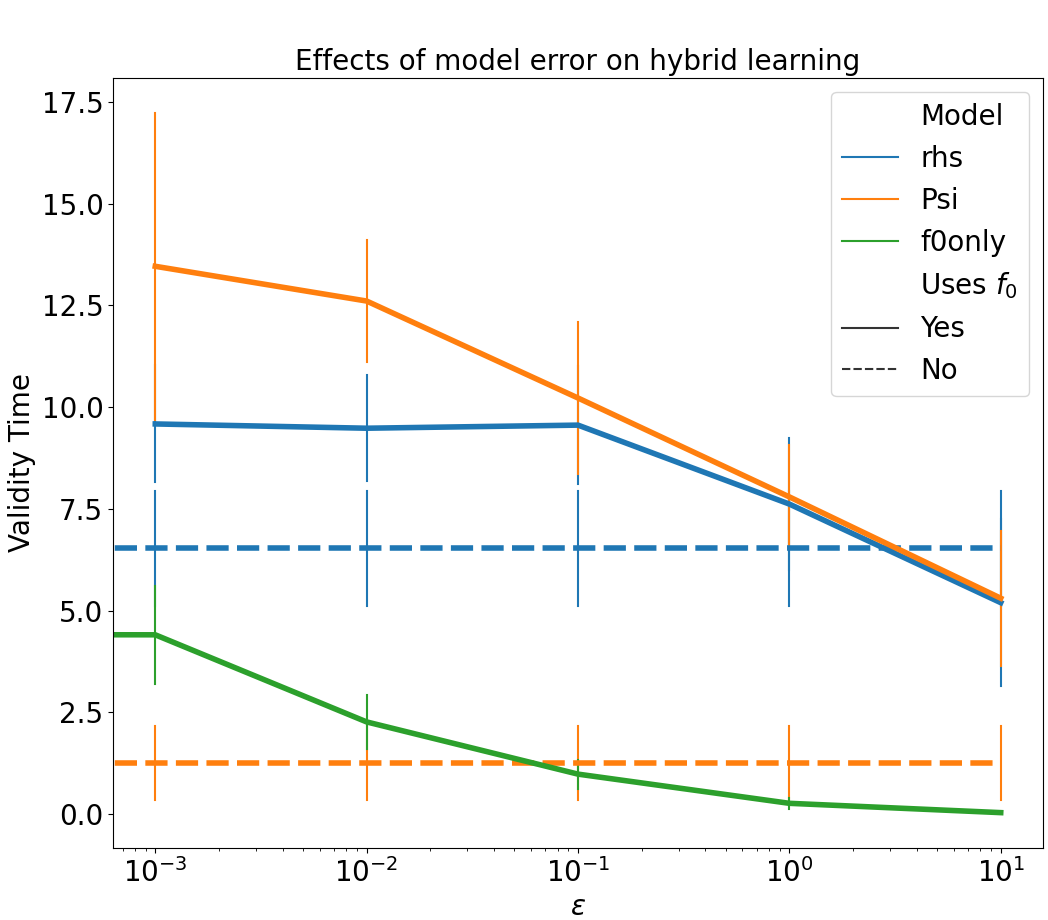}
  \caption{}
  \label{fig:compare_eps_GP}
\end{subfigure}

 \caption{These plots shows the temporal length of the forecast validity of our learnt models of L63 \eqref{eq:l63}, each as a function of model error, as parameterized by $\epsm$ \eqref{eq:mEps}
 (with $D=200$, $T=100$, and $\Delta t=0.001$).
Continuous-time methods are shown in blue, discrete-time approaches in orange.
Dotted lines indicate purely data-driven methods to learn the
entire vector field defining the dynamics; solid lines indicate methods
that learn perturbations to the imperfect mechanistic models $f_0$ or $\Psi_0$.
Integration using the imperfect mechanistic model, without recourse
to data, is shown in green.
In \Cref{fig:compare_eps_pathak}, we employ the linear form of model error $m_1$ defined in \eqref{eq:mPathak}.
In \Cref{fig:compare_eps_GP}, we let $m_1$ be a single draw from a Gaussian Process, whose structure is shown in
\Cref{fig:gp_contour}. Here, we plot means, with error bars as 1 standard deviation.
 }
 \label{fig:compare_eps}
\end{figure}

Next, we show that hybrid methods simplify the machine learning task in terms of complexity of the learned function and, consequently, the amount of data needed for the learning.
\Cref{fig:compare_T} examines prediction performance (validity time)
as a function of training data quantity using random feature maps with $D=2000$ and a fixed parametric model error ($\epsm = 0.2$  in \eqref{eq:mEps})  and sampling rate $\Delta t = 0.01$.
We see that the hybrid methods substantially outperform the data-only approaches in regimes with limited training data.
For the continuous-time example, we see an expected trend, where
the data-only methods are able to catch up to the hybrid methods with the acquisition of more data.
The discrete-time models do not exhibit this behavior, but we expect the data-only discrete-time model to eventually catch up, albeit with additional training data and number of parameters.
Note that greater expressivity is also required from data-only methods---our choice of a large $D=2000$ aims to give all methods ample expressivity, and thus test convergence with respect to training data quantity alone.
These results demonstrate that the advantage of hybrid modeling is magnified when training data are limited and cannot fully inform \emph{de novo} learning.
\Cref{fig:compare_rfDim} further studies the impact of expressivity by again fixing a parametric model error ($\epsm=0.05$ in \eqref{eq:mEps}), training length $T=100$, and sampling rate $\Delta t= 0.001$.
We see that all methods improve with a larger number of random features,
but that relative superiority of hybrid methods is maintained even for $D=10000$.

\begin{figure}[!htbp]
\centering
  \includegraphics[width=0.7\textwidth]{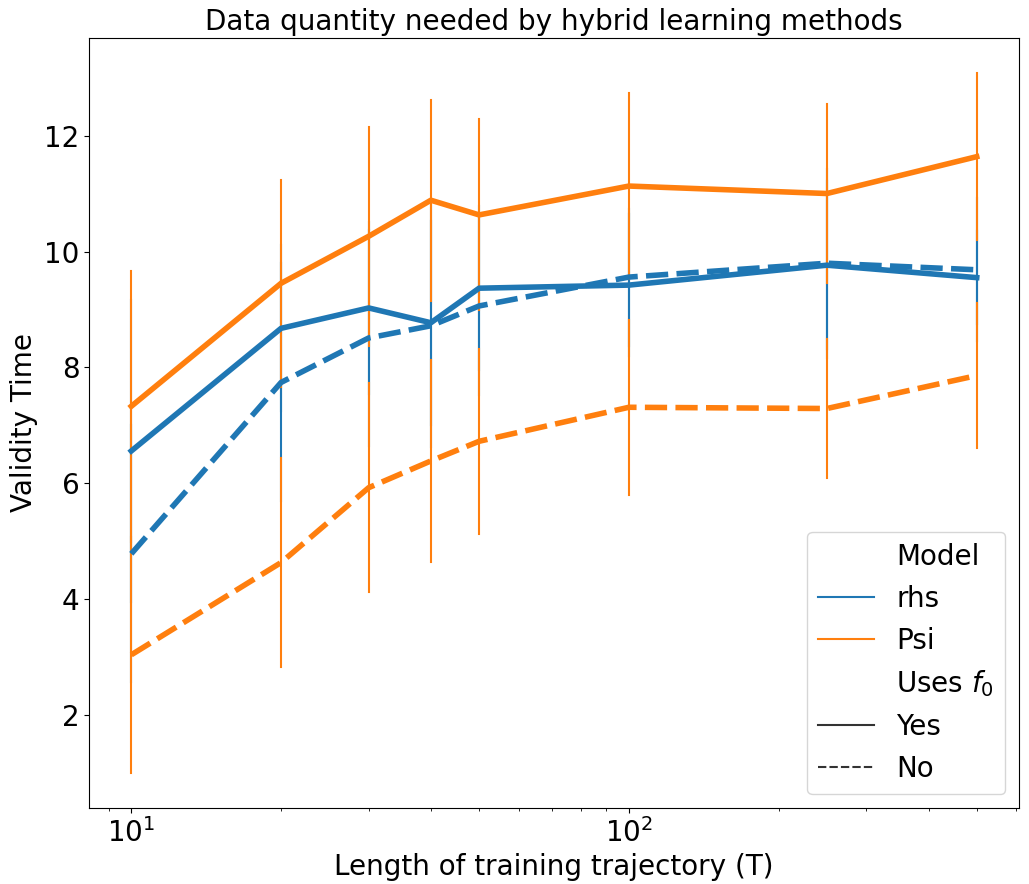}
 \caption{Here we examine the performance of the proposed methods as a function of the length of the interval over which the training data is provided, where $\Delta t=0.01$, ($\epsm=0.2$ in \eqref{eq:mEps}), and $D=2000$ are held constant for the L63 example \eqref{eq:l63}.
 See description of \Cref{fig:compare_eps} for explanation of legend.
We observe that all methods improve with increasing training lengths.
We see that, in continuous-time, the primary benefit in
hybrid modeling is when the training data are limited.}
\label{fig:compare_T}
\end{figure}

\begin{figure}[!htbp]
\centering
  \includegraphics[trim=1.5cm 1cm 3cm 2cm, clip, width=0.7\textwidth]{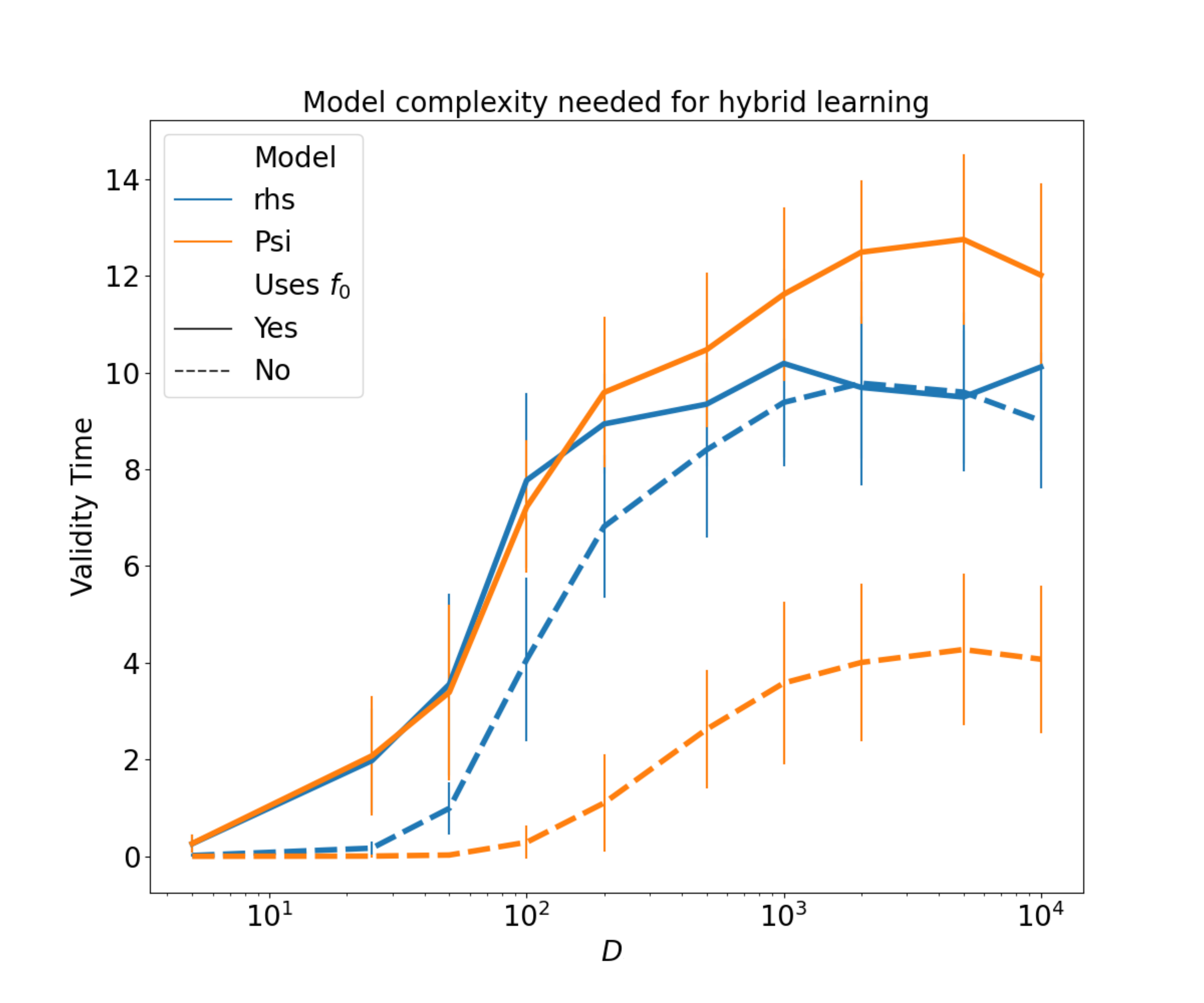}
  \caption{Here we examine the performance of the proposed methods as a function of model complexity, where $\Delta t=0.001$, $\epsm=0.05$, and $T=100$ are held constant for the L63 example \eqref{eq:l63}.
  See description of \Cref{fig:compare_eps} for explanation of legend.
  We observe that all methods improve with increasing number of parameters,
  and that hybrid methods are especially beneficial when available complexity is limited.}
\label{fig:compare_rfDim}
\end{figure}

Finally, we study trade-offs between learning in discrete- versus continuous-time for the L63 example \eqref{eq:l63}.
\Cref{fig:compare_dt} examines prediction performance (validity time)
as a function of data sampling rate $\Delta t$ using random feature maps with $D=200$ with a fixed parametric model error ($\epsm = 0.05$ in \eqref{eq:mEps}) and an abundance of training data $T=1000$.
We observe that for fast sampling rates ($\Delta t < 0.01$),
the continuous-time and discrete-time hybrid methods have similar performance.
For $\Delta t > 0.01$, derivatives become difficult to estimate from the data and the performance of the continuous-time methods rapidly decline.
However, the discrete-time methods sustain their predictive performance for slower sampling rates ($\Delta t \in (0.01, 0.1)$). At some point, the discrete-time methods deteriorate as well, as the discrete map
becomes complex to learn at longer terms because of the
sensitivity to initial conditions that is a hallmark of chaotic
systems. Here, the discrete-time methods begin to fail around $\Delta t=0.2$; note that they can be extended to longer time intervals by increasing $D$ and amount of training data, but returns diminish quickly.

 \begin{figure}[!htbp]
 \centering
   \includegraphics[width=0.7\textwidth]{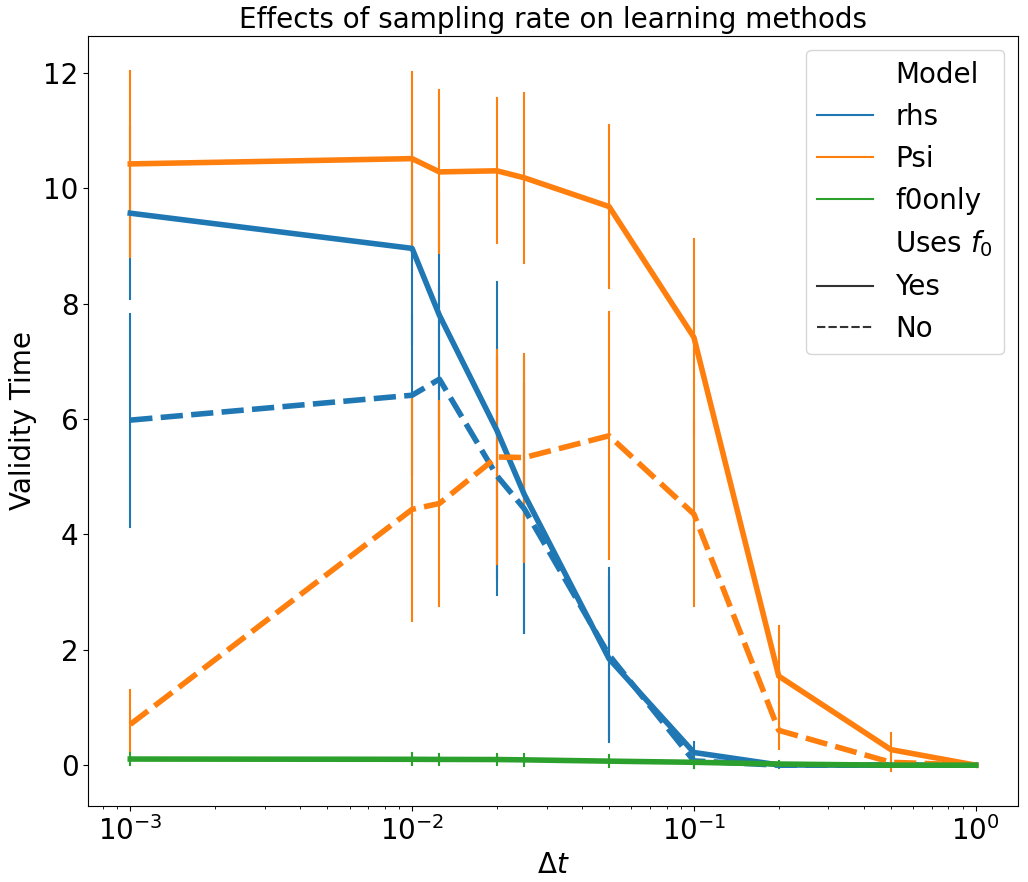}
  \caption{This shows temporal forecast validity as a function of the step size of training data for the tested methods in the L63 example \eqref{eq:l63}.
  We hold fixed $D=200$, $\epsm=0.05$, and $T=1000$.
  See description of \Cref{fig:compare_eps} for explanation of legend.
  We see that while purely data-driven discrete-time methods struggle at short time steps, the hybrid version thrives in this scenario. All approaches, of course, eventually decay as large time steps create more complex forward maps, due to sensitivity to initial
conditions. We also see continuous-time methods work well for small time steps, then deteriorate in tandem with quality of estimated derivatives.
  }
 \label{fig:compare_dt}
 \end{figure}

\subsubsection{Lorenz '96 Multiscale (L96MS) System}
\label{ssec:l96}
\paragraph{\textbf{Setting}}
Here, we consider the multiscale system \citep{lorenz_predictability-problem_1996}
of form \eqref{eq:fgdag}, where each variable $X_k \in \R$ is coupled to a subgroup of fast variables $Y_{k} \in \R^J$. We have $X \in \R^K$ and $Y \in \R^{K \times J}$. For $k=1 \dots K$ and $j=1 \dots J$, we write
\begin{subequations}
  \label{eq:l96ms}
\begin{align}
\dot{X}_k &= f_k(X) + h_x \overline{Y}_k \\
\dot{Y}_{k,j} &= \frac{1}{\scaleeps} r_{j}(X_k, Y_k) \\
\overline{Y}_k &= \frac{1}{J} \sum_{j=1}^J Y_{k,j} \\
f_k(X) &= -X_{k-1}( X_{k-2} - X_{k+1}) - X_{k} + F \\
r_{j}(X_k, Y_k) &= -Y_{k,j+1}(Y_{k,j+2} - Y_{k,j-1}) - Y_{k,j}  +  h_y X_{k} \\
X_{k + K} &= X_k, \quad Y_{k+K,j} = Y_{k,j}, \quad Y_{k,j+J} = Y_{k+1,j}
\end{align}
\end{subequations}
where $\scaleeps > 0$ is a scale-separation parameter, $h_x, h_y \in \R$ govern the couplings between the fast and slow systems, and $F > 0$ provides a constant forcing.
We set $K = 9, \ J = 8,\  h_x = -0.8,\ h_y = 1,\ F = 10$; this leads to
chaotic dynamics for $\scaleeps$ small.
When studying scale-separation, we consider $\scaleeps \in \{ 2^{-7}, \ 2^{-5},\ 2^{-3},\ 2^{-1} \}$.

We consider the setting in which we learn Markovian {random features
models in variable $X$ alone, from $X$ data generated by the
coupled $(X,Y)$ system.
Large scale-separation between the observed ($X$) and
unobserved ($Y$) spaces can simplify the problem of accounting
for the unobserved components;
in particular, for sufficient scale-separation, we expect
 a Markovian term to recover a large majority of the residual errors.
In fact, we further simplify this problem by learning a scalar-valued model error $M$
that is applied to each $X_k$ identically in the slow system:
$$\dot{X}_k \approx f_k(X) + M(X_k).$$
This choice stems from observations about statistical interchangeability
amongst the slow variables of the system; these properties of
the L96MS model in the scale-separated regime are
discussed in \citep{fatkullin_computational_2004}.
We can directly align our reduction of \eqref{eq:l96ms} with the Markovian hybrid learning framework in \eqref{eq:fmdag}
as follows:
\begin{align*}
  \dot{X} &\approx f_0(X) + m(X) \\
  f_0(X) &:= [f_1(X), \ \cdots \ , f_K(X) ]^T \\
  m(X) &:= [M(X_1), \ \cdots \ , M(X_K) ]^T.
\end{align*}

\paragraph{\textbf{Results}}

We plot the performance gains of our hybrid learning approaches in \Cref{fig:l96bignew} by considering
validity times of trajectory forecasts, estimation of the invariant measure,
and ACF estimation.
In all three metrics (and for all scale-separations $\scaleeps$), \emph{de novo} learning in discrete ($\Psi^\dag \approx m$) and continuous-time ($f^\dag \approx m$)
is inferior to using the nominal mechanistic model $f_0$.
We found that the amount of data used in these experiments is insufficient to learn the full system from scratch.
On the other hand, hybrid models in discrete ($\Psi^\dag \approx \Psi_0 + m$) and continuous-time ($f^\dag \approx f_0 + m$)
noticeably outperformed the nominal physics.

Surprisingly, \Cref{fig:l96bignew} shows that the Markovian closure methods still qualitatively reproduce the invariant statistics even for large $\scaleeps$ settings where we would expect substantial memory effects.
\Cref{fig:l96bignew} also demonstrates this quantitatively using KL-divergence between invariant measures and mean-squared-error between ACFs.
It seems that for this dissipative system, memory effects tend to average out in the invariant statistics.
However, the improvements in validity time for trajectory-based forecasting deteriorate for $\scaleeps = 2^{-1}$.

To visualize this non-Markovian structure, and how it might be exploited, we examine the residuals from $f_0$ in \Cref{fig:l96residuals}
and observe that there are discernible trajectories walking around the Markovian closure term.
For small $\scaleeps$, these trajectories oscillate rapidly around the closure term.
For large $\scaleeps$ (e.g. $2^{-1}$), however, we observe a slow-moving residual trajectory around the Markovian closure term. This indicates the presence of a stronger memory component, and thus would benefit from
a non-Markovian approach to closure modeling.

\citet{jiang_modeling_2020} show that the memory component in this setting with $\scaleeps = 2^{-1}$
can be described using a closure term with a simple delay embedding of the previous state at lag $0.01$.
They learn the closure using a kernel method cast in an RKHS framework, for which random feature methods provide an approximation.

\begin{figure}[!htbp]
\centering
  \includegraphics[trim=3cm 0 5cm 4cm, clip, width=1\textwidth]{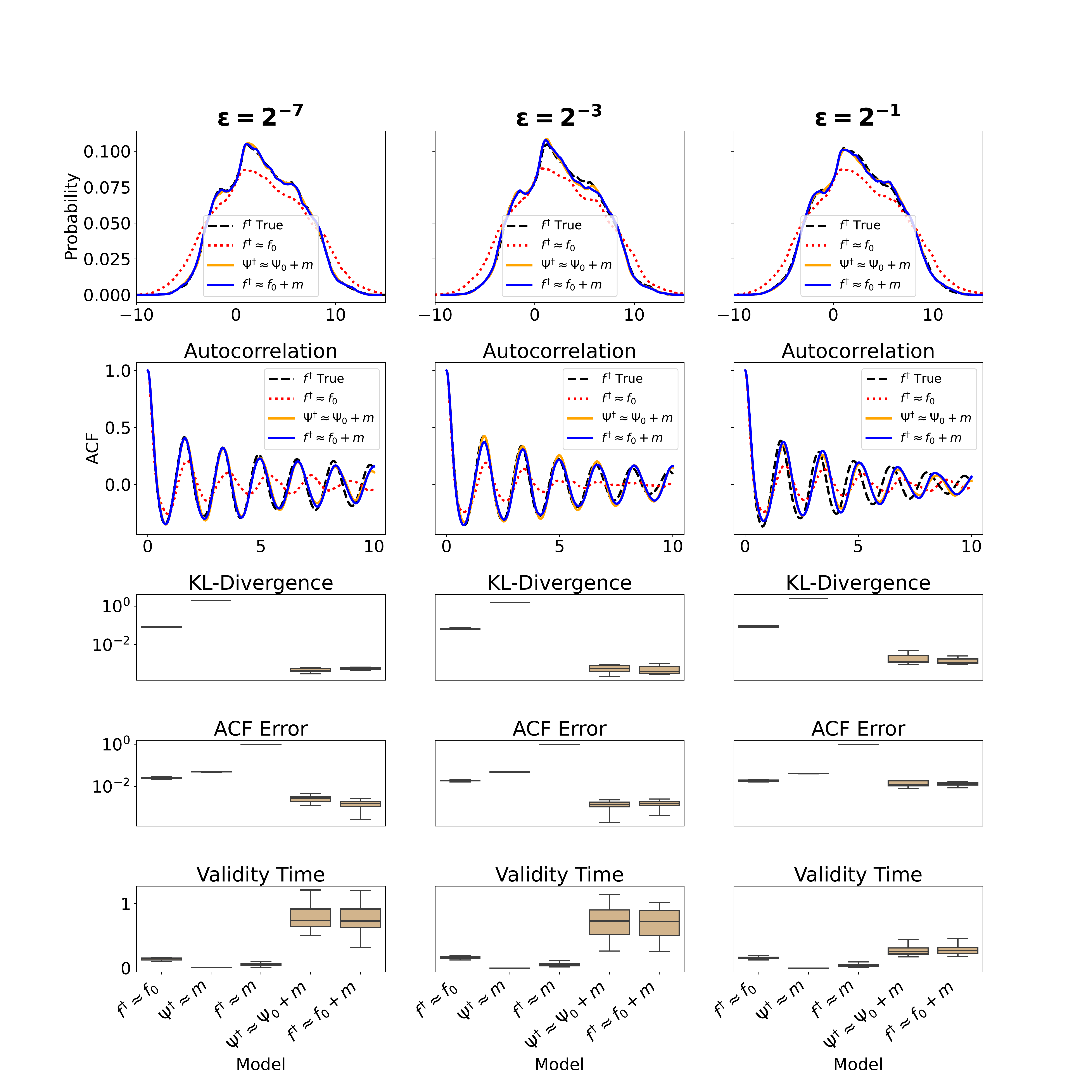}
 \caption{This figure shows the performance of different approaches to modeling the L96MS slow subsystem \eqref{eq:l96ms}.
In $f^\dag \approx f_0$, we only use the nominal physics $f_0$.
In $\Psi^\dag \approx m$ and $f^\dag \approx m$, we try to learn the entire right-hand-side using only data (in discrete- and continuous-time settings, respectively).
In $\Psi^\dag \approx \Psi_0 + m$ and $f^\dag \approx f_0 + m$, we focus on learning Markovian residuals
for the known physics (in discrete- and continuous-time settings, respectively).
\agree{The residual-based correctors substantially outperform the nominal physics and purely data-driven methods according to all presented metrics: invariant measure (shown qualitatively in the first row and quantitatively in the third row),
ACF (shown qualitatively in the second row and quantitatively in the fourth row),
and trajectory forecasts (shown in the final row).
The boxplots show the distributions of quantitative metrics (e.g. KL-divergence, squared errors, validity time), which come from different models, each trained on a different trajectory, and generated using an independent random feature set.
Notably, the Markovian residual-based methods' performance deteriorates for small scale-separation ($\scaleeps=2^{-1}$), where the Markovian assumption breaks down.}
}
\label{fig:l96bignew}
\end{figure}

\begin{figure}[!htbp]
\centering
  \includegraphics[width=\textwidth]{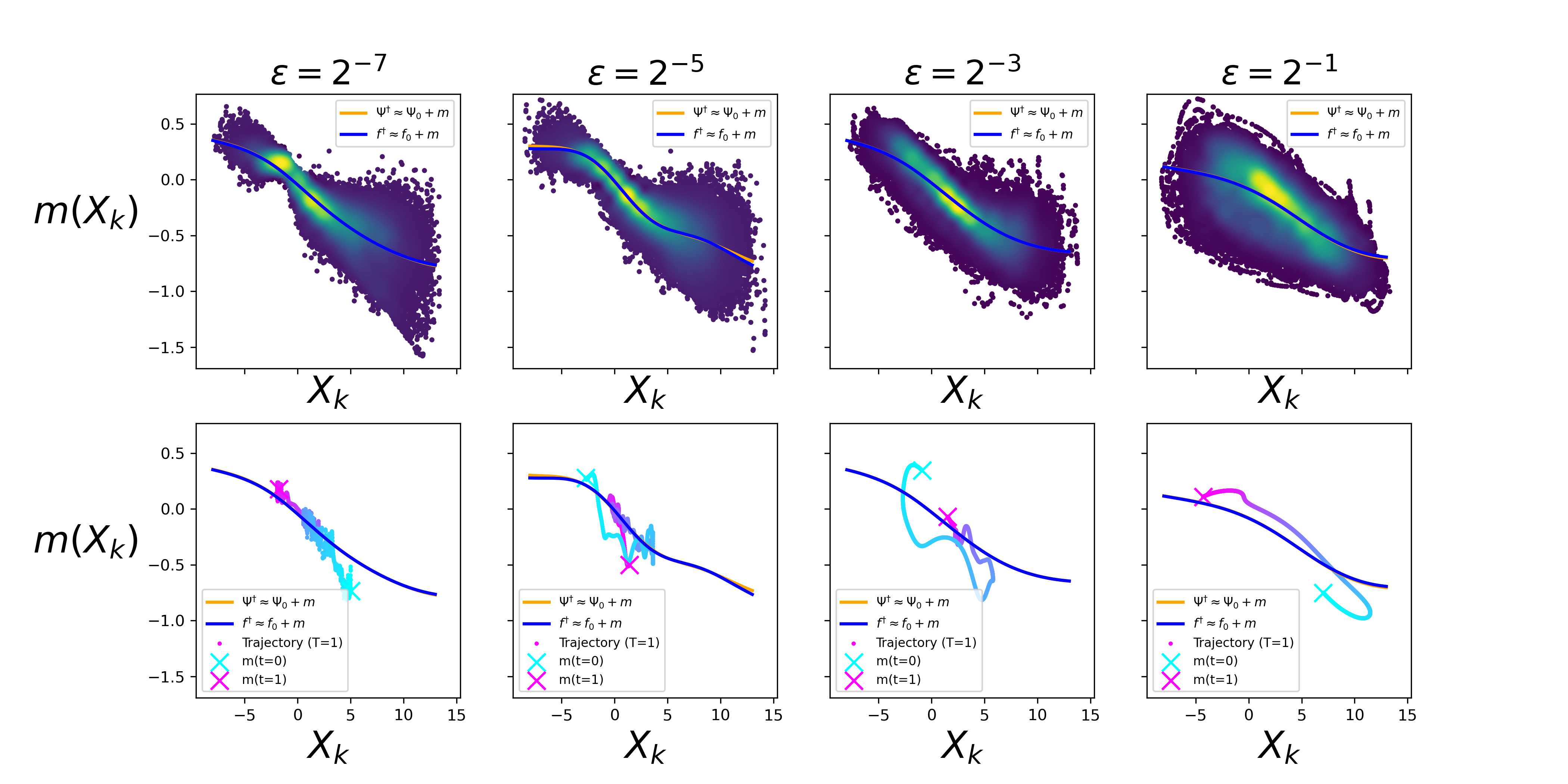}
 \caption{This figure shows the observed and estimated residuals of the nominal physics model $f_0$
 for the L96MS slow subsystem \eqref{eq:l96ms} at different scale-separation factors.
 The first row shows the density of these residuals (yellow is high density, blue is low), as well as the fit of our closure terms in continuous- (blue) and discrete- (orange) time
 (the discrete model was normalized by dividing by $\Delta t$).
The second row shows temporal structure in the errors of our residual fit by superimposing a short (T=1) one-dimensional trajectory (this represents $\sim0.1\%$ of training data).
 }
\label{fig:l96residuals}
\end{figure}

\subsection{Learning From Partial, Noisy Observations}
\label{ssec:partialNoisyObs}

In this section, we focus on the non-Markovian setting outlined in \Cref{ssec:MultiScale},
and attempt to model the dynamics of the observable using \eqref{eq:ml_cont_mem}, with $f_1$, $f_2$
given by two-layer, fully connected neural networks with GeLU activations \citep{hendrycksGELU},
and perform the learning by minimizing \eqref{eq:JcmemMatt} from \Cref{ssec:mlda},
using 3DVAR for the data assimilation \cite{law_data_2015,law_analysis_2013}, with the ADAM optimizer \citep{adamOpt}.
The learning rate was initialized at $0.01$ and tuned automatically using a scheduler that halved the learning rate if the training error had not decreased over 10 (mini-batched) epochs. Data were sampled at $\Delta t=0.01$ in all cases, and normalized to have mean zero and unit variance.
Numerical integration was performed with the \texttt{torchdiffeq} implementation of the Dormand-Prince adaptive fifth-order Runge-Kutta method: for the L63 example, simple backpropagated autodifferentiation was performed through this solver; for the L96MS example, we used the adjoint method provided by \citep{rubanova_latent_2019}.

\subsubsection{Lorenz '63}
We first consider modeling the dynamics of the first-component of the L63 system in \eqref{eq:l63}, where we noisily observe the first-component -- that is, we observe a noisy trajectory of $u_x$ (i.i.d. additive zero-mean, variance-one Gaussian), but do not observe the remaining components $u_y, u_z$.
We jointly trained on 100 trajectories, each of length $T=10$ and randomly initialized from a box around the attractor; we chose this approach to ensure that
we had data coverage both on and off the attractor
although we note that similar success is obtained with a single trajectory
of length $T=1000$.). The neural network had width $50.$
We chose an assimilation time of $\tau_1=3$ and a forecast time of $\tau_2=0.1$.
The optimization ran for approximately 200 epochs, and took roughly 24hrs on a
single GPU.
Adequate results were obtained using a fixed 3DVAR gain matrix
$K = [0.5, 0, 0]^T$. However, we present results using the algorithm
in which $K=\theta_{\textrm{DA}}$
is jointly learned along with parameters $\theta_{\textrm{DYN}}$,
as described in \Cref{ssec:mlda}; this demonstrates that the gain need not be
known \emph{a priori}.

First, we present results using knowledge that the correct hidden dimension
$d_r=2$:  in \Cref{fig:A}, we show an example of the trained model being assimilated (using 3DVAR with learnt $K$) during the first 3 time units,
then predicting for another 7 time units; recall that training was performed
using only a $\tau_2=0.1$ forecasting horizon, but we evaluate on a longer
time horizon to provide a robust test metric. Observe that the learnt hidden dynamics in gray are synchronized with the data, then used to perform the forecast.
In \Cref{fig:B,fig:C}, we show that by solving the system for long time (here, $T=10^4$), we are able to accurately reproduce invariant statistics (invariant measure and autocorrelation, resp.) for the true system.
In \Cref{fig:D}, we show the evolution of the learnt $K$.

Next, we let $d_r=10$, exceeding the true dimension of the hidden states;
thus we are able to explore issues caused by learning an overly expressive
(in terms of dimension of hidden states) dynamical model.
\Cref{fig:l63dy10} shows dynamics for a learnt model in this setting;
we found its reproduction of invariant statistics to be similar to the cases in \Cref{fig:B,fig:C}, but omit the plots for brevity.
This success aligns with the approximation theory, as
discussed in Remark \ref{rem:drdy},
and provides empirical reassurance that the methodology can behave well
in situations where the dimension of the hidden variable is unknown
and dimension $d_r$ used in learning exceeds its true dimension.
Nevertheless, we construct an example in \Cref{ssec:irnn} in which a specific
embedding of the true dynamics in a system of higher dimension
can lead to poor approximation; this is caused by an instability in the
model which allows departure from the invariant manifold on which the
true dynamics is accurately captured. However, we emphasize that this
phenomenon is not observed empirically in the experiment reported
here with $d_r=10.$
Nonetheless we also note expected decreases in efficiency caused by over-estimating
the dimension of the hidden variable, during both model training and testing;
thus determining the smallest choice for $d_r$, compatible with good
approximation, is important.
Recent research has addressed this challenge in the discrete-time setting
by applying manifold learning to a delay-embedding space, then using the learnt
manifold to inform initialization and dimensionality of LSTM hidden states \citep{kemethInitializingLSTMInternal2021}.

Note that our early attempts at achieving these numerical results,
using the optimization ideas in \Cref{sssec:H,sssec:W},  yielded unstable models
that exhibited blow-up on shorter time scales (e.g. $T<1000$);
however, by incorporating data assimilation
as in \cite{chenAutodifferentiableEnsembleKalman2021},
and further tuning the optimization to achieve lower training errors, we were able to obtain a model that, empirically, did not exhibit blow-up, even when solved for very long time (e.g. $T=10^5$).
We also note that we were unable to achieve such high-fidelity results using the methods of \citep{ouala_learning_2020} on neural networks with non-linear activation functions; this may be explained by noting that
\citet{ouala_learning_2020} achieved their results using linear,
identity-based activations, resulting in inference of polynomial models containing
the true L63 model.

\begin{figure}
\centering

\begin{subfigure}{0.8\textwidth}
  \includegraphics[width=\textwidth]{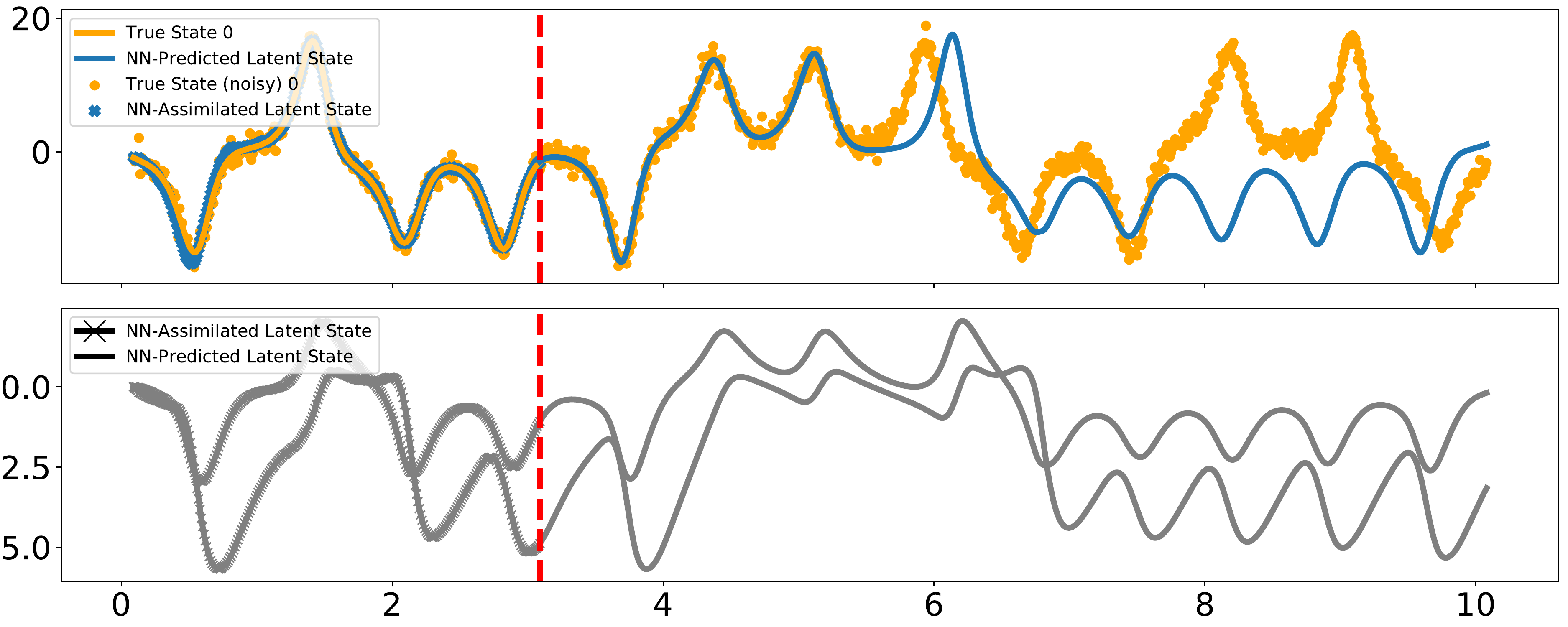}
  \caption{}
  \label{fig:A}
\end{subfigure}\\

\begin{subfigure}{0.31\textwidth}
  \includegraphics[width=\textwidth]{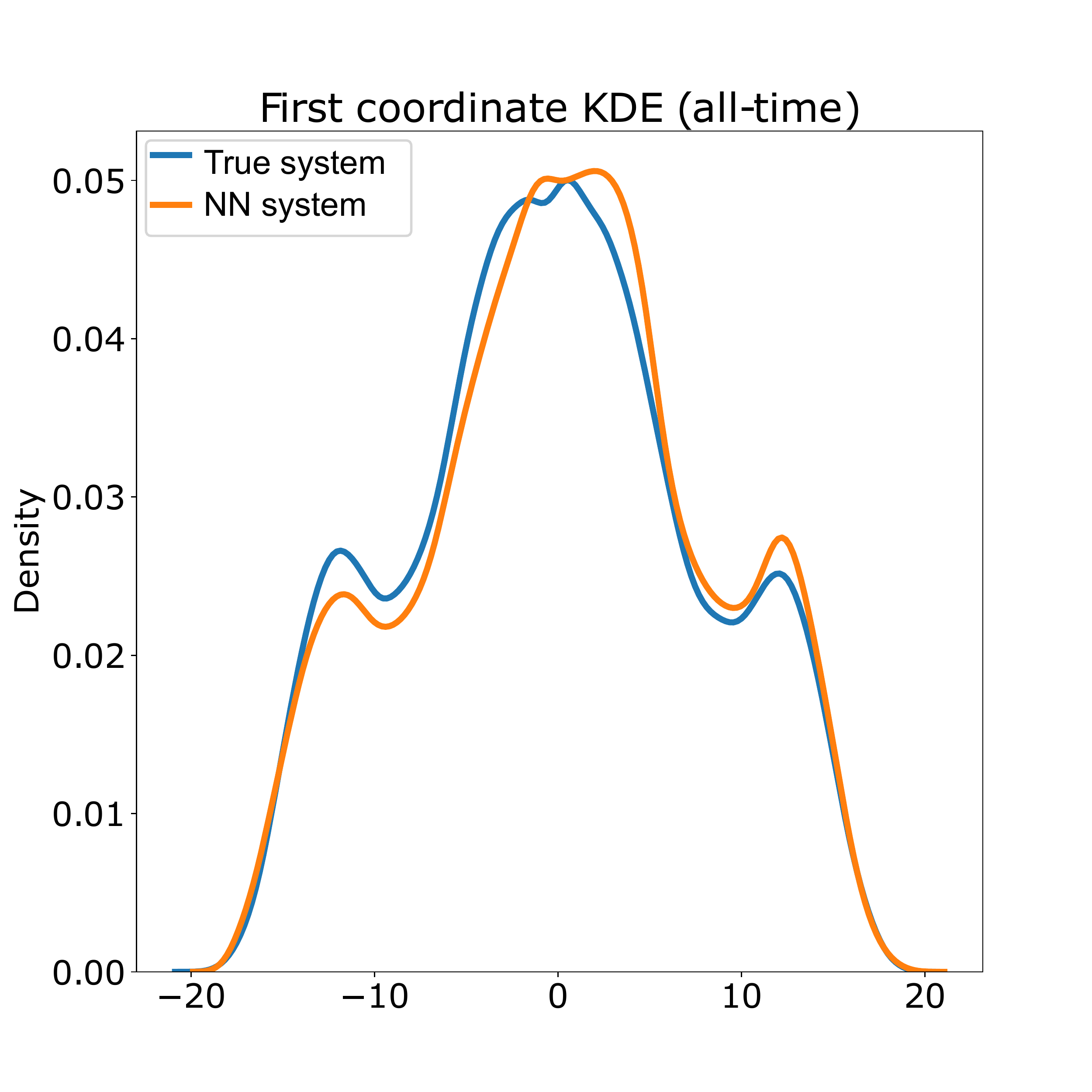}
  \caption{}
  \label{fig:B}
\end{subfigure} \hfill 
\begin{subfigure}{0.31\textwidth}
  \includegraphics[width=\textwidth]{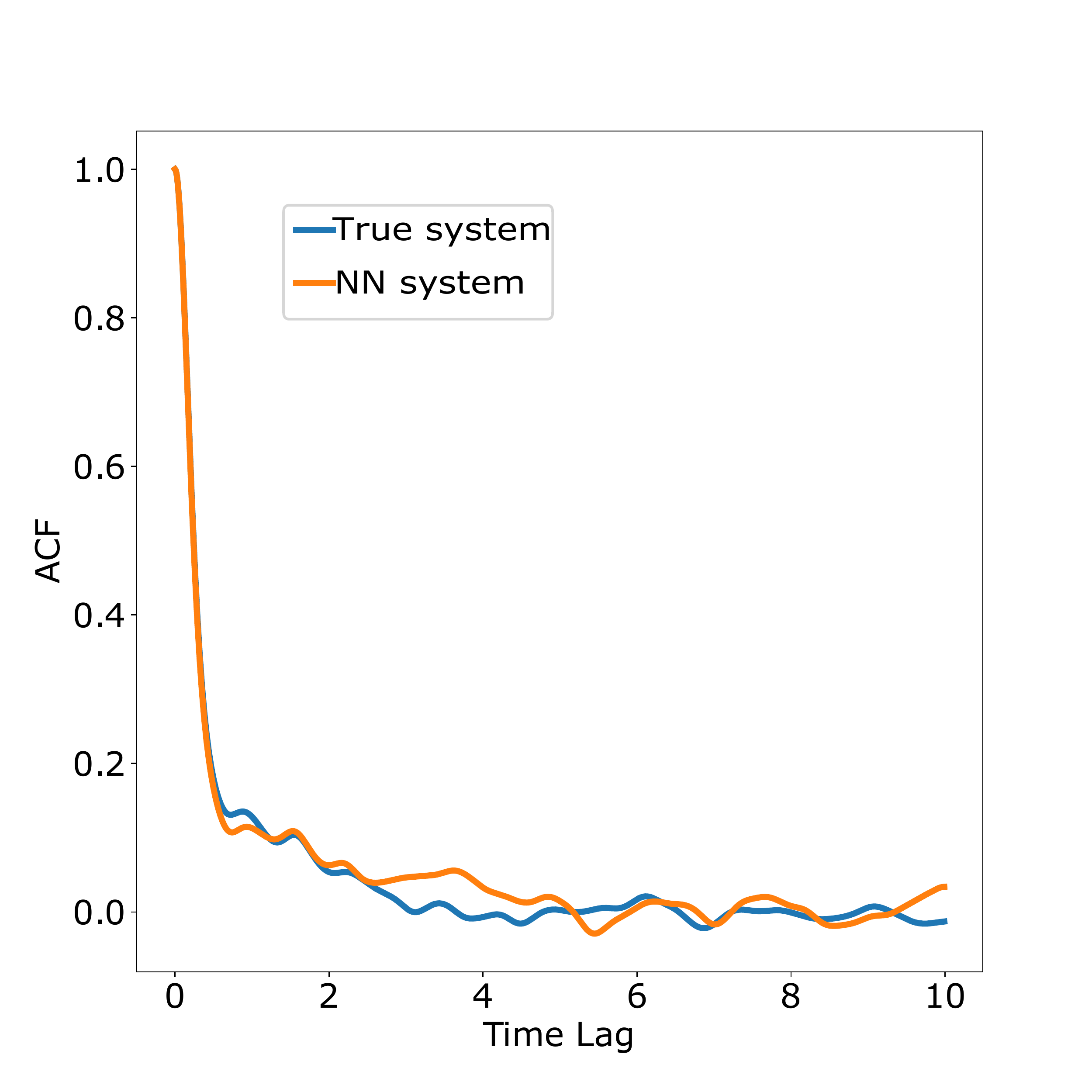}
  \caption{}
  \label{fig:C}
\end{subfigure} \hfill
\begin{subfigure}{0.31\textwidth}
  \includegraphics[width=\textwidth]{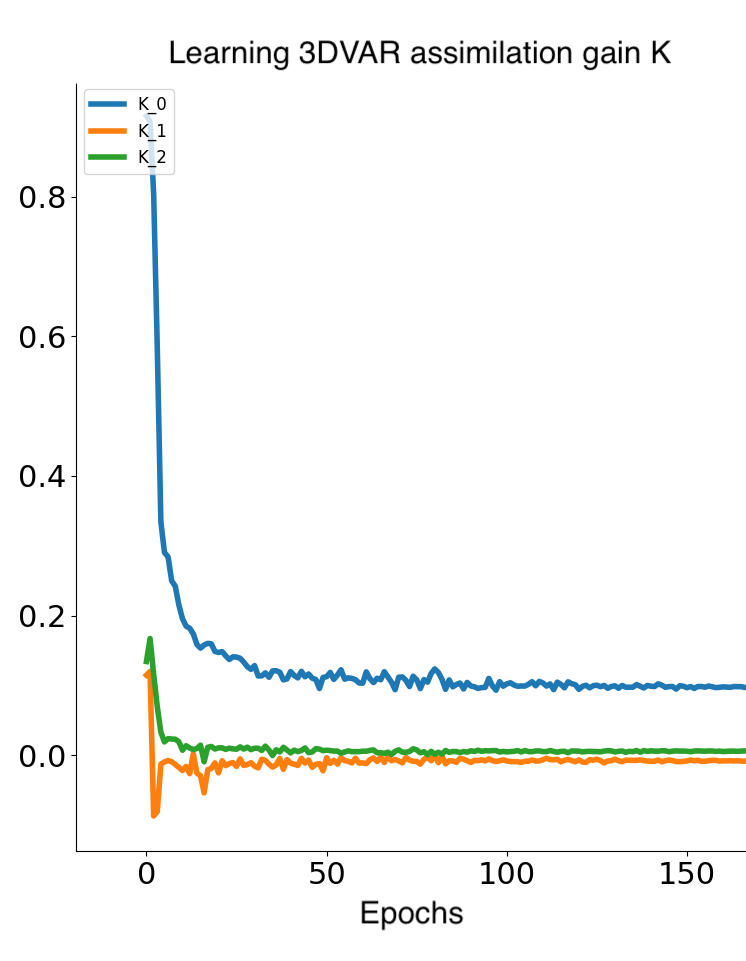}
  \caption{}
  \label{fig:D}
\end{subfigure}

\caption{
This figure concerns learning of a continuous-time RNN to model the L63 system \eqref{eq:l63}, based on noisy observation of only the first component; it uses an augmented state space $d_r = 2$.
\Cref{fig:A} shows how the trained model can be used for forecasting---by first synchronizing to data using 3DVAR, then forecasting into the future.
The top-half depicts dynamics of the observed component (model-solutions in blue; observations in yellow);
the bottom-half depicts the augmented state space (both hidden components are shown in gray).
We observed a validity time of roughly 3 model time units.
\Cref{fig:B,fig:C} shows that long-time solutions of the learnt model accurately mirror invariant statistics (invariant measure and autocorrelation, resp.) for the true system.
\Cref{fig:D} shows the learning process for estimating a 3DVAR gain $K$.}
\label{fig:l63dy2}
\end{figure}

\begin{figure}[!htbp]
\centering
  \includegraphics[width=0.8\textwidth]{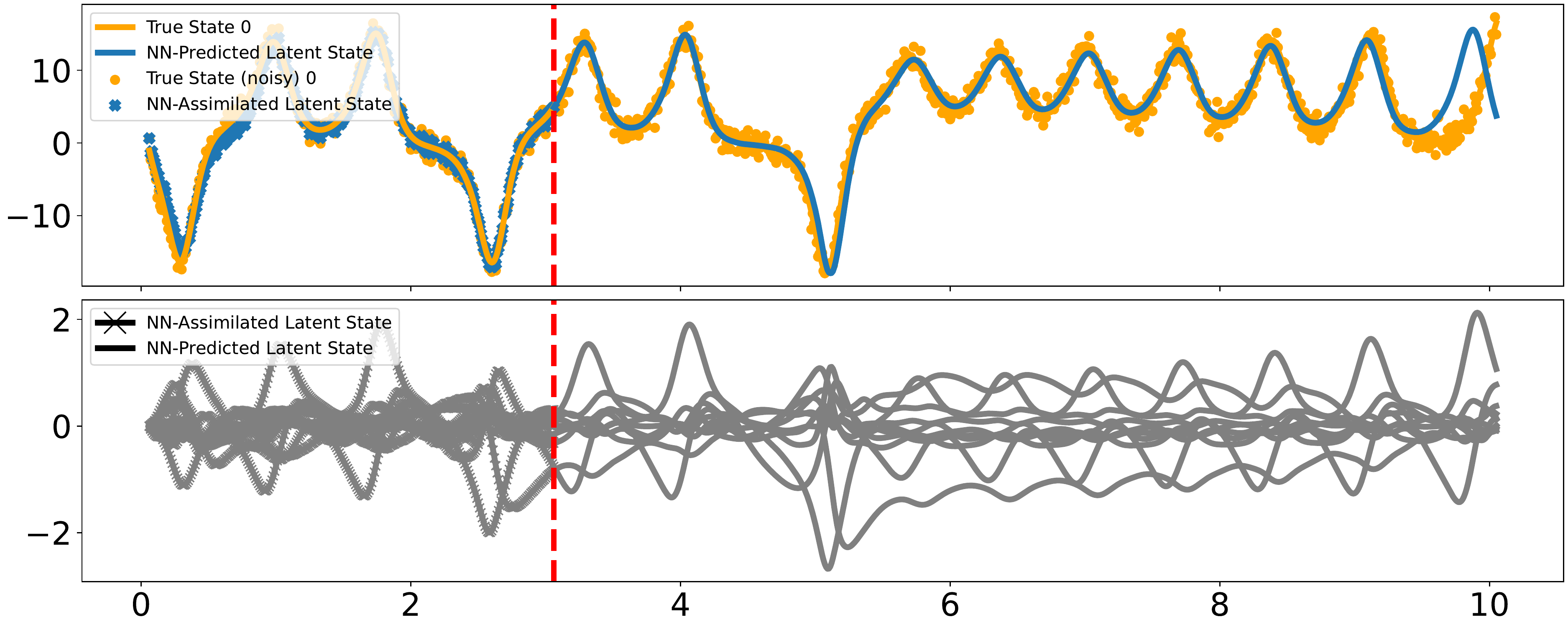}
  \caption{
  This figure concerns learning of a continuous-time RNN to model the L63 system \eqref{eq:l63}, based on noisy observation of only the first component; it uses an augmented state space $d_r = 10$.
  The top-half depicts dynamics of the observed component (model-solutions in blue; observations in yellow);
  the bottom-half depicts the augmented state space (all 10 hidden components are shown in gray).
  In the first 3 time units, the model is assimilated to a sequence of observed data using 3DVAR, then in the subsequent 7 time units, a forecast is issued. We found this model to have similar short-term and long-term fidelity when compared to the model presented in \Cref{fig:A,fig:B,fig:C,fig:D}, which used the correct hidden dimension $d_r=2$.}
\label{fig:l63dy10}
\end{figure}

\subsubsection{Lorenz '96 Multiscale ($\scaleeps=2^{-1}$)}
Recall that Markovian closures fail to capture autocorrelation statistics for the slow components of this model in the case of $\scaleeps=2^{-1}$ (see top right panel of \Cref{fig:l96acf}).
As evidenced by the slow-moving trajectory around the
Markovian closure in \Cref{fig:l96residuals},
this is a case ripe for non-Markovian modeling.
We investigate the applicability of our continuous-time ODE formulation in \eqref{eq:ml_cont_mem}, using a neural network of width of $1000$.
We applied the above described  methodology for minimizing \eqref{eq:JcmemMatt},
under the data setting described in \Cref{ssec:l96}, to learn hidden dynamics.
Similarly to the previous section, we jointly trained on 100 trajectories, each of length $T=20$ and randomly initialized from a box around the attractor.
We chose an assimilation time of $\tau_1=2$ and a forecast time of $\tau_2=1$; note that longer times can become quite costly, especially for high-dimensional systems; nevertheless, the assimilation time $\tau_1$ appears intrinsically tied to the amount of memory present in the system.

In \Cref{fig:B96,fig:C96}, we plot comparisons of the true and learnt (via \eqref{eq:ml_cont_mem}) ACF and invariant measure,
and observe substantial improvement over the Markovian closure.
However, this learnt model exhibited instabilities when solved for longer than $T=500$.
We expect that this can be remedied via further training (as was found for the L63 example);
however, the incorporation of stability constraints into the model,
as in \cite{schneider2021learning}, would be valuable.
In order to train this larger model for longer time, further studies of efficient optimization must also be performed in this setting (\citep{chenAutodifferentiableEnsembleKalman2021} has begun highly relevant investigation in this direction).

In \Cref{fig:K96}, we visualize the learnt 3DVAR gain (which encodes the learnt model's covariance structure), in which each row corresponds to the gain for a given component of the learnt model as a function of observed components (indexed in the columns); trends are elucidated via hierarchical clustering and a row-based normalization of the learnt matrix $K$.
It clearly learns a consistent diagonal covariance structure for the observables.
More impressively, it illustrates cross-covariances between observed and hidden components that mirror the compartmentalized structure of the model in \eqref{eq:l96ms};
note that each observed component has a distinct grouping of hidden variables which have high correlation (white) primarily with that component and low correlation (black) with other observables.
This type of analysis may provide greater interpretability of learnt models of hidden dynamics.

\begin{figure}
\centering


\begin{subfigure}{0.45\textwidth}
  \includegraphics[width=\textwidth]{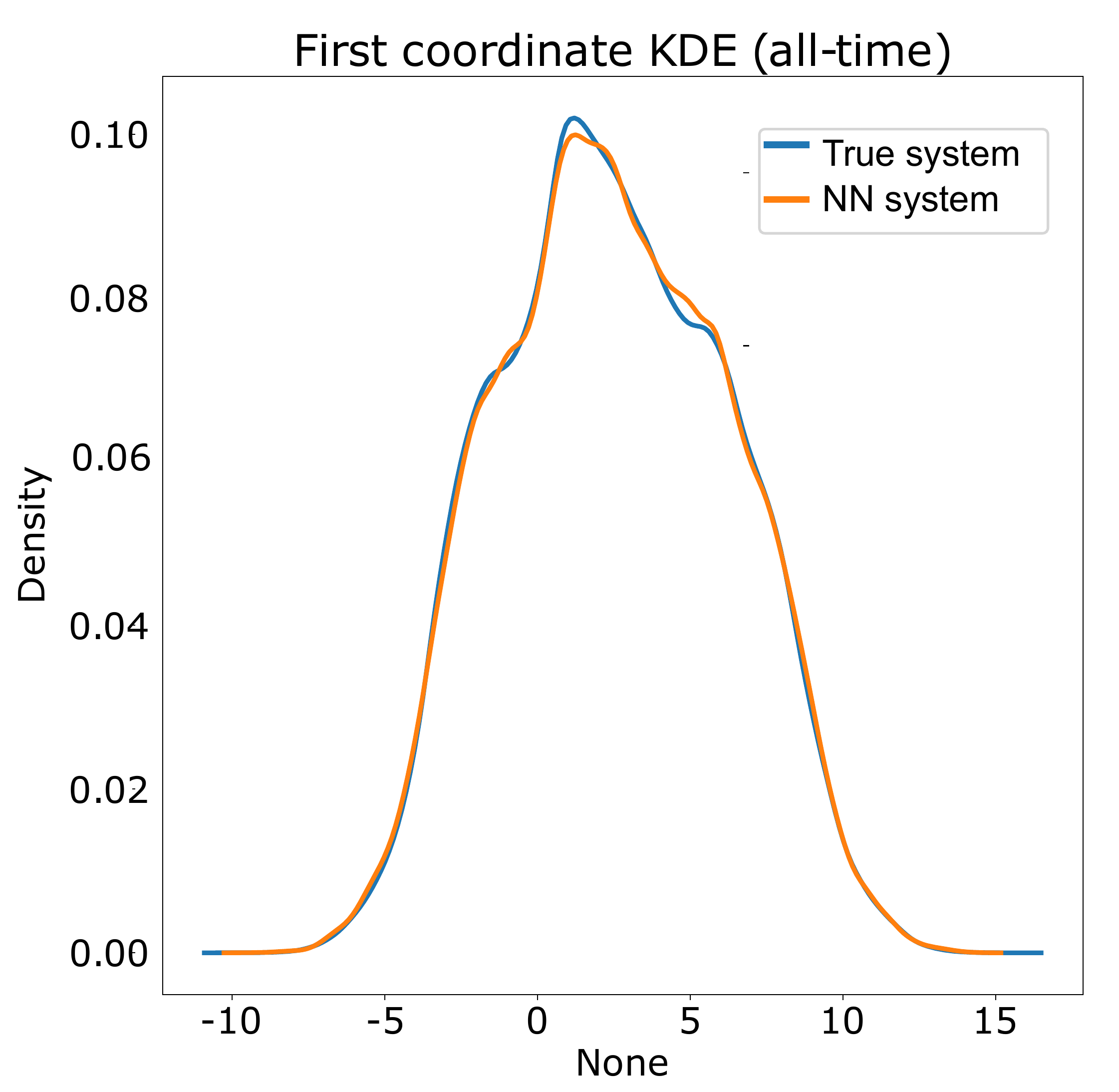}
  \caption{}
  \label{fig:B96}
\end{subfigure} \hfill 
\begin{subfigure}{0.45\textwidth}
  \includegraphics[width=\textwidth]{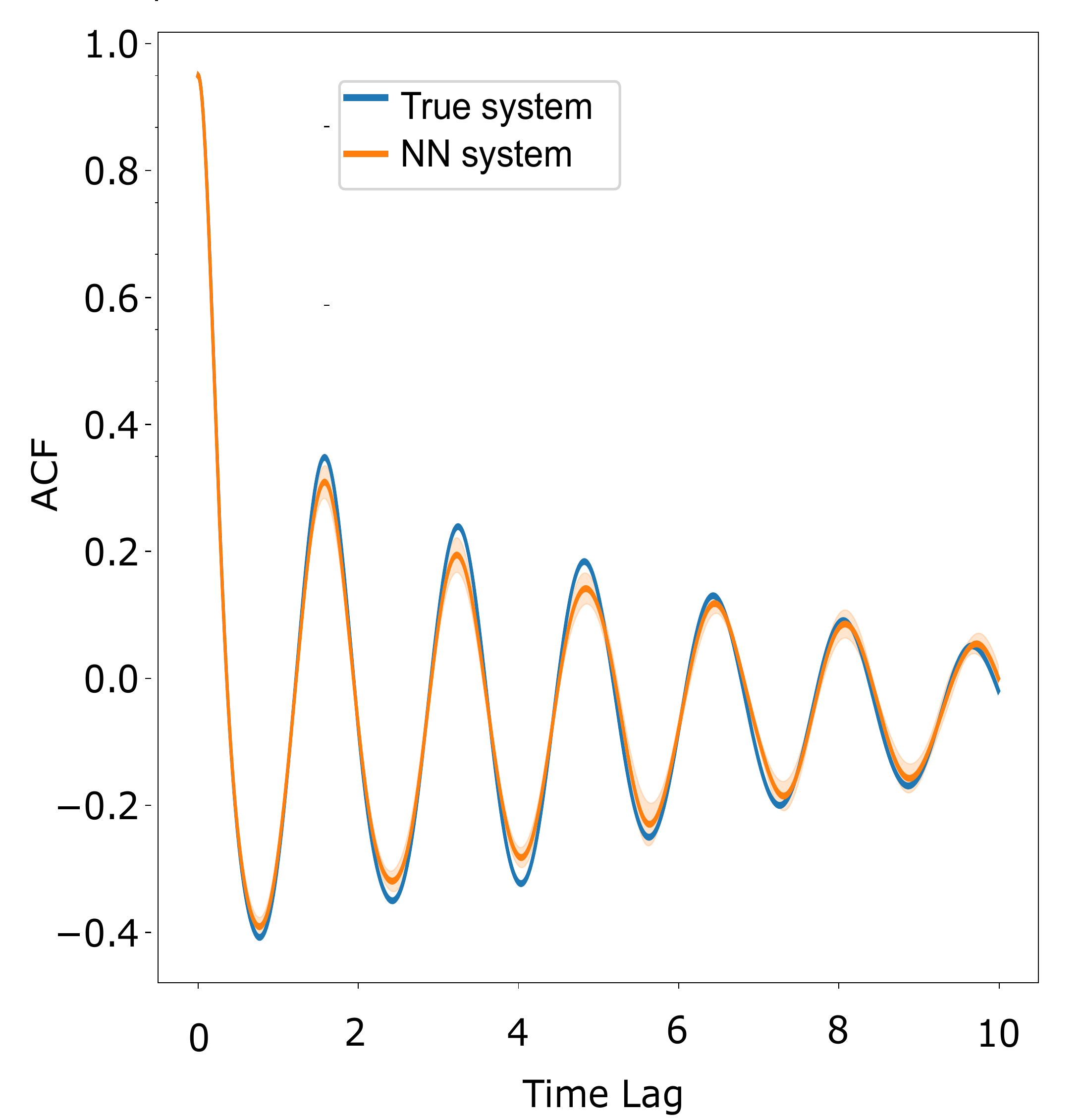}
  \caption{}
  \label{fig:C96}
\end{subfigure}

\caption{
This figure concerns learning of a continuous-time RNN to model the first 9 (slow) components of the L96MS system ($\scaleeps=0.5$ in \eqref{eq:l96ms}), based on noisy observations of these slow components; it uses an augmented state space $d_r = 72$.
We trained using noised observations (standard deviation $0.01$) of only the first 9 components of the true $81-$dimensional system.
These plots show that this model can accurately reproduce both the invariant measure (\Cref{fig:B96}) and ACF (\Cref{fig:C96}) for these observed states.
These statistics were calculated by running the learnt model for $T=500$ model time units;
longer runs encountered instabilities that caused trajectories to leave the attractor and blow-up.}
\label{fig:l96dy72}
\end{figure}

\begin{figure}
\centering

  \includegraphics[width=0.8\textwidth]{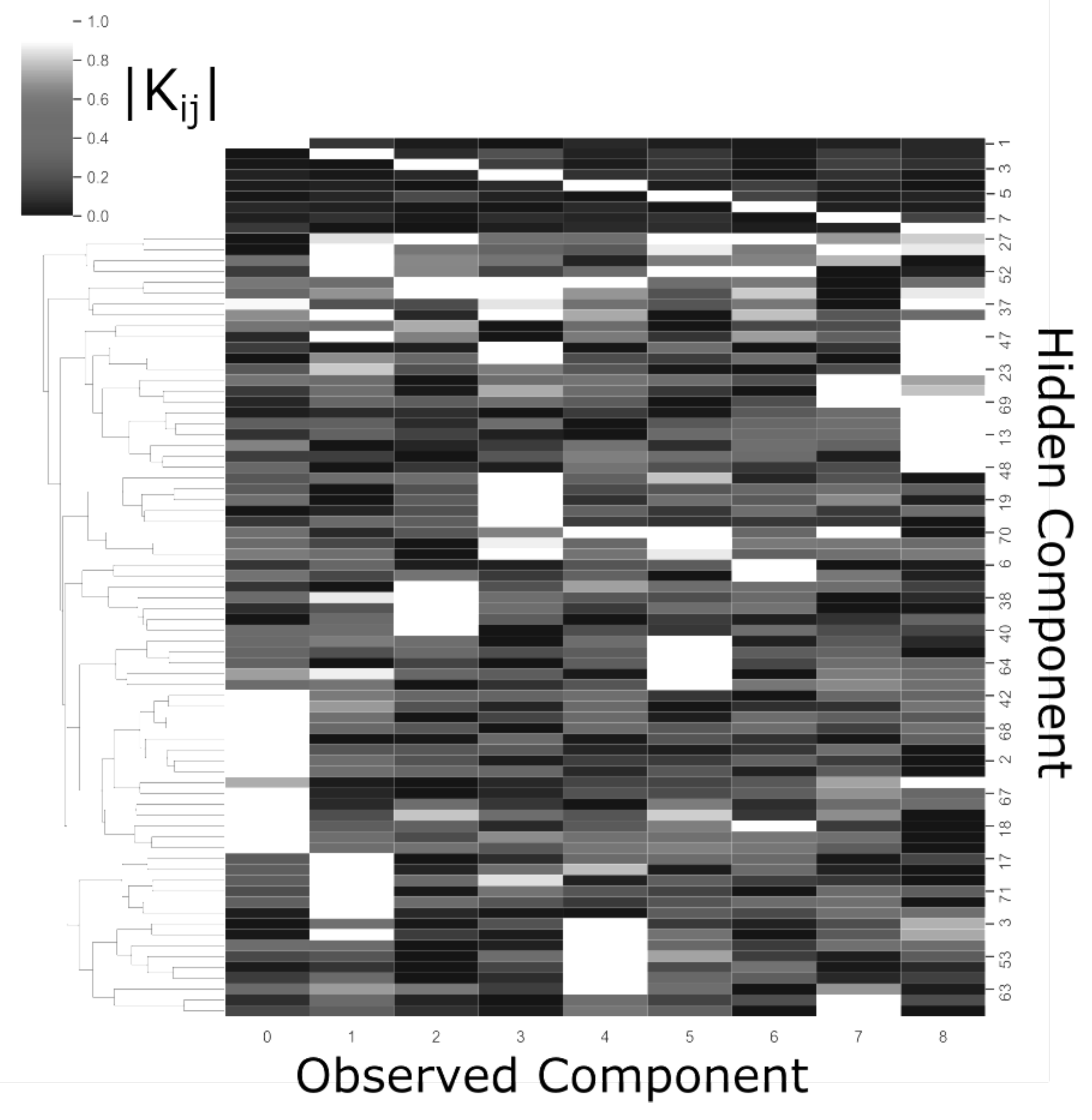}

\caption{Here we visualize the learnt 3DVAR gain matrix $K$ (81 x 9) ($\theta_\textrm{DA}$ in \Cref{ssec:mlda}) associated with the non-Markovian learning of L96MS \eqref{eq:l96ms}.
We first compute entry-wise absolute values, then apply a row-normalization;
white indicates highest correlation, and black indicates lowest correlation.
The top 9 rows shown directly correspond to the first 9 rows of $K$.
The bottom 72 rows are re-ordered (via hierarchical clustering) to illustrate associations between the 9 observed components and the 72 hidden variables.}
\label{fig:K96}
\end{figure}

\subsection{Initializing and Stabilizing the RNN}
\label{ssec:irnn}

As mentioned in Remark \ref{rem:da}
the RNN approximates an enlarged
system which contains solutions of the original system as
trajectories confined to the invariant manifold $m=m^\dag(x,y)$;
see identity \eqref{eq:identity}.
However, this invariant manifold may be unstable, either
as a manifold within the continuous-time model \eqref{eq:fgmdag2},
or as a result of numerical instability.
We now demonstrate this with numerical experiments.
This instability points to the need
for data assimilation to be used with RNNs if prediction of
the original system is desired, not only to initialize the system
but also to stabilize the dynamics to remain near to
the desired invariant manifold.

To illustrate these challenges, we consider the problem of modeling evolution of
a single component of the L63 system \eqref{eq:l63}. Consider
this as variable $x$ in \eqref{eq:fgdag}.
As exhibited in \eqref{eq:fgmdag2}, model error may be
addressed in this setting by learning
a representation that contains the hidden states $y$
in \eqref{eq:fgdag} (i.e. the other two unobserved components of \eqref{eq:l63}), but since the dimension of the hidden states is typically not
known {\em a priori} the dimensions of the latent variables in the RNN
(and the system it approximates)
may be greater than those of $y$; in the specific construction we use
to prove the existence of an approximating RNN we introduce a
vector field for evolution of the error $m$ as well as $y$. We now
discuss the implications of embedding the true dynamics in a higher dimensional
system in the specific context of the embedded
system \eqref{eq:fgmdag2}. However the observations apply to any
embedding of the desired dynamics \eqref{eq:fgdag}
(with $\epsilon=1$) within any higher dimensional system.

We choose examples for which \eqref{eq:identity} implies that $m-m^\dag$ is
constant in time. Then, under \eqref{eq:fgmdag2},
$$\bigl(m-m^\dag(x,y)\bigr)(t)={\rm constant};$$
that is, it is constant in time. The
desired invariant manifold (where the constant is $0$) is thus stable.
However this stability only holds in a neutral sense:
linearization about the manifold exhibits a zero eigenvalue
related to translation of $m-m^\dag$ by a constant.
We now illustrate that this embedded invariant manifold can be
unstable; in this case the instability is caused by numerical
integration, which breaks the conservation of $m-m^\dag$
in time.

\emph{Example 1:}
Consider equation \eqref{eq:l63} which we write in form \eqref{eq:fgmdag}
by setting $x=u_x$ and $y=(u_y, u_z).$
Then we let $f_0(u_x) := -a u_x$
yielding $m^\dag(u_y) = a u_y$. Thus $f^\dag=f_0+m^\dag$ is defined by
the first component of the right-hand side of \eqref{eq:l63}. The function
$g^\dag(u_y, u_z)$ is then given by the second and third components of
the right-hand side of \eqref{eq:l63}.
Applying the methodology leading to \eqref{eq:fgmdag2} to \eqref{eq:l63} results in the
following four dimensional system:
\begin{subequations}
  \label{eq:l63_4da}
\begin{align}
\dot{u}_x & = f_0(u_x) + m,  &u_x(0) &= x_0, \\
\dot{u}_y &= bu_x - u_y - u_xu_z,  &u_y(0) &= y_0, \\
\dot{u}_z &= -cu_z + u_xu_y,  &u_z(0) &= z_0, \\
\dot{m} &= a \bigl(bu_x - u_y - u_xu_z\bigr),  &m(0)&= m^\dag(y_0).
\end{align}
\end{subequations}

Here we have omitted the $u_y$-dependence from the equation \eqref{eq:l63} for $u_x$,
and aim to learn this error term;
we introduce the variable $m$ in order to do so.
This system, when projected into
$u_x, u_y, u_z$, behaves identically to \eqref{eq:l63}
when $m(0) = m^\dag(y_0)$.
Thus the $4-$dimensional system in \eqref{eq:l63_4da} has an embedded invariant manifold
on which the dynamics is coincident with that of the $3-$dimensional
L63 system.

We numerically integrate the $4-$dimensional system in \eqref{eq:l63_4da} for 10000 model time units (initialized at $x_0 = 1, y_0=3, z_0=1, m_0=a y_0=30$),
and show in \Cref{fig:da_l63_ay_inv} that the resulting measure for $u_x$ (dashed red) is nearly identical
to its invariant measure in the traditional $3-$dimensional L63 system in \eqref{eq:l63} (solid black).
However, we re-run the simulation for a perturbed $m(0) = m^\dag(y_0) + 1$, and see in \Cref{fig:da_l63_ay_inv} (dotted blue)
that this yields a different invariant measure for $u_x$.
This result emphasizes the importance of correctly initializing an RNN not only for efficient trajectory forecasting, but also for accurate
statistical representation of long-time behavior.

\begin{figure}[!htbp]
\centering
  \includegraphics[width=0.4\textwidth]{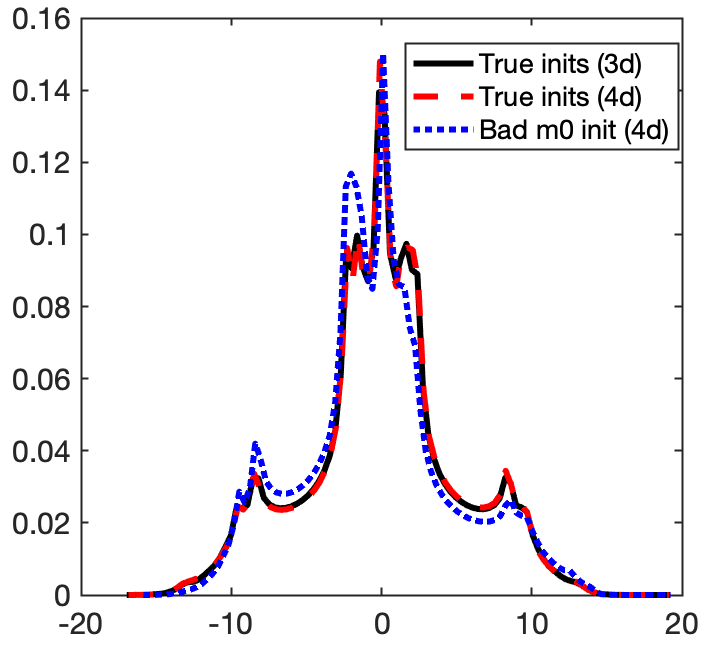}
 \caption{Here, we show that the invariant density for the first component of L63 (black)
can be reproduced by a correctly initialized augmented $4-$d system (dashed red) in \eqref{eq:l63_4da}.
However, incorrect initialization of $m(0)$ in \eqref{eq:l63_4da} (dotted blue)
yields a different invariant density.
 }
\label{fig:da_l63_ay_inv}
\end{figure}

\emph{Example 2:}
Now we consider \eqref{eq:l63} which we write in form \eqref{eq:fgmdag}
by setting $x=u_z$ and $y=(u_x, u_y).$
We let $f_0(u_z):= -c u_z$ and $m^\dag(u_x, u_y):= u_x u_y$, so that $f^\dag = f_0 + m^\dag$
corresponds to the third component of the right-hand side of \eqref{eq:l63}.
Function $g^\dag(u_x,u_y)$ is defined by the first two components
of the right-hand side of \eqref{eq:l63}.
We again form a $4-$dimensional system corresponding to \eqref{eq:l63}
using the methodology that leads to \eqref{eq:fgmdag2}:
\begin{subequations}
  \label{eq:l63_4db}
\begin{align}
\dot{u}_x & = a(u_y - u_x),  &u_x(0) &= x_0, \\
\dot{u}_y &= bu_x - u_y - u_xu_z,  &u_y(0) &= y_0, \\
\dot{u}_z &= f_0(u_z) + m,  &u_z(0) &= z_0, \\ 
\dot{m} &= u_x \dot{u}_y + u_y \dot{u}_x,  &m(0)&= m^\dag(x_0, y_0).
\end{align}
\end{subequations}
We integrate \eqref{eq:l63_4db} for 3000 model time units (initialized at $x_0 = 1, y_0=3, z_0=1, m_0=x_0 y_0 =3$),
and show in \Cref{fig:da_l63_ay2_traj_long} that the $3-$dimensional
Lorenz attractor is unstable with respect to perturbations in the
numerical integration of the $4-$dimensional system.
The solutions for $u_x, u_y, u_z$ eventually collapse to a fixed point after the growing discrepancy between $m(t)$ and $m^\dag$
becomes too large. The time at which collapse occurs may be
delayed by using smaller tolerances within the numerical integrator
(we employ {\sc Matlab rk45}) demonstrating that the instability is
caused by the numerical integrator.
This collapse is very undesirable if prediction of long-time
statistics is a desirable goal. On the other hand,
\Cref{fig:da_l63_ay2_traj_short} shows short-term accuracy of the $4-$dimensional system in \eqref{eq:l63_4db} up to 12 model time
units when correctly initialized ($m_0 = m^\dag(x_0, y_0)$, dashed red), and accuracy up to 8 model time units when initialization of $m_0$ is perturbed ($m_0 = m^\dag(x_0, y_0) + 1$, dotted blue).
This result demonstrates the fundamental challenges of representing chaotic attractors in enlarged dimensions and may help explain
observations of RNNs yielding good short-term accuracy, but inaccurate long-term statistical behavior.
While empirical stability has been observed in some discrete-time LSTMs \citep{vlachas_backpropagation_2020,harlim_machine_2021},
the general problem illustrated above is likely to manifest in any problems
where the dimension of the learned model exceeds that of the true model;
the issue of how to address initialization of such models, and its interaction
with data assimilation, therefore merits further study.

\begin{figure}[!htbp]
\centering
  \includegraphics[width=0.7\textwidth]{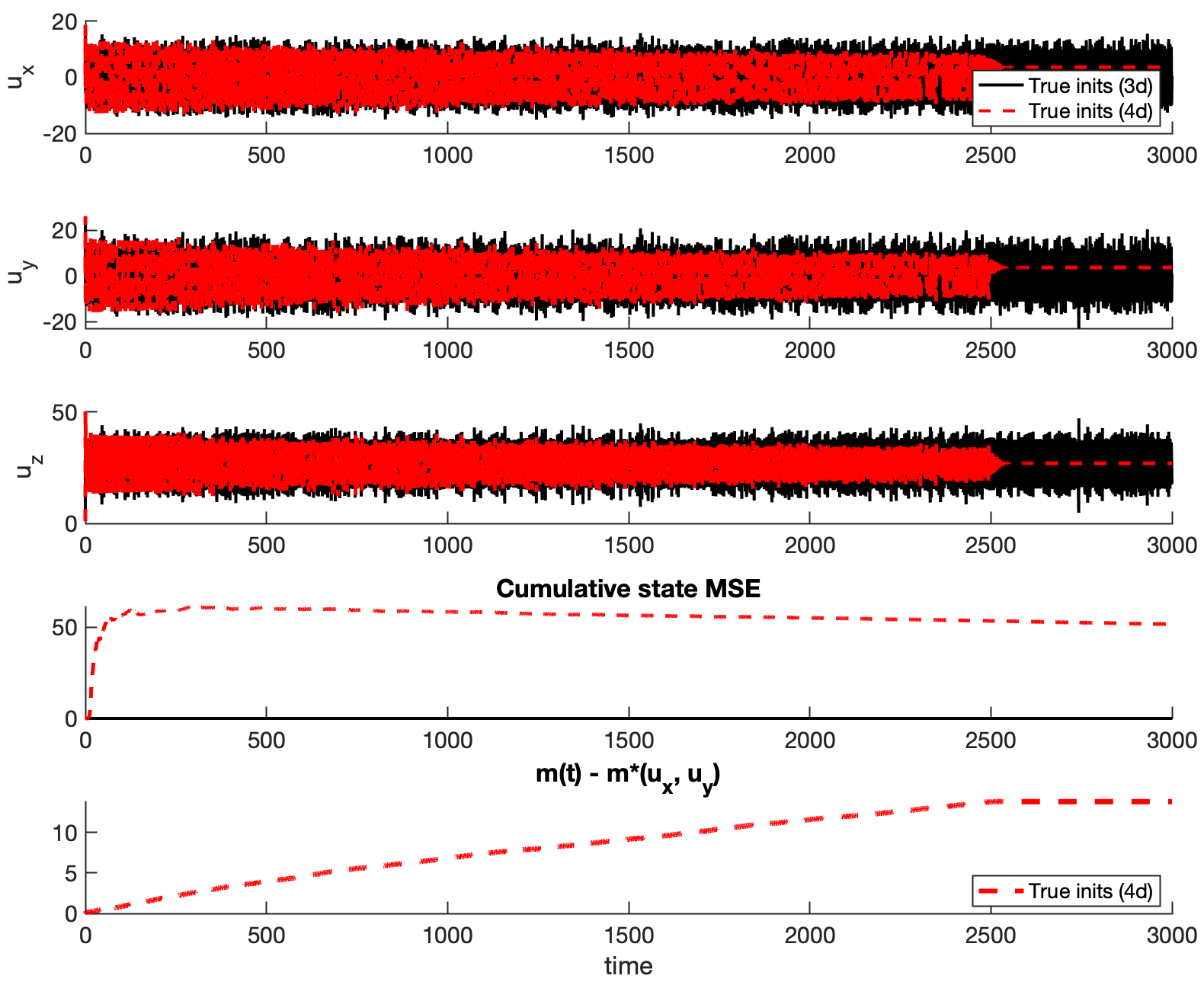}
  \caption{Here we show that the embedded $3-$dimensional manifold
of L63, within the $4-$dimensional system given by \eqref{eq:l63_4db}, is unstable. Indeed
the correctly initialized $4-$dimensional system (dashed red) has solution which decays to a fixed point. The bottom figure shows divergence of the numerically integrated model error term $m(t)$ and the state-dependent term $m^\dag$; this growing discrepancy is likely
responsible for the eventual collapse of the $4-$dimensional system.}
\label{fig:da_l63_ay2_traj_long}
\end{figure}

\begin{figure}[!htbp]
\centering
  \includegraphics[width=0.8\textwidth]{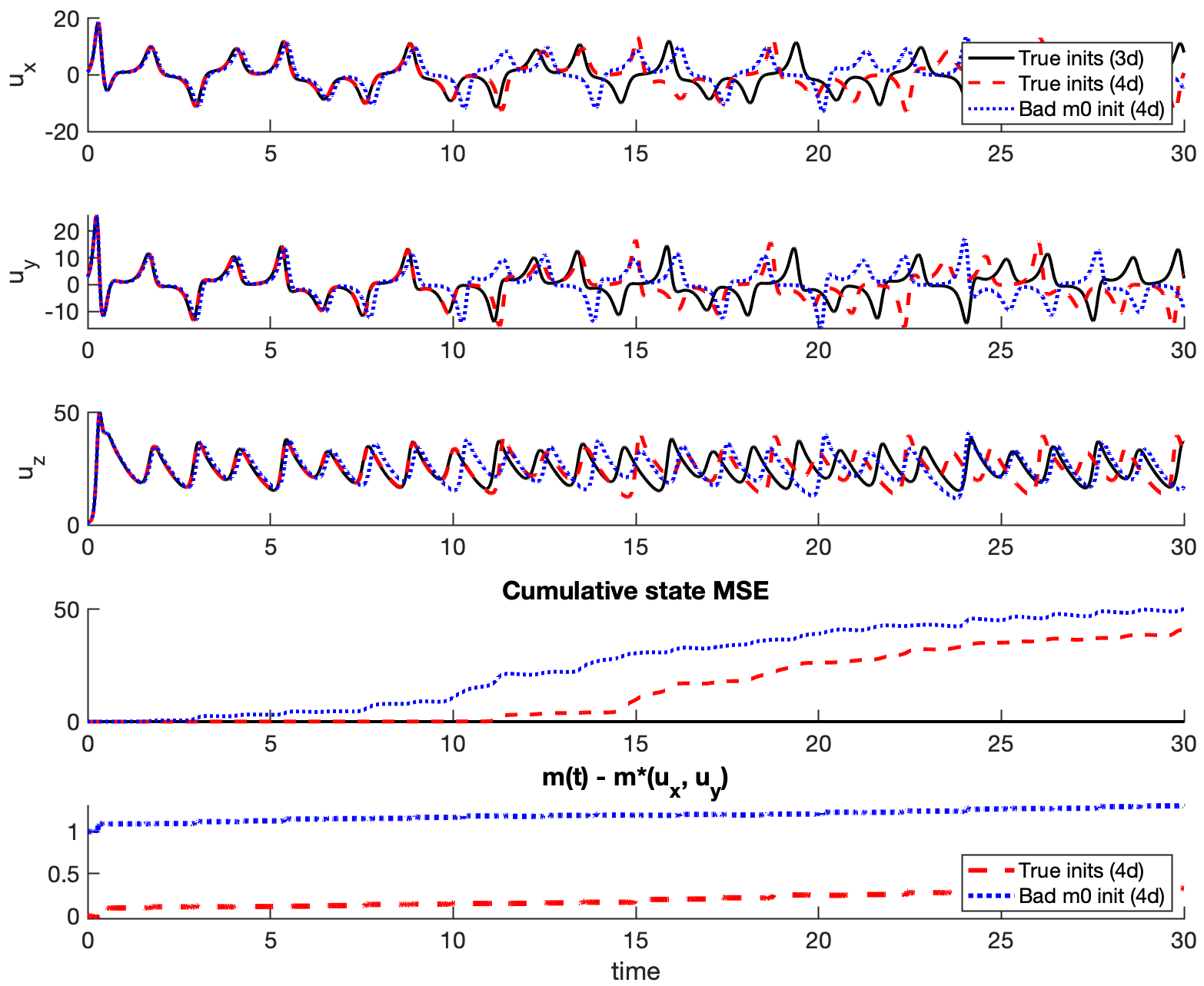}
  \caption{Here, we show short-term accuracy for the $4-$dimensional
system in \eqref{eq:l63_4db}.
Predictions using the correct initialization of $m_0$ (dashed red) remain accurate for nearly twice as long
as predictions that use a perturbed initialization ($m_0 = m^\dag(u_x, u_y) + 1$).
The bottom figure shows that $m(t)$ diverges from the state-dependent
$m^\dag$ more quickly for the poorly initialized model,
but in both cases errors accumulate over time.}
\label{fig:da_l63_ay2_traj_short}
\end{figure}

\section{Conclusions}
\label{sec:conclusions}

In this work we evaluate the utility of blending mechanistic models
of dynamical systems with data-driven methods,
demonstrating the power of hybrid approaches.
We provide a mathematical framework that
is consistent across parametric and non-parametric models,
encompasses both continuous- and discrete-time, and allows for
Markovian and memory-dependent model error. We also provide
basic theoretical results that underpin the adopted approaches.
The unified framework elucidates commonalities between seemingly
disparate approaches across various applied and theoretical disciplines.
It would be desirable if the growing recognition of the need
for hybrid modeling were to motivate flexible incorporation of
mechanistic models into open-source software for continuous-time
Markovian and non-Markovian modeling of error  \citep{ouala_learning_2020,rubanova_latent_2019,chang_antisymmetricrnn_2019,erichson_lipschitz_2020,anantharaman_accelerating_2020,gupta_neural_2021}.

Our work is focused on immutable mechanistic models
($f_0$ and $\Psi_0$),
but these models themselves often have tunable parameters.
In principle one can jointly learn parameters for the mechanistic model and closure term.
  However, the lack of identifiability between modifying the closure and modifying the physics
  brings up an interesting question in explainability.
  Future work might focus on decoupling the learning of parameters and closure terms
  so that maximal expressivity is first squeezed out of the mechanistic model \citep{plumlee_orthogonal_2017,plumlee_bayesian_2017}.

Our numerical results demonstrate
the superiority of hybrid modeling over learning an entire system from scratch, even when the available mechanistic model has large infidelities.
Hybrid modeling also showed surprisingly large performance gains over using mechanistic models with only small infidelities.
We quantify these improvements in terms of data hunger, demands for model complexity, and overall predictive performance, and find that all three are significantly improved by hybrid methods in our experiments.

We establish bounds on the excess risk and generalization error that decay as $1 / \sqrt{T}$
when learning model discrepancy from a trajectory of length $T$ in an ergodic continuous-time Markovian setting.
We make minimal assumptions about the nominal physics (i.e. $f_0 \in C^1$); thus, our result equivalently holds for learning the entire vector field $f^\dag$ (i.e $f_0 \equiv 0$).
However the upper bounds on excess risk and generalization error
scale with the size of the function being learned, hence
going some way towards explaining the superiority of hybrid modeling
observed in the numerical experiments.
Future theoretical work aimed at quantifying the
benefits of hybrid learning versus purely
data-driven learning is of interest.
We also note that the ergodic assumption underlying our theory
will not be satisfied by many dynamical models, and alternate
statistical learning theories need to be developed in such settings.

We illustrate trade-offs between discrete-time and continuous-time modeling approaches by studying their performance as a function of training data sample rate.
We find that hybrid discrete-time approaches can alleviate instabilities seen in purely data-driven discrete-time models at small timesteps; this is likely due to structure in the integrator $\Psi_0$, which has the correct parametric dependence on timestep.
In the continuous-time setting, we find that performance is best when derivatives can accurately be reconstructed from the data, and deteriorates
in tandem with differentiation inaccuracies (caused by large timesteps); continuous-time hybrid methods appear to offer additional robustness to inaccurate differentiation when compared to purely data-driven methods.
In cases of large timesteps and poorly resolved derivatives, ensemble-based data assimilation methods may still allow for accurate learning of  residuals to the flow field for continuous-time modeling \citep{gottwald_supervised_2021}.

Finally, we study non-Markovian memory-dependent model error, through
numerical experiments and theory, using RNNs.
We prove universal approximation for continuous-time hybrid RNNs and
demonstrate successful deployment of the methodology.
Future work focusing on the effective training of these models, for more
complex problems, would be of great value; ideas from data assimilation
are likely to play a central role \cite{chenAutodifferentiableEnsembleKalman2021}.
Further work on theoretical properties of reservoir computing (RC) variants on RNNs
would also be of value: they
benefit from convex optimization, and may be viewed
as random feature methods between Banach spaces.
These RNN and RC methods will
benefit from constraining the learning to ensure stability of the latent
dynamical model. These issues are illustrated via numerical experiments
that relate RNNs to the question of stability of
invariant manifolds representing embedded desired dynamics within
a higher dimensional system.

\section{Appendix}
\label{sec:appendix}

\subsection{Proof of Excess Risk/Generalization Error Theorem}
\label{sec:riskyproof}

Note that in both \eqref{eq:CLT} and \eqref{eq:LIL} $\varphi(\cdot)$ is
only evaluated on (compact) $\cA$ obviating the need for any boundedness
assumptions on the functions $\{f_\ell\}_{\ell=0}^p$ and $m^\dag$ in what follows.

\begin{lemma}
\label{lem:1}
Let Assumptions \ref{asp:1} and \ref{asp:2} hold. Then
there is $\Sigma$ positive semi-definite symmetric in
$\mathbb{R}^{p \times p}$
such that $\tst \to \ts$ almost surely, and $\sqrt{T}(\tst-\ts)
\Rightarrow N(0,\Sigma)$ with respect to $x(0) \sim \mu.$
Furthermore, there is constant
$C \in (0,\infty)$ such that, almost surely w.r.t. $x(0) \sim \mu$,
$${\rm limsup}_{T \to \infty}\Bigl(\frac{T}{\log\log T}\Bigr)^{\frac12}
\|\tst-\ts\| \le C.$$
\end{lemma}

\begin{proof}
By rearranging the equation for $\ts$ we see that
\begin{align*}
\AT \tst & = \bT,\\
\AT \ts &= \bi +(\AT-\Ai)\ts.
\end{align*}
Thus, subtracting,
\begin{equation}
\label{eq:errorth}
(\tst-\ts)=\AT^{-1}(\bT-\bi) - \AT^{-1}(\AT-\Ai)\ts.
\end{equation}

Because $\{f_\ell(\cdot)\}$ and $m^{\dag}(\cdot)$ are H\"older
(Assumption \ref{asp:2}, and discussion immediately preceding it), so are
$\bigl\langle f_i(\cdot), f_j(\cdot) \bigr\rangle$ and
$\bigl\langle m^{\dag}(\cdot), f_j(\cdot) \bigr\rangle.$
Thus each entry of matrix $\AT$ (resp. vector $\bT$) converges
almost surely to its corresponding entry in $\Ai$ (resp. $\bi$),
by the ergodicity implied by Assumption \ref{asp:1}, and the
pointwise ergodic theorem.
The almost sure convergence of $\tst$ to $\ts$ follows, after noting
that $\Ai$ is invertible.
Furthermore, also by Assumption \ref{asp:1},
there are constants $\{\sigma_{ij}\}, \{\sigma_{j}\}$ such that
\begin{align*}
&\sqrt{T}\Bigl((\AT)_{ij}-(\Ai)_{ij}\Bigr) \Rightarrow N(0,\sigma_{ij}^2),\\
&\sqrt{T}\Bigl((\bT)_{j}-(\bi)_{j}\Bigr) \Rightarrow N(0,\sigma_{j}^2).
\end{align*}
Since arbitrary linear combinations of
the $\{(\AT)_{ij}\}, \{(\bT)_{j}\}$ are time-averages of
H\"older functions, it follows
that $\sqrt{T}\{\AT-\Ai,\bT-\bi\}$ converges in distribution to a Gaussian,
by the Cram\'er-Wold Theorem \citep{grimmett_probability_2020}.
Weak convergence of $\sqrt{T}(\tst-\ts)$ to a Gaussian
follows from \eqref{eq:errorth} by use of
the Slutsky Lemma \citep{grimmett_probability_2020},
since $A_T$ converges almost surely to invertible $A_{\infty}.$ Matrix $\Sigma$
cannot be identified explicitly in terms of only the $\{\sigma_{ij}\}, \{\sigma_{i}\}$
because of correlations between $A_T$ and $b_T$.
The almost sure bound on $\|\tst-\ts\|$ follows from \eqref{eq:errorth}
after multiplying by $(T/\log\log T)^{\frac12}$, noting
that $\AT \to \Ai$ almost surely, and the almost sure bounds
on $(T/\log\log T)^{\frac12}\{\|\AT-\Ai\|,\|\bT-\bi\|\}$,
using Assumption \ref{asp:1}.
\end{proof}

In what follows it is helpful to define
\begin{align*}
R_T^+&=(\tst-\ts)\bigl(\|\tst\|+\|\ts\|+1\bigr),\\
G_T^+&=\mathcal{I}_T(\ms_\infty)-\mathcal{I}_\infty(\ms_\infty).
\end{align*}

\begin{lemma}
\label{lem:2}
Let Assumption \ref{asp:2} hold. Then, assuming $x(0) \sim \mu$,
there is constant $C>0$
such that the excess risk $R_T$ satisfies
$$R_T \le C\|R_T^+\|.$$
Furthermore the generalization error satisfies
$$|G_T| \le 2C\|R_T^+\|+|G_T^+|.$$
\end{lemma}

\begin{proof}
For the bound on the excess risk we note that
\begin{align*}
R_T &= \mathcal{L}_\mu(\ms_T,m^\dag)-\mathcal{L}_\mu(\ms_\infty,m^\dag)\\
&=\int_{\R^{d_x}} \Bigl\langle (\ms_T-\ms_\infty)(x),(\ms_T+\ms_\infty-2m^\dag)(x)\Bigr\rangle\,\mu(dx)\\
&\le \Bigl( \int_{\R^{d_x}}\bigl\|(\ms_T-\ms_\infty)(x)\bigr\|^2\,\mu(dx) \Bigr)^{\frac12} \Bigl( \int_{\R^{d_x}}\bigl\|(\ms_T+\ms_\infty-2m^\dag)(x)\bigr\|^2\,\mu(dx) \Bigr)^{\frac12}.
\end{align*}
The first follows from the boundedness of the $\{f_\ell\}_{\ell=1}^{p}$ and $m^\dag$,
since the first term in the product above is bounded by
a constant multiple of $\|\tst-\ts\|$ and the second term
by a constant multiple of $\|\tst\|+\|\ts\|+\sup_{\mathcal{A}} \|m^\dag\|.$

For the bound on the generalization error we note that
\begin{align*}
G_T&=\mathcal{I}_T(\ms_T)-\mathcal{I}_\infty(\ms_T)\\
&=\mathcal{I}_T(\ms_T)-\mathcal{I}_T(\ms_\infty)\\
&\quad\quad\quad\quad\quad +\mathcal{I}_T(\ms_\infty)-\mathcal{I}_\infty(\ms_\infty)\\
&\quad\quad\quad\quad\quad\quad\quad\quad\quad \ \ \ +\mathcal{I}_\infty(\ms_\infty)-\mathcal{I}_\infty(\ms_T)\\
&=\bigl(\mathcal{I}_T(\ms_T)-\mathcal{I}_T(\ms_\infty)\bigr) + G_T^+ - R_T.
\end{align*}
The third term in the final identity is the excess risk that we
have just bounded; the first term may be bounded
in the same manner that we bounded the excess risk, noting that
integration with respect to $\mu$ is simply replaced by integration with
respect to the empirical measure generated by the trajectory data which,
by assumption, is confined to the attractor $\mathcal{A}$;
the second term is simply $G_T^+.$ Thus the result follows.
\end{proof}

\begin{proof}[Proof of Theorem \ref{t:erge}]
By Assumption \ref{asp:1},
with choice of $\varphi(x) = \|m^\dag(x) - \ms_\infty(x)\|^2,$
 $\sqrt{T}G_T^+$ converges in distribution to a
scalar-valued centred Gaussian.
By Lemma \ref{lem:1} and the Slutsky
Lemma \citep{grimmett_probability_2020}, $\sqrt{T}R_T^+$ converges in distribution to a
centred Gaussian in $\mathbb{R}^p.$ By the Cramer-Wold Theorem \citep{grimmett_probability_2020}
$\sqrt{T}(R_T^+,G_T^+)$ converges in distribution to a
centred Gaussian in $\mathbb{R}^{p+1}.$

The convergence in distribution results for excess risk $R_T$ and generalization error $|G_T|$ then follow from Lemma \ref{lem:2}, under
Assumption \ref{asp:1}. Furthermore, by Lemma \ref{lem:1},
there is constant $C_1>0$ such that
$${\rm limsup}_{T \to \infty}\Bigl(\frac{T}{\log\log T}\Bigr)^{\frac12}
\|R_T^+\| \le C_1;$$
similarly, possibly by enlarging $C_1$, Assumption \ref{asp:1} gives
$${\rm limsup}_{T \to \infty}\Bigl(\frac{T}{\log\log T}\Bigr)^{\frac12}
|G_T^+| \le C_1.$$
The desired almost sure bound on $R_T +|G_T|$ follows from Lemma \ref{lem:2}.

\end{proof}

\subsection{Proof of Continuous-Time ODE Approximation Theorem (General Case)}
\label{sec:rnnproofGen}

\begin{proof}
Recall equations \eqref{eq:fgmdag2}.
By \ref{Aapprox}, for any $\delta_o > 0$
there exist dimensions $N_g$ and $N_m$ and parameterizations $\theta_g \in \R^{N_g}, \theta_m \in \R^{N_m}$
such that for any $\bigl(x,y\bigr) \in B(0, 2\rho_T)$,
and in the maximum norm,
\begin{equation*}
 \label{eq:ghapproxGen}
 \begin{aligned}
\| g^\dag(x,y) - f_2(x,y; \ \theta_g) \| &\leq \delta_o \\
\| m^\dag(x,y) - f_1(x,y; \ \theta_m) \| &\leq \delta_o.
\end{aligned}
\end{equation*}

By using these,
we can rewrite \eqref{eq:fgmdag2} as
\begin{equation}
 \begin{aligned}
 \label{eq:truewitherrorGen}
 \dot{x} &= f_0(x) + f_1(x,y; \ \theta_m) + e_x(t) \\
 \dot{y} &= f_2(x,y; \ \theta_g) + e_y(t) \\
 \end{aligned}
 \end{equation}
 where, uniformly for $\bigl(x(0),y(0)\bigr) \in B(0, \rho_0),$
\begin{equation*}
 \begin{aligned}
\sup_{t \in [0,T]} \| e_y(t) \| \leq \delta_o \\
\sup_{t \in [0,T]} \| e_x(t) \| \leq \delta_o.
\end{aligned}
\end{equation*}
By removing the bounded error terms, we obtain the approximate system:
\begin{equation}
 \begin{aligned}
 \label{eq:fullapproxGen}
 \dot{x}_\delta &= f_0(x_\delta) + f_1(x_\delta,y_\delta; \ \theta_m) \\
 \dot{y}_\delta &= f_2(x_\delta, y_\delta; \ \theta_g) \\
 \end{aligned}
 \end{equation}
Next, we obtain a stability bound on the discrepancy between the approximate system \eqref{eq:fullapproxGen} and the true system (originally written as \eqref{eq:fgmdag} and re-formulated as \eqref{eq:truewitherrorGen}).
First, let $w = (x, y)$, $w_\delta = (x_\delta, y_\delta)$ and define $F$ to be the concatenated right-hand-side of \eqref{eq:fullapproxGen}. Note that
$F$ is $L-$Lipschitz in the maximum norm on $B(0, 2\rho_T)$, for some $L$ related to the Lipschitz continuity of $f_0$, $f_1$, and $f_2$.
Then we can write the true and approximate systems, respectively, as
(using the maximum norm)
\begin{subequations}
\label{eq:ddaGen}
 \begin{align}
\dot{w} &= F(w)+e_w(t)\\
\dot{w}_\delta &= F(w_\delta),
 \end{align}
 \end{subequations}
where
$$\sup_{t \in [0,T]} \| e_w(t) \| \le \sup_{t \in [0,T]} \| e_y(t) \|
+\sup_{t \in [0,T]}\| e_x(t) \| \le 2\delta_o.$$
Let $Pw=(x,y)$.
Then, for any $t \in [0,T]$, and for all $Pw(0), Pw_{\delta}(0) \in
B(0,\rho_0)$

\begin{align*}
\| w(t) - w_\delta(t)  \|
&\leq \big \| w(0) - w_\delta(0)\big \|  + \int_0^t \big \| e_w(s)\big \| ds + \int_0^t \big \| F\big(w(s)\big) - F(w_\delta(s)\big)\big \| ds.
\end{align*}
This follows by writing \eqref{eq:ddaGen} in integrated form, subtracting
and taking norms.
Using the facts that $\big \| e_w(s) \big\| \leq 2 \delta_o$ and
$F$ is $L-$Lipschitz we obtain, for $t \in [0,T]$,
\begin{align*}
\big \| w(t) - w_\delta(t) \big \|  &\leq \big \| w(0) - w_\delta(0) \big \|  + 2 \delta_o T + L \int_0^t \big \| w(s) - w_\delta(s) \big \| ds.\\
\intertext{By the integral form of the Gronwall Lemma, it follows that
for all $t \in [0,T]$:}
\big \| w(t) - w_\delta(t) \big \|  &\leq \bigl[\| w(0) - w_\delta(0)\|  + 2 \delta_o T\bigr] \exp(Lt).
\end{align*}
Thus,
\begin{equation*}
 \begin{aligned}
 \label{eq:stabilitysupGen}
\sup_{t \in [0,T]} \big \| w(t) - w_\delta(t) \big \| \leq \bigl[ \| w(0) - w_\delta(0) \| + 2 \delta_o T \bigr] \exp (LT).
 \end{aligned}
\end{equation*}
By choice of initial conditions and $\delta_o$ sufficiently small we can
achieve a $\epsapprox>0$ approximation.
Finally, we note that the approximate system \eqref{eq:fullapproxGen}
is a function of parameter $\theta_\delta = [\theta_m, \theta_g] \in \R^{N_\delta}$ with $n_\delta = N_g + N_m.$
\end{proof}

\subsection{Proof of Continuous-Time RNN Approximation Theorem (Linear in Observation)}
\label{sec:rnnproof}

\begin{proof}
Recall equations \eqref{eq:fgmdag2}.
By approximation theory by means of
two-layer feed-forward neural networks \citep{cybenko_approximation_1989},
for any $\delta_o > 0$
there exist embedding dimensions $N_g$ and $N_h$ and parameterizations
\begin{equation*}
    \begin{aligned}
    \theta_g &= \{C_g \in \R^{d_y \times N_g},\ B_g\in \R^{N_g \times d_x},\ A_g\in \R^{N_g \times d_y},\ c_g \in \R^{N_g} \}, \\
		\theta_h &= \{C_h \in \R^{d_x \times N_h},\ B_h \in \R^{N_h \times d_x},\ A_h\in \R^{N_h \times d_y},\ c_h \in \R^{N_h} \}
  \end{aligned}
\end{equation*}
such that for any $\bigl(x,y\bigr) \in B(0, 2\rho_T)$, and in the maximum
norm,
\begin{equation*}
 \label{eq:ghapprox}
 \begin{aligned}
\| g^\dag(x,y) - C_g\sigma(B_gx + A_gy + c_g) \| &\leq \delta_o \\
\| h^\dag(x,y) - C_h\sigma(B_hx + A_hy + c_h) \| &\leq \delta_o.
\end{aligned}
\end{equation*}
Without loss of generality we may assume that $C_g$ and $C_h$ have
full rank since, if they do not, arbitrarily small changes can be
made which restore full rank.
By using these parameterizations and embedding dimensions,
we can rewrite \eqref{eq:fgmdag2} as
\begin{equation}
 \begin{aligned}
 \label{eq:truewitherror}
 \dot{x} &= f_0(x) + m \\
 \dot{y} &= C_g\sigma(B_gx + A_gy + c_g) + e_y(t) \\
 \dot{m} &= C_h\sigma(B_hx + A_hy + c_h) + e_{m^\dag}(t)
 \end{aligned}
 \end{equation}
where, uniformly for $\bigl(x(0),y(0)\bigr) \in B(0, \rho_0),$
\begin{equation*}
 \begin{aligned}
\sup_{t \in [0,T]} \| e_y(t) \| \leq \delta_o \\
\sup_{t \in [0,T]} \| e_{m^\dag}(t) \| \leq \delta_o.
\end{aligned}
\end{equation*}
By removing the bounded error terms, we obtain the approximate system:
\begin{equation}
 \begin{aligned}
 \label{eq:fullapprox}
 \dot{x}_\delta &= f_0(x_\delta) + m_\delta \\
 \dot{y}_\delta &= C_g\sigma(B_gx_\delta + A_gy_\delta + c_g) \\
 \dot{m}_\delta &=  C_h\sigma(B_hx_\delta + A_hy_\delta + c_h)
 \end{aligned}
 \end{equation}
Here $m_\delta(t)$ is initialized at $m^\dag(x(0),y(0)).$
Next, we obtain a stability bound on the discrepancy between the approximate system \eqref{eq:fullapprox} and the true system (originally written as \eqref{eq:fgmdag} and re-formulated as \eqref{eq:truewitherror}).
First, let $w = (x, y, m)$, $w_\delta = (x_\delta, y_\delta, m_\delta)$ and define $F$ to be the concatenated right-hand-side of \eqref{eq:fullapprox}. Note that
$F$ is $L-$Lipschitz in the maximum norm, for some $L$ related to the Lipschitz continuity of $f_0$, approximation parameterization $\theta_\delta$, and regularity of nonlinear activation function $\sigma$.
Then we can write the true and approximate systems, respectively, as
\begin{subequations}
\label{eq:dda}
 \begin{align}
\dot{w} &= F(w)+e_w(t)\\
\dot{w}_\delta &= F(w_\delta),
 \end{align}
 \end{subequations}

where
$$\sup_{t \in [0,T]} \| e_w(t) \| \le \sup_{t \in [0,T]} \| e_y(t) \|
+\sup_{t \in [0,T]}\| e_{m^\dag}(t) \| \le 2\delta_o.$$

Let $Pw=(x,y)$ and $P^\perp w=m;$ recall that $P^\perp w(0)$
is defined in terms of $Pw(0).$
Then, for any $t \in [0,T]$, and for all $Pw(0), Pw_{\delta}(0) \in
B(0,\rho_0)$

\begin{align*}
\| w(t) - w_\delta(t)  \|
&\leq \big \| w(0) - w_\delta(0)\big \|  + \int_0^t \big \| e_w(s)\big \| ds + \int_0^t \big \| F\big(w(s)\big) - F(w_\delta(s)\big)\big \| ds.
\end{align*}
By following the logic in \Cref{sec:rnnproofGen}, we have
%
\begin{equation*}
 \begin{aligned}
 \label{eq:stabilitysup}
\sup_{t \in [0,T]} \big \| w(t) - w_\delta(t) \big \| \leq \bigl[ \| w(0) - w_\delta(0) \| + 2 \delta_o T \bigr] \exp (LT).
 \end{aligned}
\end{equation*}
By choice of initial conditions and $\delta_o$ sufficiently small we can
achieve a $\epsapprox>0$ approximation.

Finally, we note that the approximate system \eqref{eq:fullapprox}
may be written as a recurrent neural network of form \eqref{eq:hybridRNN_cont}
as follows. Consider the equations
\begin{equation}
 \begin{aligned}
 \label{eq:fullapproxsub1}
 \dot{x}_\delta &= f_0(x_\delta) + C_hn_\delta \\
 \dot{z}_\delta &= \sigma(B_gx_\delta + A_gC_gz_\delta + c_g) \\
 \dot{n}_\delta &=  \sigma(B_hx_\delta + A_hC_gz_\delta + c_h)
 \end{aligned}
 \end{equation}
where we have defined $(z_\delta,n_\delta)$ in terms of
$(y_\delta,m_\delta)$ by $y_\delta = C_g z_\delta$
and $m_\delta = C_h n_\delta$.
Now note that \eqref{eq:fullapproxsub1} is equivalent to
\eqref{eq:hybridRNN_cont}, with recurrent state $r_{\delta}$ and
parameters $\theta_\delta$ given by:

\begin{itemize}
  \item $r_\delta = \begin{bmatrix} z_\delta \\ n_\delta \end{bmatrix}$
  \item $C_\delta = \begin{bmatrix} 0 & C_h \end{bmatrix}$
  \item $B_\delta = \begin{bmatrix} B_g \\ B_h \end{bmatrix}$
  \item $A_\delta = \begin{bmatrix} A_g C_g & 0 \\ A_h C_g & 0 \end{bmatrix}$
  \item $c_\delta = \begin{bmatrix} c_g \\ c_h \end{bmatrix}$
\end{itemize}
Any initial condition on $(y_\delta(0),m_\delta(0))$ may be achieved
by choice of initializations for $(z_\delta(0),n_\delta(0))$, since
$C_g,C_h$ are of full rank.

\end{proof}

\subsection{Random Feature Approximation}
\label{sec:rfAppendix}

Random feature methods lead to function approximation for mappings between
Hilbert spaces $X \to Y$. They operate by constructing a probability
space $(\Theta, \nu, \mathcal{F})$ with $\Theta \subseteq \R^p$ and
feature map $\varphi \colon X \times \Theta \to Y$ such that
$k(x,x') := \Expect^{\vartheta} [ \varphi(x ; \ \vartheta) \otimes \varphi(x' ; \ \vartheta)] \in \cL(Y,Y)$
forms a reproducing kernel in an associated reproducing kernel Hilbert space (RKHS) $K$.
Solutions are sought within
$\textrm{span}\{\varphi( \cdot \ ; \ \vartheta_l)\}_{l=1}^m$
where the $\{\vartheta_l\}$ are picked i.i.d. at random.
Theory supporting the approach was established in finite dimensions
by \citet{rahimi_random_2008}; the method  was recently applied
in the infinite dimensional setting in \citep{nelsen_random_2020}.

We now explain the precise random features setting adopted in
\Cref{sec:imp}, and hypothesis classes given by \eqref{eq:hc} and
\eqref{eq:hc2}.
We start with random feature functions $\varphi(\cdot \ ; \ \vartheta) \colon \R^{d_x} \to \R$,
with $\vartheta = [w, b]$,
\begin{equation}
\label{eq:rfd}
  \begin{aligned}
    w &\in \R^{d_x} \sim \mathcal{U}(-\omega, \omega)\\
    b &\in \R \sim \mathcal{U}(-\beta, \beta) \\
    \varphi(x; \ w, b) &:= \tanh(w^Tx + b),
\end{aligned}
\end{equation}

and $\omega, \beta > 0$.
We choose $D$ i.i.d. draws of $w, b$, and stack the resulting random feature functions to form the map
$\phi(x) \colon \R^{d_x} \to \R^{D}$ given by

$$\phi(x) := \begin{bmatrix} \varphi(x; \ w_1, b_w) & \dots & \varphi(x; \ w_D, b_D) \end{bmatrix}^T.$$
We define hypothesis class \eqref{eq:hc} by introducing
matrix $C \colon \R^D \to \R^{d_x}$
and seeking approximation to model error in the form
$m(x)= C \phi(x)$
by optimizing a least squares function over matrix $C$.
This does not quite correspond to the random features model with
$X=Y=\R^{d_x}$ because, when written as a linear span of vector
fields mapping $\R^{d_x}$ into itself, the vector fields are
not independent. Nonetheless we found this approach convenient
in practice and employ it in our numerics.

To align with the random features model with $X=Y=\R^{d_x}$,
we choose $D=d_x$ and draw $p$ functions $\phi(\cdot)$,
labelled as $\{f_\ell(\cdot)\}$
i.i.d. at random from the preceding construction, leading to
hypothesis class \eqref{eq:hc2}: we then seek
approximation to model error in the form
$m(x)=\sum_{\ell=1}^p \theta_\ell f_\ell(x).$ We
find this form of random features model most convenient
to explain the learning theory perspective on model error.

\subsection{Derivation of Tikhonov-Regularized Linear Inverse Problem}
\label{sec:ipderivation}
Here, we show that optimization of \eqref{eq:J_rfrc}

  $$\mathcal{J}_T(C) = \frac{1}{2T} \int_0^T \left\|\dot{x}(t) - f_0(x) - C\phi \big(x(t)\big) \right\|^2dt + \frac{\lambda}{2} \|C\|^2$$

reduces to a Tikhonov-regularized linear inverse problem.
Since \eqref{eq:Jd_rfrc} is quadratic in $C$, there exists a unique global minimizer for $C^\ast$ such that $\frac{\partial \mathcal{J}_T}{\partial C}(C^\ast) = 0$.
The minimizer $C^\ast$ satisfies:

\begin{equation*}
  \label{eq:lipd2}
      (Z + \lambda I) C^T = Y
\end{equation*}

where
\begin{equation*}
  \label{eq:ZYagain}
  \begin{aligned}
  Z &= \overline{[\phi \otimes \phi]}_T \\
  Y &= \overline{[\phi \otimes (\dot{x} - f_0)]}_T.
  \end{aligned}
\end{equation*}

and
$$[A \otimes B]_t := A(t) B^T(t)$$
$$\overline{A_T} := \frac{1}{T} \int_0^T A(t) dt$$
for $A(t) \in \R^{m \times n}, \ B(t) \in \R^{m \times l}$.

To see this, observe that

\begin{equation*}
  \begin{aligned}
    \mathcal{J}_T(C) &= \frac{1}{2T} \int_0^T \| \dot{x}(t) - f_0\big(x(t)\big) - C\phi(x(t))\|^2 dt + \frac{\lambda}{2} \|C\|^2 \\
                     &= \frac{1}{2T} \int_0^T \| \dot{x}(t) - f_0\big(x(t)\big)\|^2,\\
                      &\quad\quad\quad+ \big \langle C\phi \big(x(t)\big), C\phi \big(x(t)\big) \big \rangle
                      -2 \big \langle \dot{x}(t) - f_0\big(x(t)\big) , C \phi \big(x(t)\big) \big \rangle
                      dt
                       + \frac{\lambda}{2} \langle C , C \rangle
  \end{aligned}
\end{equation*}

and

\begin{equation*}
  \begin{aligned}
    \frac{\partial \mathcal{J}_T(C)}{\partial C}
                     &= \frac{1}{2T} \int_0^T  2 C \big[\phi \big(x(t)\big) \otimes \phi\big(x(t)\big) \big]
                      -2 \big[(\dot{x}(t) - f_0\big(x(t)\big)) \otimes \phi\big(x(t)\big)\big]
                      dt
                       + \lambda C.
  \end{aligned}
\end{equation*}

By setting the gradient to zero, we see that

\begin{equation*}
C \bigg[\frac{1}{T} \int_0^T \big[\phi\big(x(t)\big) \otimes \phi\big(x(t)\big)\big]dt + \lambda I\bigg] = \frac{1}{T} \int_0^T \big[\big(\dot{x}(t) - f_0\big(x(t)\big)\big) \otimes \phi\big(x(t)\big) \big]dt.
\end{equation*}

Finally, we can take the transpose of both sides, apply our definitions of $Y, Z$, and use symmetry of $Z$ to get

\begin{equation*}
  [Z + \lambda I] C^T = Y.
\end{equation*}

\bibliography{L96_paper}

\end{document}